\documentclass[11pt,reqno,tbtags,a4paper]{amsart}
\usepackage{amssymb}
\usepackage{mathabx}  
\usepackage[utf8]{inputenc}   
\usepackage{xpunctuate}
\usepackage{url}
\usepackage[square,numbers]{natbib}
\bibpunct[, ]{[}{]}{;}{n}{,}{,}

\DeclareSymbolFont{cmsymbols}{OMS}{cmsy}{m}{n}  
\SetSymbolFont{cmsymbols}{bold}{OMS}{cmsy}{b}{n}
\DeclareMathSymbol{\succ}{\mathrel}{cmsymbols}{"1F}
\DeclareMathSymbol{\prec}{\mathrel}{cmsymbols}{"1E}
\DeclareMathSymbol{\succeq}{\mathrel}{cmsymbols}{"17}
\DeclareMathSymbol{\preceq}{\mathrel}{cmsymbols}{"16}

\title
{Almost sure and moment convergence for triangular P\'olya urns}

\date{21 March, 2024}

\author{Svante Janson}
\thanks{Supported by the Knut and Alice Wallenberg Foundation
and
the Swedish~Research~Council~(VR)
}
\address{Department of Mathematics, Uppsala University, PO Box 480,
SE-751~06 Uppsala, Sweden}
\email{svante.janson@math.uu.se}
\newcommand\urladdrx[1]{{\urladdr{\def~{{\tiny$\sim$}}#1}}}
\urladdrx{http://www.math.uu.se/~svante/}


\numberwithin{equation}{section}

\renewcommand\le{\leqslant}
\renewcommand\ge{\geqslant}

\allowdisplaybreaks

\theoremstyle{plain}
\newtheorem{theorem}{Theorem}[section]
\newtheorem{lemma}[theorem]{Lemma}
\newtheorem{proposition}[theorem]{Proposition}
\newtheorem{corollary}[theorem]{Corollary}

\theoremstyle{definition}

\newcommand\xqed[1]{%
    \leavevmode\unskip\penalty9999 \hbox{}\nobreak\hfill
    \quad\hbox{#1}}

\newtheorem{exampleqqq}[theorem]{Example}
\newenvironment{example}{\begin{exampleqqq}}
  {\xqed{$\triangle$}\end{exampleqqq}}

\newtheorem{remarkqqq}[theorem]{Remark}
\newenvironment{remark}{\begin{remarkqqq}}
  {\xqed{$\triangle$}\end{remarkqqq}}

\newtheorem{definition}[theorem]{Definition}
\newtheorem{problem}[theorem]{Problem}

\newtheorem*{ack}{Acknowledgement}

\theoremstyle{remark}

\newcounter{dummy}
\makeatletter
\newcommand\myitem[1][]{\item[#1]\refstepcounter{dummy}\def\@currentlabel{#1}}
\makeatother

\newenvironment{romenumerate}[1][-10pt]{
\addtolength{\leftmargini}{#1}\begin{enumerate}
 \renewcommand{\labelenumi}{\textup{(\roman{enumi})}}%
 \renewcommand{\theenumi}{\labelenumi}%
 }{\end{enumerate}}

\newenvironment{PXenumerate}[1][]{
\addtolength{\leftmargini}{-10pt}%
\begin{enumerate}
 \renewcommand{\labelenumi}{\textup{(#1\arabic{enumi})}}%
 \renewcommand{\theenumi}{\labelenumi}%
 }{\end{enumerate}}

\newcounter{Aenumi}
\newenvironment{Aenumerate}{
\begin{enumerate}
\setcounter{enumi}{\value{Aenumi}}     
 \renewcommand{\theenumi}{\textup{(A\arabic{enumi})}}%
 \renewcommand{\labelenumi}{{\bf\theenumi}}%
 }{\end{enumerate}\setcounter{Aenumi}{\value{enumi}}}

\newenvironment{PQenumerate}[1]{
\begin{enumerate}
 \renewcommand{\theenumi}{\textup{(#1)}}%
 \renewcommand{\labelenumi}{\bf\theenumi}%
 }{\end{enumerate}}

\newenvironment{alphenumerate}[1][-10pt]{
\addtolength{\leftmargini}{#1}\begin{enumerate}
 \renewcommand{\labelenumi}{\textup{(\alph{enumi})}}%
 \renewcommand{\theenumi}{\textup{(\alph{enumi})}}%
 }{\end{enumerate}}

\newcounter{thmenumerate}
\newenvironment{thmenumerate}
{\setcounter{thmenumerate}{0}%
 \def\item{\par
 \refstepcounter{thmenumerate}\textup{(\roman{thmenumerate})\enspace}}
}
{}

\newcounter{xenumerate}   

\newcommand\pfitemx[1]{\par#1:}
\newcommand\pfitemref[1]{\pfitemx{\ref{#1}}}

\newcounter{Case}
\newcounter{casex}
\newcounter{case}
\newcommand\resetCase{\setcounter{Case}{0}}
\newcommand\resetcase{\setcounter{case}{0}}
\newcommand\resetcasex{\setcounter{casex}{0}}
\makeatletter
\newcommand\pfCase[1]{\medskip\noindent\refstepcounter{Case}%
\def\@currentlabel{\arabic{Case}}
\emph{Case \arabic{Case}: #1}. \noindent}
\newcommand\pfCaseY[2]{\medskip\noindent\emph{Case #1: #2.} \noindent}
\newcommand\pfcase[1]{\medskip\noindent\refstepcounter{case}%
\def\@currentlabel{(\roman{case})}
  (\roman{case}) \emph{#1}. \noindent}
\newcommand\pfcasex[1]{\medskip\noindent\refstepcounter{casex}%
\def\@currentlabel{(\roman{casex}${}'$)}
  (\roman{casex}${}'$) \emph{#1}. \noindent}
\newcommand\pfcasexx[1]{\medskip\noindent\refstepcounter{casex}%
\def\@currentlabel{(\roman{case})(\alph{casex})}
  \@currentlabel{} \emph{#1}. \noindent}

\newcounter{cases}
\newcommand\xcase[1]{\medskip\noindent\refstepcounter{cases}%
  \emph{Case \arabic{cases}}, \emph{#1}: \noindent}

\newcounter{steps}

\newcommand\stepx[1]{\medskip\noindent\refstepcounter{steps}%
 \emph{Step \arabic{steps}: #1}.\noindent}
\newcommand\stepxx{\medskip\noindent\refstepcounter{steps}%
 \emph{Step \arabic{steps}}. \noindent}
\newcommand\stepparti[1]{\stepx{#1, first part}}
\newcommand\resetsteps{\setcounter{steps}{0}}
\newcommand\resetstepx{\setcounter{steps}{0}}

\newcounter{claims}
\newcommand\claim[1]{\refstepcounter{claims}\medskip\noindent
\def\@currentlabel{(\roman{claims})}%
\@currentlabel{} 
\emph{#1}}

\makeatother

\newcommand{\refT}[1]{Theorem~\ref{#1}}
\newcommand{\refTs}[1]{Theorems~\ref{#1}}

\newcommand{\refL}[1]{Lemma~\ref{#1}}
\newcommand{\refLs}[1]{Lemmas~\ref{#1}}
\newcommand{\refR}[1]{Remark~\ref{#1}}

\newcommand{\refS}[1]{Section~\ref{#1}}
\newcommand{\refSs}[1]{Sections~\ref{#1}}
\newcommand{\refSS}[1]{Section~\ref{#1}}
\newcommand{\refStep}[1]{Step~\ref{#1}}
\newcommand{\refSteps}[1]{Steps \ref{#1}}
\newcommand{\refP}[1]{Problem~\ref{#1}}
\newcommand{\refD}[1]{Definition~\ref{#1}}
\newcommand{\refE}[1]{Example~\ref{#1}}
\newcommand{\refEs}[1]{Examples~\ref{#1}}

\newcommand{\refApp}[1]{Appendix~\ref{#1}}

\newcommand{\refCase}[1]{Case~\ref{#1}}

\begingroup
  \divide\count255 by 60
  \multiply\count255 by -60
  \advance\count255 by \time
  \ifnum \count255 < 10 \xdef\klockan{\the\count1.0\the\count255}
  \else\xdef\klockan{\the\count1.\the\count255}\fi
\endgroup

\newcommand{\sumno}{\sum_{n=0}^\infty}

\newcommand{\sumk}{\sum_{k=1}^\infty}

\newcommand{\sumn}{\sum_{n=1}^\infty}

\newcommand{\sumiq}{\sum_{i=1}^q}
\newcommand{\sumjq}{\sum_{j=1}^q}

\newcommand\set[1]{\ensuremath{\{#1\}}}
\newcommand\bigset[1]{\ensuremath{\bigl\{#1\bigr\}}}

\newcommand\xpar[1]{(#1)}
\newcommand\bigpar[1]{\bigl(#1\bigr)}
\newcommand\Bigpar[1]{\Bigl(#1\Bigr)}

\newcommand\lrpar[1]{\left(#1\right)}
\newcommand\bigsqpar[1]{\bigl[#1\bigr]}
\newcommand\sqpar[1]{[#1]}

\newcommand\cpar[1]{\{#1\}}
\newcommand\bigcpar[1]{\bigl\{#1\bigr\}}

\newcommand\abs[1]{\lvert#1\rvert}
\newcommand\bigabs[1]{\bigl\lvert#1\bigr\rvert}
\newcommand\Bigabs[1]{\Bigl\lvert#1\Bigr\rvert}

\newcommand\lrabs[1]{\left\lvert#1\right\rvert}
\def\rompar(#1){\textup(#1\textup)}    
\newcommand\xfrac[2]{#1/#2}

\newcommand\xqfrac[2]{#1/(#2)}

\newcommand\parfrac[2]{\lrpar{\frac{#1}{#2}}}

\def\xexp(#1){e^{#1}}
\newcommand\ceil[1]{\lceil#1\rceil}

\newcommand\ntoo{\ensuremath{{n\to\infty}}}

\newcommand\mtoo{\ensuremath{{m\to\infty}}}

\newcommand\ttoo{\ensuremath{{t\to\infty}}}

\newcommand\bmin{\land}
\newcommand\bmax{\lor}
\newcommand\norm[1]{\lVert#1\rVert}
\newcommand\bignorm[1]{\bigl\lVert#1\bigr\rVert}
\newcommand\Bignorm[1]{\Bigl\lVert#1\Bigr\rVert}

\newcommand\upto{\nearrow}
\newcommand\punkt{\xperiod}    
\newcommand\iid{i.i.d\punkt}    
\newcommand\ie{i.e\punkt}
\newcommand\eg{e.g\punkt}

\newcommand\cf{cf\punkt}
\newcommand{\as}{a.s\punkt}
\newcommand{\aex}{a.e\punkt}

\newcommand{\tend}{\longrightarrow}
\newcommand\dto{\overset{\mathrm{d}}{\tend}}
\newcommand\pto{\overset{\mathrm{p}}{\tend}}
\newcommand\asto{\overset{\mathrm{a.s.}}{\tend}}

\newcommand\eqd{\overset{\mathrm{d}}{=}}

\newcommand\bbR{\mathbb R}

\newcommand\bbZ{\mathbb Z}

\newcommand\bbZgeo{\mathbb Z_{\ge0}}

\newcounter{CC}
\newcounter{cc}

\newcommand\E{\operatorname{\mathbb E}{}} 
\renewcommand\P{\operatorname{\mathbb P{}}}

\newcommand\Var{\operatorname{Var}}

\newcommand\Exp{\operatorname{Exp}}
\newcommand\Po{\operatorname{Po}}

\newcommand\Be{\operatorname{Be}}

\newcommand\Gei{\operatorname{Ge}}
\newcommand\NBi{\operatorname{NegBin}}

\newcommand\ga{\alpha}
\newcommand\gb{\beta}
\newcommand\gd{\delta}
\newcommand\gD{\Delta}
\newcommand\gf{\varphi}
\newcommand\gam{\gamma}
\newcommand\gG{\Gamma}

\newcommand\kk{\kappa}
\newcommand\gl{\lambda}

\newcommand\gs{\sigma}

\newcommand\eps{\varepsilon}
\renewcommand\phi{\xxx}  

\newcommand\cA{\mathcal A}
\newcommand\cB{\mathcal B}
\newcommand\cC{\mathcal C}
\newcommand\cE{\mathcal E}
\newcommand\cF{\mathcal F}

\newcommand\cL{{\mathcal L}}

\newcommand\cN{\mathcal N}

\newcommand\cT{{\mathcal T}}
\newcommand\cU{{\mathcal U}}
\newcommand\cUp{{\cU}^+}
\newcommand\cV{\mathcal V}
\newcommand\cW{\mathcal W}
\newcommand\cX{{\mathcal X}}
\newcommand\cY{{\mathcal Y}}
\newcommand\cZ{{\mathcal Z}}

\newcommand\tB{\widetilde B}

\newcommand\tN{\widetilde N}

\newcommand\tX{{\widetilde X}}
\newcommand\tY{{\widetilde Y}}
\newcommand\tZ{{\widetilde Z}}

\newcommand\indic[1]{\boldsymbol1_{\cpar{#1}}}

\newcommand\etta{\boldsymbol1} 

\newcommand\smatrixx[1]{\left(\begin{smallmatrix}#1\end{smallmatrix}\right)}
\newcommand\matrixx[1]{\begin{pmatrix}#1\end{pmatrix}}

\newcommand\qw{^{-1}}
\newcommand\qww{^{-2}}
\newcommand\qq{^{1/2}}

\newcommand\intoi{\int_0^1}
\newcommand\intot{\int_0^t}
\newcommand\intox{\int_0^x}
\newcommand\intoo{\int_0^\infty}

\newcommand\oi{\ensuremath{[0,1]}}

\newcommand\ooo{[0,\infty)}
\newcommand\ooox{[0,\infty]}

\newcommand\dd{\,\mathrm{d}}
\newcommand\ddx{\mathrm{d}}

\newcommand{\pgf}{probability generating function}
\newcommand{\mgf}{moment generating function}

\newcommand{\gsf}{$\gs$-field}
\newcommand{\ui}{uniformly integrable}

\newcommand\lhs{left-hand side}
\newcommand\rhs{right-hand side}

\newcommand\GW{Galton--Watson}

\newcommand\GWp{\GW{} process}

\newcommand\xoo{_1^\infty}

\newcommand\xoon{_{n=1}^\infty}

\newcommand\nn{^{(n)}}

\newcommand\too{_{t\ge0}}
\newcommand\glx{\gl^*}
\newcommand\xglx{{\bar\gl}^*}
\newcommand\chgl{\widecheck\gl}
\newcommand\chglx{\widecheck\gl^*}
\newcommand\chxi{\widecheck\xi}
\newcommand\chX{\widecheck{X}}
\newcommand\chcX{\widecheck{\cX}}
\newcommand\tXX{\widetilde{X}^{**}}
\newcommand\tVV{\widetilde{V}^{**}}
\newcommand\tZZ{\widetilde{Z}^{**}}
\newcommand\tMM{\widetilde{M}^{**}}
\newcommand\xtXX{\widetilde{\overline X}^{**}}
\newcommand\tZc{\tZ^\dag}
\newcommand\tZcc{\tZ^\ddag}
\newcommand\tZcl{\tZ^\dag_\ell}
\newcommand\tZccl{\tZ^\ddag_\ell}
\newcommand\mmi{m}
\newcommand\kkk{^{(k)}}
\newcommand\kkkx{^{(k)*}}
\newcommand\WXi{\cX_i}
\newcommand\kko{\kappa_0}
\newcommand\hgl{\widehat\gl}
\newcommand\hkk{\widehat\kappa}
\newcommand\hkko{\widehat\kappa_0}
\newcommand\hcX{\widehat{\cX}}
\newcommand\hcY{\widehat{\cY}}
\newcommand\hcZ{\widehat{\cZ}}
\newcommand\hcN{\widehat{\cN}}
\newcommand\refAA{\ref{A0}--\ref{A+}}
\newcommand\refAAX{\ref{A0}--\ref{A2}}
\newcommand\refAAXZ{\ref{A0}--\ref{A2} and \ref{A-5}}

\newcommand\refAAZ{\ref{A0}--\ref{A2}, \ref{A-5}, and
  \ref{A-6}--\ref{A-7}}
\newcommand\refAAZZ{\ref{A0}--\ref{A2}, \ref{A-5}, and
  \ref{A-7}--\ref{A-8}}
\newcommand\refAAZC{\ref{A0}--\ref{A2}, \ref{A-5}, and
  \ref{A-7}}

\newcommand\TT{\widehat T}
\newcommand\mmZ{Z^{[\le m]}}
\newcommand\hB{\widehat B}
\newcommand\hc{\widehat c}
\newcommand\setq{\ensuremath{\set{1,\dots,q}}}
\newcommand\setr{\ensuremath{\set{1,\dots,r}}}
\newcommand\nuj{j}
\newcommand\abscont{absolutely continuous}
\newcommand\bC{\overline C}
\newcommand\sA{\mathsf A}
\newcommand\sD{\mathsf D}
\newcommand\sE{\mathsf E}

\newcommand\sL{\mathsf L}
\newcommand\sP{\mathsf P}
\newcommand\sQ{\mathsf Q}
\newcommand\sQx{\mathsf Q_*}
\newcommand\sQxm{\mathsf Q_{*-}}
\newcommand\sQm{\mathsf Q^-}
\newcommand\sQp{\mathsf Q^+}
\newcommand\sQmin{\mathsf Q_{\min}}
\newcommand\sDnk{\sD_\nu^\kk}
\newcommand\sLk{\sL_\kk}

\newcommand\fX{\mathfrak X}
\newcommand\fY{\mathfrak Y}
\newcommand\fZ{\mathfrak Z}

\newcommand\gda{{\gd/\ga}}
\newcommand\BESQ{\mathrm{BESQ}}
\newcommand\bx{\mathbf x}
\newcommand\bX{\mathbf X}
\newcommand\bY{\mathbf Y}
\newcommand\bcX{\boldsymbol{\cX}}
\newcommand\hbcX{\widehat{\bcX}}
\newcommand\ba{\mathbf a}
\newcommand\bxi{\boldsymbol\xi}
\newcommand\even{_{\textrm{even}}}
\newcommand\odd{_{\textrm{odd}}}
\newcommand\ellr{r}
\newcommand\ctime{con\-tin\-u\-ous-time} 
\newcommand\dtime{dis\-crete-time}
\newcommand\Zk{Z^{(k)}}

\newcommand\Zx[1]{Z^{(#1)}}
\newcommand\cZx[1]{\cZ^{(#1)}}
\newcommand\gbal{b}
\newcommand\iii{_{ii}}
\newcommand\kkx{\kk^*}
\newcommand\xxi{\bar{\xi}}
\newcommand\txi{\tilde{\xi}}
\newcommand\xX{\overline{X}}

\newcommand{\Polya}{P\'olya}

\newcommand\cadlag{c\`adl\`ag}

\begin{document}

\begin{abstract} 
We consider triangular Pólya urns and show under very weak conditions
a general strong limit theorem of the form $X_{ni}/a_{ni}\asto \mathcal{X}_i$,
where $X_{ni}$ is the number of balls of colour $i$ after $n$ draws;
the constants $a_{ni}$ are explicit and of the form $n^\ga\log^\gam n$;
the limit is a.s.\ positive, and may be either deterministic or random,
but is in general unknown.

The result extends to urns with subtractions under weak conditions, but a
counterexample shows that some conditions are needed.

For balanced urns we also prove moment convergence in the main results 
if the replacements have the corresponding moments.

The proofs are based on studying the corresponding continuous-time urn using
martingale methods, and showing corresponding results there.
In the main part of the paper, we assume for convenience that all
replacements have finite second moments; in an appendix this is relaxed to
$L^p$ for some $p>1$. 
\end{abstract}

\maketitle

\tableofcontents

\section{Introduction}\label{S:intro}

A (generalized) \Polya{} urn contains balls of different colours.
A ball is drawn at random from the urn, and is replaced by a set of balls
that depends on the colour of the drawn balls; more generally, the replacement
set may be random, with a distribution depending on the drawn colour.
We assume that the set $\sQ$ of colours is finite, and let $q:=|\sQ|$ be the
number of colours.
(See \eg{} \cite{JohnsonKotz}, \cite{Mahmoud}, 
\cite{SJ154} for the history and further references.)
It is often convenient to
assume that $\sQ=\set{1,\dots,q}$, 
but for us it will  be convenient not to assume this.

The \Polya{} urn process can be defined formally as follows.
The composition of the urn at time $n$ is given by the vector
$\bX_n=(X_{ni})_{i\in\sQ}\in \ooo^q$,
where $X_{ni} $ is the number of balls of colour $i$. 
The urn starts with a given vector $\bX_0$, 
and evolves according to a discrete-time Markov process. 
Each colour $i$ has an \emph{activity} $a_i\ge0$,
and a (generally random) 
\emph{replacement vector} $\bxi_i=(\xi_{ij})_{j\in\sQ}$.
At each time $ n+1\ge 1 $, 
the urn is updated by drawing
one ball at random from the urn, with the probability of any ball
proportional to its activity. 
(In many cases $a_i=1$ for all $i$, 
so all balls are drawn with equal probability;
the reader may concentrate on this case
until \refS{SpfT1}.)
Thus, the drawn ball has colour $i$  
with probability 
\begin{equation}
\label{urn}
 \frac{a_iX_{ni}}{\sum_{j}a_jX_{nj}}. 
\end{equation}
If the drawn ball has colour $i$, it is replaced
together with $\xi\nn_{ij}$ balls of colour $j$, $j\in\sQ$,
where the random vector 
$ \bxi\nn_{i}=(\xi\nn_{ij})_{j\in\sQ} $
is a copy of $ \bxi_i$ that is
independent of everything else that has happened so far.  
Thus, the urn is updated to 
\begin{align}\label{xin}
\bX_{n+1}=\bX_n+\bxi\nn_i.  
\end{align}

\begin{remark}
  Note that, as in many other papers on \Polya{} urns, we do not assume that
  $X_{ni}$ are integers; any real numbers $X_{ni}\ge0$ are allowed.
In general, it is thus a misnomer to call $X_{ni}$ the ``number'' of balls
of colour 
$i$; it is more precise to regard $X_{ni}$ as the amount of colour $i$ in
the urn. Nevertheless, we will continue to use the traditional terminology,
which thus has to be interpreted liberally by the reader.
\end{remark}

\begin{remark}
  Since the drawn ball is replaced, it is also really a misnomer to call
  $\bxi_i$ the ``replacement'' vector; it is really an \emph{addition} vector.
(The  replacements are really $\xi_{ij}+\gd_{ij}$.)
Nevertheless, we use the terminology above,  which is used in many papers.
\end{remark}

\begin{remark}
We allow the replacement vectors $\bxi_i$ to be random. 
(Some papers consider only the special case of deterministic $\bxi_i$,
which is an important special case that appears in many applications.)
We may say that the urn has 
\emph{deterministic (or non-random) replacements} if all $\xi_{ij}$
are deterministic, and otherwise 
\emph{random replacements}.
These terms should thus be interpreted as conditioned on (the colour of) the
drawn ball. 
\end{remark}

\begin{remark}\label{Rmatrix}
  It is often convenient to describe the replacements by the
\emph{replacement matrix}
$(\xi_{ij})_{i,j\in\sQ}$.
Note, however, that unless the replacements are deterministic, this may be
somewhat misleading, since the rows should be regarded as separate random
vectors, not necessarily defined on the same probability space; there is
 no need for a joint distribution of different rows.
\end{remark}

\begin{remark}\label{Rsub}
  In the first part of the paper we assume that the replacements $\xi_{ij}\ge0$ 
(Condition \ref{A+} below), meaning that we only add balls to the urn and
never remove any. In \refS{S-}, and also usually in later sections, we more
generally  allow also that balls may be removed from the urn 
(assuming some hypotheses).
\end{remark}

\begin{remark}
  We allow some activities $a_i$ to be 0; this means that balls of colour
  $i$ never are drawn. 
See \refS{SpfT1} for an important example of this.
(If all $a_i>0$, we may reduce to the standard case
  $a_i=1$ by considering the urn $(a_iX_{ni})_i$, with corresponding replacement
  vectors $(a_j\xi_{ij})_j$, but we will not use this.)
\end{remark}

\begin{remark}
We assumed tacitly above that the denominator 
$\sum_j a_j X_{nj}$
in \eqref{urn}
is $>0$ for every $n\ge0$, so that the definition makes sense. 
This holds, for example, under assumptions \ref{A0} and \ref{A+} below.
(Urns that do not satisfy this, and therefore may stop at some finite time,
have also been studied, but they will not be considered here.)
\end{remark}

We are interested in asymptotic properties of $\bX_n$ as \ntoo.

In the present paper, we study \emph{triangular urns}, \ie, \Polya{} urns 
such that, for a suitable labelling of the colours by $1,\dots,q$, we have
$\xi_{ij}=0$ when $i>j$. (See also \refSS{SSQ}.)
This includes the original \Polya{} urns studied by \citet{Markov1917},
\citet{EggPol}, and \citet{Polya} (all for $q=2$), where the replacement
matrix is 
\emph{diagonal}: 
$\xi_{ij}=0$ when $i\neq j$, but also many other interesting cases.
See \refS{Sex} for some examples.

There are  many previous papers on triangular urns; we mention here
only a few that are particularly relevant to the present paper;
see also the references in the examples in \refS{Sex}.
\citet{Athreya1969} studied diagonal urns with random replacements
and showed \as{} convergence of the proportions of different colours,
using the embedding mthod of \citet{AthreyaKarlin} that is also the basis of
the present paper.
\citet{Gouet89, 
Gouet93} 
proved (in particular)
an \as{} convergence result 
for triangular urns with 2 colours and deterministic replacements,
assuming also that the urn is \emph{balanced}, 
meaning $\sum_j\xi_{ij}=\gbal$ for some constant $\gbal$
(and all $a_i=1$, see further \refS{Sbalanced}).
\citet{SJ169} studied triangular urns with 2 colours and deterministic
$\bxi_i$, and proved convergence in distribution (but not \as)
of the components $X_{ni}$
after suitable normalizations; there are several cases, and the limits are
sometimes normal and sometimes not.
This was partially extended 
by \citet{Aguech}, who also studied triangular urns with 2 colours, 
but allowed random replacements $\bxi_i$
(under some hypotheses, see \refE{E2X}); moreover, he proved
convergence \as, and not just in distribution.
\citet{BoseDM} 
(and \cite{BoseDM1} for $q=2$)
studied triangular urns with an arbitrary (finite) number of
colours;
they assumed that the replacements are deterministic, and that 
the urn is balanced, and then, under some further assumptions,
showed convergence \as{} of the components $X_{ni}$,
suitably normalized, see \refE{EBose}.

The main purpose of the present paper is to extend these results by
\citet{Gouet89,Gouet93}, 
\citet{Aguech}, and
\citet{BoseDM}, and show \as{} convergence for triangular urns with any
(finite) number of colours, allowing replacements $\bxi_i$ that are both
random and unbalanced.
Our main result is the following, using the technical assumptions \refAA{}
in \refSS{SSA} and the notation defined in \eqref{gli}--\eqref{gam} 
in \refSS{SSnot2} below. See also the extensions \refTs{T1-} 
and \ref{Tp} where the technical assumptions are weakened
(allowing urns with subtractions and reducing our moment assumption,
respectively); since the proof of the theorem is rather long and 
we want to focus on the main ideas, we use first
the conditions \refAA{} (which suffice for many applications), and 
add later the extra arguments needed for the extensions.

\begin{theorem}\label{T1}
Let $(X_{ni})_{i\in\sQ}$ be a triangular \Polya{} urn satisfying the
conditions \refAA{} below. 
Then, for every colour $i\in\sQ$, there exists a random variable $\hcX_i$
with $0<\hcX_i<\infty$ \as{} such that as \ntoo:
\begin{romenumerate}
  
\item \label{T1+}
If\/ $\hgl>0$, then
\begin{align}\label{t1b}
  \frac{X_{ni}}{ n^{\glx_i/\hgl}\log^{\gam_i}n}
\asto \hcX_i
.\end{align}

\item \label{T10}
If\/ $\hgl=0$, then
\begin{align}\label{t1c}
  \frac{X_{ni}}{n^{\kk_i/\hkko}}
\asto \hcX_i
.\end{align}
\end{romenumerate}
\end{theorem}
Note that the exponent $\gam_i$ may be both positive and negative;
see the examples in \refS{Sex}.
Note also that the various exponents in \eqref{t1b} and \eqref{t1c} are
explicitly given in \refSS{SSnot2}, but the limiting random variables $\hcX_i$
are known only in some special cases; in general they are unfortunately unknown.

The virtue of \refT{T1} is that it is very general,
but as discussed in the remaks below, 
more precise results are known in some special cases.

In \refT{T1}, the main case is
\ref{T1+},   $\hgl>0$, and the reader should focus on this case.
The case $\hgl=0$ is more special and of less interest for applications, but
it is included for completeness; by \eqref{hgl} and \eqref{gli}, this case
occurs when $\xi_{ii}=0$ for every $i\in\sQ$. (We might call such urns
\emph{strictly triangular}.) 

Another major result in the present paper (\refT{TMD})
says in particular that in the special case of balanced triangular urns,
if the replacements have finite  moments, then
the \as{} limits in \refT{T1} hold also in
$L^p$ for any $p<\infty$, and thus moments converge.
(It is an open problem whether this extends to some class of unbalanced
triangular urns.)

\begin{remark}\label{Rfluct}
\refT{T1}  describes the first-order asymptotics of the urn.
We will see in \refS{Sdeg} that
the limiting random variable $\hcX_i$ is deterministic (\ie, a constant) in
some cases, but not in general.
In cases where \refT{T1} yields a limit $\hcX_i$ that
is deterministic (and perhaps also otherwise), 
it is interesting to study fluctuations
(\ie, second order terms)  and try to find limits 
(\eg{} in distribution, after a suitable normalization)
for the difference of the two sides of \eqref{t1b} or \eqref{t1c}.
Such results in some cases are given in \cite{Gouet93},
\cite{SJ169} and \cite{Aguech}; 
see the examples in \refS{Sex}, but we will not pursue this problem here,
and leave it as an open problem.
\end{remark}

\begin{remark}\label{Rdistr}
Convergence almost surely implies convergence in distribution. Thus, as a
corollary, \eqref{t1b} holds also with convergence in distribution.
However, our proof does not seem to 
provide a method to find the limit distribution,
\ie{} the distribution of $\hcX_i$, except in some very simple cases.
Moreover, the limits $\hcX_i$ are (in general) dependent.

 For $q=2$ and deterministic $\xi_{ij}$, limit distributions were
given in \cite{SJ169}. (Sometimes degenerate, sometimes not.) The
results there thus describe the distribution of $\hcX_i$ in this case,
although the descriptions in some cases are complicated.
Some further examples of known limit distributions are given in some examples
in \refS{Sex}.
We leave the general case as another open problem.
\end{remark}

The proof of \refT{T1} is given in \refSs{S1}--\ref{SpfT1}
below, after some preliminaries in \refS{Snot}.
The proof is based on the embedding by \citet{AthreyaKarlin}
of a \Polya{} urn into a continuous-time multitype
branching process (\refSS{SSCT});
we then apply martingale methods to obtain a continuous-time version of
\refT{T1} (\refT{TC}); finally, this implies results for the embedded
discrete-time urn.
The proof is generalized to urns with subtraction in \refS{S-}, and to urns
with a weaker moment condition in \refApp{ALp}.
Since the proofs are rather long and technical, we prefer to first present
the proof in the basic case \refT{T1} (which is enough for most applications)
and later discuss the modifications required for the extensions,
instead of proving the most general results immediately.

\begin{remark}\label{Rmart}
\citet{Gouet89,Gouet93} and
\citet{BoseDM}, in  special cases (see \refE{E2} and \refE{EBose}), 
instead study $X_{ni}$ directly
and use martingale metods in discrete time. It seems that this approach
(also used by several authors for non-triangular urns) works well for
balanced urns, but that the embedding into continuous time works better for
unbalanced urns.
\end{remark}

\begin{remark}\label{Rnontri1}
We consider in this paper only triangular \Polya{} urns.
Another important class of urns consists of the irreducible urns.
In this case \as{} convergence (under some technical conditions)
was shown by
\citet{AthreyaKarlin}, see also
\cite[Section V.9.3]{AN} and
\cite[Theorem 3.21]{SJ154}.

It might be possible to combine the methods of the present paper and the
methods for irreducible urns to obtain  results on \as{} convergence for all 
types of \Polya{} urns (under some technical conditions), 
see \refR{Rnontri2},
but the present
paper is long as it is and we leave this as a speculation for future
research. (Note also the counterexamples \refE{E+-} and \ref{E--}, showing
that some conditions are needed even in the triangular case.)
\end{remark}

\subsection{Contents}
\refS{Snot} contains preliminaries, including some definitions and notation.
\refS{S1} consists of a series of lemmas that comprise the main technical
part of our proofs. They lead to the main theorem for continuous time in
\refS{STC}, which in turn is used to prove \refT{T1} in \refS{SpfT1}.
\refSs{Sdep} (continuous time) and \ref{Sdeg} (discrete time) contain results
on whether the limit random variable are degenerate (i.e., constant) or not,
and some related results.

\refS{S-} extend the previous results to urns with subtracion, where we
allow $\xi_{ii}=-1$. With some extra technical conditions, the previous
results hold in this case too, with only minor modifications of the proofs.

The following sections contain some complements.
\refS{Smean} is a short section comparing random and non-random replacements
with the same means.

\refS{Sbalanced} contains some general results on balanced urn,
mainly as preliminaries to the following sections.

\refS{Sdraw} considers the number of times a given colour is drawn;
it is shown that the results of earlier sections extend to this case.

\refS{Smoments} contains results on convergence in $L^2$ and in
$L^p$, and closely related results on convergence of moments
in, for example, \refT{T1}.
For the main result (\dtime), we have to assume that urn is balanced,
and we state an open problem for more general urns.

\refS{Smer} discusses briefly another open problem (rates of convergence).

\refS{Sex} contains a number of examples that illustrate the results and
their limitations, and also give connections to previous literature.

Finally, \refApp{AA} contains some simple general lemmas on absolute continuity
that we believe are known, but for which we were unable to find references.
\refApp{ALp} gives proofs of $L^p$ versions of $L^2$ estimates used in the
main part of the paper; this  yields both the extension of \refT{T1}
mentioned above, and a proof of the results in \refS{Smoments}.
\refApp{Ajb} gives a rather technical proof of one claim in \refE{E+-}.

\section{Some notation and other preliminaries}\label{Snot}

We use throughout the paper the notation $\sQ$, $q$,
$\bX_n=(X_{ni})_{i\in\sQ}$, $a_i$, and $\bxi_i=(\xi_{ij})_{j\in\sQ}$
introduced in the introduction.

\subsection{Standing assumptions}\label{SSA}

In the rest of the paper we assume
\setcounter{Aenumi}{-1}
\begin{Aenumerate}
\item\label{Atri}   
 The \Polya{} urn is triangular.
\end{Aenumerate}
(Unless we explicitly say so, for example when we discuss this property in
\refSS{SSQ}.)
For the central part of the paper (\refSs{Snot}--\ref{Sdeg})
we make also some standing technical  assumptions:

\begin{Aenumerate}
  
\item \label{A0}
The initial urn $\bX_0$ is non-random. 
Moreover, each $X_{0i}\ge0$ and $\sum a_iX_{0i}>0$.
(The results may be extended to random $\bX_0$ by conditioning on $\bX_0$.)

\item \label{A00}
If $a_i=0$, then $\bxi_i=0$, \ie, $\xi_{ij}=0$ for every $j\in\sQ$.
(This is without loss of generality, since $a_i=0$ means that balls of
colour $i$ never are drawn, and thus $\bxi_i$ does not matter.)

\item \label{A3}
For every $i\in\sQ$, either $X_{0i}>0$, or there exists $j\neq i$ such that
$\P(\xi_{ji}>0)>0$ (or both).
(This too is without loss of generality, since otherwise balls of colour $i$
can never appear, so $X_{ni}=0$ \as{} for all $n$, and we may remove the
colour $i$ from $\sQ$.)

\item \label{A2}\xdef\ALnude{A\arabic{enumi}}
$\E\xi_{ij}^2<\infty$ for all $i,j\in\sQ$.

\item \label{A+} \xdef\Anude{A\arabic{enumi}}
$\xi_{ij}\ge0$ (a.s.)\ for all $i,j\in\sQ$.
\end{Aenumerate} 

Note that \ref{A+} implies that every $X_{ni}$ is (weakly) increasing in
$n$. In particular, the urn never gets empty. Combined with \ref{A0} we see
that $\sum_i a_i X_{ni}>0$ for every $n$, and thus the 
probabilities \eqref{urn} and the urn process are well
defined.
We discuss extensions to urns not satisfying \ref{A+} (\ie, urns with
subtractions) in \refS{S-}; see \refTs{T1-}--\ref{TC-}.

\begin{remark}\label{Rmom}
The second moment condition \ref{A2} is for technical convenience.
In fact, the results (including \refT{T1}) hold  
assuming only $\E\xi_{ij}^p<\infty$ for some $p>1$.
However, this adds further arguments to an already long proof, so
we assume second moments in the main part of the paper and show the
extension to $p>1$ in \refApp{ALp}.

It seems possible that the results could extend further by
assuming only that $\E\xi_{ij}\log\xi_{ij}<\infty$,
as for diagonal urns in \cite{Athreya1969}
using  related results for (single-type) branching processes
\cite[Theorem III.7.2]{AN}.
(The results do not extend without modifications to cases
without this assumption, see \refE{ELlogL}.)
We have not pursued this, and leave it as an open problem.
\end{remark}

\subsection{General notation}

As usual, we ignore events of probability 0. We often write ``almost
surely'' or ``a.s.''\ for emphasis, but we may also tacitly omit this.

We use 
``increasing'' in the weak sense. Similarly, ``positive function'' 
means in
the weak sense, \ie{} $\ge0$.
(However, ``positive constant'' always means strictly positive; we sometimes
add ``strictly'' for emphasis, but not always.)

We let $s\land t:=\min\set{s,t}$
and $s\lor t:=\max\set{s,t}$.

We let $\bbZgeo:=\set{0,1,2,\dots}$ and $\bbZ_+:=\set{1,2,\dots}$.

We use $\asto$, $\pto$, and $\dto$ to denote
convergence almost surely, in probability,  and in distribution, respectively.

``Absolutely continuous'' 
(for a probability distribution in $\bbR$ or $\bbR^d$)
means with respect to Lebesgue measure. 
(Except in \refApp{AA} when a reference measure is explicitly specified.)
We may say that a random variable is \abscont{} when its distribution is.

We use some standard probability distributions:
$\Exp(\gl)$ is an exponential distribution with mean $\gl>0$;
we may also say with \emph{rate} $1/\gl$.
$\gG(\ga,b)$ is a Gamma distribution. 
(Thus $\Exp(\gl)=\gG(1,\gl)$.) 
$\Be(p)$ is a Bernoulli distribution.
$\NBi(r,p)$ is a negative binomial distribution. 
$\Gei(p)$ is a geometric distribution on \set{1,2,\dots}.  

For a random variable $W$, we let
$\cL(W)$ denotes the distribution of $W$, and, for any $p>0$,
\begin{align}\label{norm2}
  \norm{W}_p:=\bigpar{\E|W|^p}^{1/p}
.\end{align}
$L^p$ denotes the set of all random variables $W$ such that $\norm{W}_p<\infty$.

$C$ denotes unspecified constants that may vary from one occurrence to the
next. 
They may depend on
the activities $a_i$ and the distributions of the replacements
$\bxi_i$, and perhaps on other parameters clear from the context, 
but they never depend on $n$ or $t$. 

Let
\begin{align}\label{rij}
  r_{ij}:=\E\xi_{ij}.
\end{align}
Thus the matrix $(r_{ij})_{i,j\in\sQ}$ is the mean replacement matrix.
Note that $r_{ij}$ exists and is finite by \ref{A2}.
By \ref{A+}, we have $r_{ij}\ge0$, and $r_{ij}=0\iff\xi_{ij}=0$ a.s.

We occasionally 
(when we discuss two different urns at the same time)
denote a \Polya{} urn by $\cU$;
we then (somewhat informally) mean both the
urn process $\bX_{n}$ and its \ctime{} version $\bX(t)$ defined below, 
and also the
colour set $\sQ$ and the replacement matrix $(\xi_{ij})$.

\subsection{The colour graph}\label{SSQ}
(In this subsection, the urn does not have to be triangular.)
Recall that $\sQ$ is the set of colours. As said in the introduction, we
allow $\sQ$ to be any finite set, although it is possible to assume
$\sQ=\set{1,\dots,q}$ without loss of generality when this is convenient.

We regard the set $\sQ$ of colours as a directed graph, 
called  the \emph{colour  graph}, 
where for any distinct $i,j\in\sQ$  there is an edge $i\to j$ 
if and only if $\P(\xi_{ij}\neq0)>0$. 
In other words, $i\to j$ means that if a ball of colour $i$ is drawn, it is
possible (with positive probability) that some balls of colour $j$ are
added. Note that by \ref{A00},  $i\to j$ entails $a_i>0$, so balls of
colour $i$ may (and will) actually be drawn;
furthermore,
by \ref{A+}, 
\begin{align}
  \label{itoj}
i\to j \iff r_{ij}>0 \text{ and } i\neq j.
\end{align}

We say, again for two distinct colours $i,j\in\sQ$, that $i$ is
an \emph{ancestor} of $j$, and $j$ a \emph{descendant} of $i$,
if there exists a directed path in $\sQ$ from $i$ to $j$; we denote this by
$i\prec j$.
In other words, $i\prec j$ means that if we start the urn with only a ball
of colour $i$, then it is possible that balls of colour $j$ are added at
some later time.

If $i\in\sQ$, let $\sP_i:=\set{j\in\sQ:j\to i}$, the set of colours 
different from $i$ whose drawings may cause addition of balls of colour $i$.
We say that the colour $i$ is \emph{minimal}
if $\sP_i=\emptyset$.
We denote  the set of minimal colours by $\sQmin$.
Note that \ref{A3} can be formulated as: $X_{0i}>0$ for every $i\in\sQmin$.

We say that the urn is \emph{triangular}, if there exists a
(re)labelling of the colours by $1,\dots,q$ that makes the matrix
$(\xi_{ij})_{i,j\in \sQ}$ triangular a.s.
(Assuming \ref{A+}, this is equivalent to
the mean replacement matrix $(r_{ij})_{i,j\in \sQ}$ being triangular, \cf{}
\eqref{itoj}.)  
In other words, the urn is triangular if there exists a total ordering $<$
of the colours such that
\begin{align}\label{q1}
  i>j \implies \xi_{ij}=0 \quad \text{a.s.}
\end{align}

Using the definitions above to rewrite \eqref{q1},
we see that the urn is triangular if and only there
exists a total ordering $<$ such that, for $i,j\in\sQ$,
\begin{align}\label{q2}
  i\to j \implies i < j.
\end{align}
Furthermore, this is equivalent to
\begin{align}\label{q3}
  i\prec j \implies i < j.
\end{align}

\begin{proposition}\label{PT}
  The following are equivalent.
  \begin{romenumerate}
    
  \item\label{PT1} The urn is triangular.
  \item\label{PT2} The colour graph is acyclic.
  \item\label{PT3} The relation $\prec$ on $\sQ$ is a partial order.
  \end{romenumerate}
\end{proposition}
\begin{proof}
  \ref{PT1} implies \ref{PT2} and \ref{PT3} 
as a consequence of \eqref{q2} and \eqref{q3}.

\ref{PT2}$\iff$\ref{PT3} is easily seen. 

Finally, any partial order can be extended to a total order.
Thus, if \ref{PT3} holds, we may extend $\prec$ to a total order $<$, which
means that \eqref{q3} holds.
\end{proof}

Note that if the urn is triangular,
a colour $i$ is minimal if and only if it is minimal
in the partial order $\prec$.

\begin{remark}\label{Rnontri2}
  We may also note that a \Polya{} urn is irreducible if and only if
its colour graph is strongly connected.
Given any \Polya{} urn, we may decompose its
colour graph into its strongly
connected components, which are linked by the remaining edges in an acyclic way.
Hence the urn can be regarded as an acyclic directed network of
irreducible urns. This suggests, as mentioned in \refR{Rnontri1},
that the methods in the present paper perhaps might be combined with methods for
irreducible urns to obtain  results for general urns.
\end{remark}

\subsection{More notation}\label{SSnot2}
Let, for  $i\in\sQ$,
\begin{align}\label{gli}
  \gl_i:=a_ir_{ii}=a_i\E\xi_{ii}\ge0  
.\end{align}
In the continuous-time version introduced below, 
$\gl_i$ is the rate of additions to colour
$i$ by drawings of the same colour.
Since colours also may be added by drawings of another colour, we define
further
\begin{align}
  \label{glx}
\glx_i:=\max\bigset{\gl_j:j\preceq i}
\ge0
.\end{align}
In other words, $\glx_i$ is the largest $\gl_j$ for a colour $j$ such that
there exists a path (possibly of length 0) in $\sQ$ from $j$ to $i$.
Such a path may contain several colours $k$ with the same, maximal, $\gl_k$,
and we denote the largest number of them in a single path by $1+\kk_i$; i.e.,
\begin{align}\label{kk}
  \kk_i:=\max\bigset{\kk:\exists i_1\prec i_2\prec \dots\prec i_{\kk+1}\preceq i
\text{ with } \gl_{i_1}=\dots=\gl_{i_{\kk+1}}=\glx_i}
\ge0
.\end{align}

Define further
\begin{align}\label{hgl}
  \hgl&:=\max\set{\gl_i:i\in\sQ}\ge0
,\\
\label{hkk}
\hkk&:=\max\set{\kk_i:i\in\sQ \text{ and }\glx_i=\hgl}
\notag\\&\phantom:
=\max\bigset{\kk:\exists i_1\prec i_2\prec \dots\prec i_{\kk+1}
\text{ with } \gl_{i_1}=\dots=\gl_{i_{\kk+1}}=\hgl}
\ge0.
\intertext{
If $\hgl=0$ (i.e., if $\gl_i=0$ for every $i\in\sQ$), let further}
\label{hkk0}
  \hkko&:=
1+\max\bigset{\kk_i: i\in\sQ \text{ with }a_{i}>0}
\ge1
.\end{align}

If $\hgl>0$, define also
\begin{align}\label{gam}
  \gam_i:=\kk_i-\hkk\glx_i/\hgl,
\qquad i\in\sQ.
\end{align}

\subsection{Stochastic processes}
All our continuous-time stochastic processes are defined on $\ooo$ and are
assumed to be \cadlag{} (right-continuous with left limits).

We consider martingales without explicitly specifying the filtration; this
will always be the natural filtration
$(\cF_t)_{t\ge0}$, where  $\cF_t$ is 
generated by ``everything that has happened up to time $t$''.
If $T$ is a stopping time, then $\cF_{T}$ denotes the corresponding \gsf{}
generated by all events up to time $T$.

Given a stochastic process $W=(W(t))\too$, we define its maximal process by
\begin{align}\label{w*}
  W^*(t):=\sup_{0\le s\le t}|W(s)|,
\qquad 0\le t\le \infty.
\end{align}
(We consider only $s<\infty$, also when $t=\infty$.)
%
We further define
\begin{align}\label{nn3}
  \gD W(t) := W(t) - W(t-),
\qquad 0\le t<\infty,
\end{align}
where $W(t-)$ is the left limit at $t$, with $W(0-):=0$.

If $W$ is a process with locally bounded variation, we may define its
\emph{quadratic variation}
by
\begin{align}\label{nn4}
  [W,W]_t := \sum_{0\le s\le t} \bigabs{\gD W(s)}^2,
\qquad 0\le t\le\infty,
\end{align}
(summing over $s<\infty$ if $t=\infty$), 
noting that the sum always really is countable.
(We have no need for the definition for general semimartingales, see \eg{}
\cite[p.~519 and Theorem 26.6]{Kallenberg}
or \cite[Section II.6]{Protter}.)
Recall 
\cite[Corollary 3 to Theorem II.6.27, p.~73]{Protter}
that if $M$ is a local martingale and $\E[M,M]_t<\infty$ for some $t<\infty$,
then
\begin{align}\label{nn5}
  \E |M(t)|^2 = \E [M,M]_t,
\end{align}
and as a consequence, if $\E[M,M]_\infty<\infty$, then
$M$ is an $L^2$-bounded martingale and thus $M(\infty):=\lim_{\ttoo}M(t)$ exists
\as, and \eqref{nn5} holds for all $t\le\infty$.

Recall also Doob's inequality \cite[Proposition 7.16]{Kallenberg}
which as a special case yields, combined with \eqref{nn5},
\begin{align}\label{nn6}
\E M^*(t)^2 \le C \E |M(t)|^2 = C\E [M,M]_t,
\end{align}
provided $\E [M,M]_t<\infty$.
(Here  $C=4$, but we will not use this.)

\subsection{Continuous-time urn}\label{SSCT}
We will use the standard method of embedding the discrete-time urn in a
continuous-time process, due to \citet{AthreyaKarlin}, 
see also \cite[\S V.9]{AN}, 
and after them used in many papers.
We thus define the \emph{continuous-time urn} as a vector valued Markov
process $\bX(t)=\xpar{(X_i(t)}_{i\in\sQ}$ with given initial value
$\bX(0):=\bX_0$ such
that, for each $i\in \sQ$,  
``a ball of colour $i$ is drawn'' 
with intensity $a_iX_i(t)$; 
when a ball of colour $i$ is drawn, we add to $\bX(t)$ 
a copy of $\bxi_i$ (independent of the history). 
In the classical case (\eg{} \cite{AthreyaKarlin}) when each $X_i(t)$ is
integer valued, $\bX(t)$ is a
multitype continuous-time Markov branching process;
in general (allowing any real $X_i(t)\ge0$),
$\bX(t)$ is a (vector-valued) 
continuous-time
continuous-state branching process 
(abbreviated \emph{CB process})
as defined by \citet{Jirina}, see also \eg{} \citet{Li}.
(The process $\bX(t)$ is of jump-type, as in \cite[Section 3]{Jirina}.)
Note that \ref{A2} implies $\E\xi_{ij}<\infty$ for all $i,j\in\sQ$, which is
a well-known sufficient condition for non-explosions; i.e., there exists 
such a process $\bX(t)$ with $\bX(t)$ finite for all $t\in\ooo$.

Since $a_iX_i(0)>0$ for some $i$ by \ref{A0}, and $X_i(t)$ is
increasing by \ref{A+}, there will \as{} be infinitely many draws in the urn.
We let $\TT_n$ be the $n$th time that a ball is drawn, with $\TT_0:=0$. 
Then the discrete-time urn $\bX_n$ in \refS{S:intro} can be realized as
\begin{align}\label{ct}
  \bX_n := \bX(\TT_n).
\end{align}
We assume \eqref{ct} throughout the paper.

Since $\bX(t)$ does not explode, there is \as{} 
only a finite number of draws up to
any finite time $t$, and thus $\TT_n\asto\infty$ as \ntoo.
Note that we denote the discrete-time urn by $\bX_n=(X_{ni})_i$ and 
the continuous-time urn by $\bX(t)=(X_i(t))_i$.

Of course, the \ctime{} urn can also be studied for its own sake,
see for example \cite{ChenMahmoud}.

\section{Analysis of one colour}\label{S1}

In this section, we study one fixed colour $i\in\sQ$ in the continuous-time urn.

Recall that $X_i(t)$ is the number (amount) of balls of colour $i$ at time $t$
in the continuous-time  urn.
There are several possible sources of these balls:
some may be there from the beginning, some may be added when a ball of some
other colour $j$ is drawn (with $j\to i$ and thus $j\prec i$), and in both
cases these balls of colour $i$ may later be drawn and produce further
generations of balls of colour $i$ (provided $\gl_i>0$).

We begin by considering, in the following two subsections,
two simpler special cases. We will then in the final subsection, rather easily, 
treat the general case 
by combining these  special cases.

We first state some simple general results; these are more or less known,
see \eg{} \cite[Lemma 9.3]{SJ154} for a related result,
but for convenience we give full proofs.

Fix a colour $j$. 
Let $0<T_1<T_2<\dots$ be the times that a ball of
colour $j$ is drawn, and let $N(t):=\bigabs{\set{k:T_k\le t}}$ denote the
corresponding counting process; 
i.e., $N(t)$ is the number of draws of colour $j$ up to time $t$.
%
These draws occur with intensity $a_j X_j(t)$, which means that
\begin{align}\label{j7}
  \tN(t):=N(t)-a_j\intot X_j(s)\dd s,
\qquad t\ge0,
\end{align}
is a local martingale. In fact, $\tN(t)$ is a martingale, which we verify by
the following simple lemma. Recall the notation \eqref{w*}.
\begin{lemma}\label{LtN}
Suppose that $\E X_{j}(t)<\infty$ for some $t\in\ooo$.
Then,
  \begin{align}\label{ltn0}
\E N(t)&<\infty,
\\\label{ltn}
    \E \tN^*(t)&
=\E\sup_{s\le t}|\tN(s)|
<\infty.
  \end{align}
In particular, if\/ $\E X_{j}(t)<\infty$ for every $t<\infty$,
then the local martingale $\tN(t)$ is a martingale.
\end{lemma}
\begin{proof}
  By the definition of local martingale, there exists an increasing sequence
  of stopping times $\tau_m$, $m\ge1$, such that $\tau_m\upto\infty$ \as{} 
as  \mtoo, and $\tN(t\land \tau_m)$ is a martingale for each $m$.
In particular, since $\tN(0)=0$,
\begin{align}\label{ltn1}
\E N\xpar{t\land\tau_m} 
=\E \tN\xpar{t\land\tau_m} + a_j\E \int_0^{t\land \tau_m} X_j(s )\dd s
= a_j\E \int_0^{t\land \tau_m} X_j(s )\dd s.
\end{align}
Since $N(t)$ and $X_j(t)$ are increasing 
positive 
functions of $t$, we may use monotone convergence and let \mtoo{} to obtain
\begin{align}\label{ltn2}
  \E N(t) = a_j\E \int_0^{t} X_j(s )\dd s
\le a_j \E \bigsqpar{tX_j(t)} <\infty.
\end{align}
The monotonicity of $N(t)$ further implies
\begin{align}\label{ltn3}
\tN^*(t)\le  N(t)+a_j\intot X_j(s)\dd s
\le  N(t)+a_jt X_j(t),
\end{align}
and thus \eqref{ltn} follows by \eqref{ltn2}.

The final statement follows since a local martingale with integrable maximal
function is a martingale.
\end{proof}

\begin{lemma}\label{L9+}
  Suppose that $\E X_j(t)<\infty$ for some $t\in\ooo$.
  \begin{romenumerate}
  \item \label{L9+a}
Let $f$ be a positive or bounded measurable function on $[0,t]$.
Then
\begin{align}\label{l9+}
  \E \sumk \indic{T_k\le t}f(T_k)
=\E \intot f(s)\dd N(s)
= a_j\E\intot f(s)X_j(s)\dd s.
\end{align}

\item \label{L9+b}
Let also $(\eta_k)\xoo$  be a sequence of identically distributed
random variables with finite mean $\E\eta_1$
such that $\eta_k$ is independent of $T_k$.
Then
\begin{align}\label{l9+b}
  \E \sumk \indic{T_k\le t}f(T_k)\eta_k
= a_j\E\eta_1\E\intot f(s)X_j(s)\dd s.
\end{align}
  \end{romenumerate}
\end{lemma}
\begin{proof}
\resetsteps
\stepxx
The left and middle terms in \eqref{l9+} are the same, since 
$\intot f(s)\dd N(s)=\sumk \indic{T_k\le t}f(T_k)$ by the definition of $N$.

\stepxx 
Suppose that $f(s)=\indic{a< s\le b}$, the indicator function of an
interval $(a,b]\in[0,t]$.
Then, by \eqref{j7}, 
\begin{align}\label{l9+1}
  \intot f(s)\dd N(s) - a_j\intot f(s)X_j(s)\dd s
=\intot f(s)\dd\tN(s)  
=\tN(b)-\tN(a),
\end{align}
and $\E\sqpar{\tN(b)-\tN(a)}=0$ by \refL{LtN}; hence \eqref{l9+} holds for
such $f$.

\stepxx
The monotone class theorem \cite[Theorem 1.2.3]{Gut} now shows that \eqref{l9+}
holds for the indicator function $f(s)=\indic{s\in A}$ of any Borel set
$A\in[0,t]$.

\stepxx
  By linearity, \eqref{l9+} holds for any positive simple function $f$. 
Then, by monotone convergence, \eqref{l9+} holds for any 
positive measurable function, and by linearity again for any bounded
measurable function.
This proves \ref{L9+a}.

\stepxx
In \ref{L9+b}, we may decompose $\eta_k$ into its positive and negative
parts; thus it suffices to consider $\eta_k\ge0$. Then the sum in
\eqref{l9+b}
is well-defined, and
\begin{align}\label{l9+bb}
  \E \sumk \indic{T_k\le t}f(T_k)\eta_k
=  \sumk \E\bigsqpar{\indic{T_k\le t}f(T_k)\eta_k}
=  \E\eta_1  \sumk \E\bigsqpar{\indic{T_k\le t} f(T_k)},
\end{align}
and \eqref{l9+b} follows from \eqref{l9+}.
\end{proof}

\subsection{A colour not influenced by others}\label{SS11}
In this subsection, we assume that  $\xi_{ji}=0$ \as{}
for all colours $j\neq i$.  
Equivalently, $j\not\to i$ for $j\in\sQ$, i.e., $\sP_i=\emptyset$;
in other words,  $i$ is a minimal colour.
This means that $X_i(t)$ is affected only by
draws of the same colour $i$, and we may thus ignore all other colours
and regard $X_i(t)$ as a continuous-time urn with a single colour.
In other words, $X_i(t)$ is a one-dimensional
CB process, 
starting at some given $X_i(0)$ and
adding copies of $\xi_{ii}$ with intensity $a_iX_i(t)$.
We write  $X_i(0)=x_0$.
Note that our assumption \ref{A3} means $x_0>0$, but for completeness  
we allow also the trivial case $x_0=0$ in the present subsection.
(In this subsection, we really use only the assumptions
\ref{A0}, \ref{A2}, and \ref{A+}, and only for the colour $i$.)
Note also that $\glx_i=\gl_i$ by \eqref{glx}.

We will only need a few simple facts about CB processes;
see further \eg{}
\cite{Jirina,Lamperti-BAMS,Grey,Bingham,Li} 
where many more results are given.

Recall that $\gl_i=a_ir_{ii}$. If $\gl_i=0$, then there are no additions at
all, and $X_i(t)=X_i(0)=x_0$ is constant.

It is easy to see that since $r_{ii}<\infty$, the CB
process $X_i(t)$ is well defined and non-explosive (i.e., finite for all $t$),
for any given $x_0\ge0$. 
Moreover
\cite{Grey}, \cite[Proposition 2.2]{Bingham},
\begin{align}
  \label{cb4}
\E X_i(t)=e^{\gl_i t}x_0
.\end{align}
This implies by conditioning and the Markov property that
\cite[Theorem III.7.1]{AN}
\begin{align}\label{cb3}
{e^{-\gl_i t}X_i(t)}
\quad  \text{is a martingale}
.\end{align}

The following result too well known, but we include a proof for
completeness, and since we discuss modifications of it later. 
See \eg{} 
\cite[Theorem V.8.2]{AN} 
and \cite[Lemma 9.5]{SJ154} 
for a vector-valued extension.

\begin{lemma}\label{LM}
Suppose that $\xi_{ji}=0$ \as{}
for all $j\neq i$. 
Then
$
  e^{-\gl_i t} X_i(t)
$ 
is an $L^2$-bounded martingale,
and thus 
\begin{align}\label{lm1}
    e^{-\gl_i t} X_i(t)\asto \WXi
\end{align}
for some random variable $\WXi$.
Furthermore, with $x_0=X_i(0)$,
\begin{align}\label{lm2}
\E \WXi&= x_0,
\\\label{lm3}
  \Var \WXi&
=  x_0 \frac{a_i}{\gl_i} \E \xi_{ii}^2,
\end{align}
where we interpret $\frac{0}{0}$ as $0$.
Hence
\begin{align}\label{lm3b}
\Var\bigsqpar{e^{-\gl_i t} X_i(t)}
&\le  \Var \WXi
\le C x_0 
,
\\
\label{lm3a}
\E\Bigabs{\sup_{t\ge0} e^{-\gl_i t} X_i(t)}^2 
&\le C \E \WXi^2 \le C(x_0+x_0^2).  
\end{align}
Furthermore, if $x_0>0$, then $0<\WXi<\infty$ a.s.
\end{lemma}

\begin{proof}
The case $\gl_i=0$ is trivial, with $\WXi=x_0$. 
The same holds if $x_0=0$.
We may thus assume $\gl_i>0$ and  $x_0>0$.

We argue as in \cite[Section 9]{SJ154}.
As above (now taking $j=i$),
let $0<T_1<T_2<\dots$ be the times that a ball of colour $i$ is
drawn, and
let $N(t):=\bigabs{\set{i:T_i\le t}}$ denote the corresponding counting process-
Furthermore, let $\eta_k:=\gD X_i(T_k)$ be the number of balls of colour $i$
added at the $k$-th draw.
Thus $\eta_1,\eta_2,\dots$ is a sequence of independent copies of
$\xi_{ii}$.

Let $M(t):=e^{-\gl_i t}X_i(t)$. Then $M$ is a martingale by \eqref{cb3}, and
its quadratic variation is by \eqref{nn4} given by
\begin{align}\label{lm4}
  [M,M]_t=\sum_{0\le s\le t}|\gD M(s)|^2
= x_0^2+\sum_{T_k\le t}e^{-2\gl_iT_k}|\gD X_i(T_k)|^2
= x_0^2+\sum_{T_k\le t}e^{-2\gl_iT_k}\eta_k^2.
\end{align}
Hence, it follows from \refL{L9+}\ref{L9+b} that
\begin{align}\label{lm5}
\E  [M,M]_t
= x_0^2+ a_i\E[\xi_{ii}^2]\intot e^{-2\gl_is}\E X_i(s) \dd s.
\end{align}
Hence, writing $\gb:=\E\xi_{ii}^2$, \eqref{cb4} yields
\begin{align}\label{lm6}
\E  [M,M]_t
= x_0^2+a_i \gb x_0\intot e^{-\gl_is} \dd s
= x_0^2+ x_0 \frac{a_i}{\gl_i} \gb \bigpar{1- e^{-\gl_it}}.
\end{align}
This shows by \eqref{nn5} that $M$ is an $L^2$-bounded martingale; thus
\eqref{lm1} holds  for some $\WXi=M(\infty)$. Clearly $\WXi\ge0$.  
We have $\E M(\infty)=\E M_0=x_0$, which yields \eqref{lm2}.
Furthermore, \eqref{lm3} holds by \eqref{nn5} and \eqref{lm6} again,
together with \eqref{lm2}; this yields also  \eqref{lm3b}.
Doob's inequality \eqref{nn6} yields \eqref{lm3a}.

Finally, the distribution of $\WXi$ depends on $x_0$; thus let us denote
$\WXi$ by 
$\WXi(x_0)$. The CB property implies that if $x,y\ge0$, then
$\WXi(x+y)\eqd \WXi(x)+\WXi'(y)$, where $\WXi'(y)$ is a copy of $\WXi(y)$
independent of 
$\WXi(x)$. Let $p(x):=\P\bigpar{\WXi(x)=0}\in\oi$. 
It follows that for any $x,y\ge0$,
\begin{align}\label{lm7}
  p(x+y)=p(x)p(y),
\end{align}
and thus there exists $c\in\ooox$ such that
\begin{align}\label{lm8}
\P(\WXi(x)=0)= p(x)=e^{-cx}, \qquad x>0.
\end{align}
We must have $c>0$, since otherwise $\WXi=0$ \as{} for any $x_0$, which
contradicts \eqref{lm2}. 
The Markov property and \eqref{lm8} yield, for any
$t\in\ooo$, 
\begin{align}\label{lm9}
  \P\bigpar{\WXi=0\mid X_i(t)} = e^{-c X_i(t)}.
\end{align}
Thus, taking the expectation,
\begin{align}\label{lm10}
  \P\bigpar{\WXi=0} = \E e^{-c X_i(t)}.
\end{align}
It is clear that if $x_0>0$ and
$\gl_i>0$, so $a_i>0$ and $\E\xi_{ii}>0$, then there is \as{}
an infinite number of draws, and the (standard) law of large numbers shows
that $X_i(t)\asto\infty$.  Consequently, letting \ttoo{} in \eqref{lm9} yields,
by dominated convergence,
\begin{align}\label{lm11}
  \P(\WXi=0)=0.
\end{align}
\end{proof}

By \eqref{lm3}, the limit $\cX_i$ is degenerate only in the trivial cases when
$x_0=0$ or $a_i=0$.  We note that except in these cases, 
the limit has an \abscont{} distribution.

\begin{lemma}\label{LM<}
  In \refL{LM}, suppose further that $\gl_i>0$ and $x_0>0$.
Then the distribution of $\cX_i$ is \abscont.
\end{lemma}

\begin{proof}
  Let $T_1$ be the first time that a ball of colour $i$ is drawn; then
  $T_1\in\Exp(1/(a_ix_0))$. 
The distribution of the balls added at $T_1$ is independent of $T_1$, and
thus $X_i(T_1)$ is independent of $T_1$.
It follows by the strong Markov property that the stochastic process
$Y(t):=X_i(T_1+t)$, $t\ge0$, is independent of $T_1$.
By \eqref{lm1}, we have as \ttoo,
\begin{align}\label{oc1}
  e^{-\gl_i t}Y(t)=e^{\gl_i T_1}e^{-\gl_i(T_1+t)}X_i(T_1+t)\asto 
\cY:=e^{\gl_i T_1}\cX_i,
\end{align}
and it follows that $\cY$ is independent of $T_1$. Consequently,
\begin{align}\label{oc2}
  \cX_i=e^{-\gl_i T_1}\cY,
\end{align}
where the two factors on the \rhs{} are independent. 
The result follows from \eqref{oc2} since
$e^{-\gl_i T_1}$ has an absolutely continuous distribution
and
$\cY>0$ \as{} (because $\cX_i>0$ \as{} by \refL{LM}).
\end{proof}

\subsection{A colour only produced by one other colour}\label{SS12}
We continue to consider a fixed colour $i$.
In this subsection we assume that $|\sP_i|=1$; thus
there is exactly one colour $j\in\sQ$ such that $j\to i$.
We assume also that $X_i(0)=0$, so there are initially no balls of colour $i$.
Recall that the assumption $j\to i$ means $r_{ij}>0$, and thus implicitly
$a_j>0$.

Since the colour graph is acyclic, there is no feedback from colour $i$ to
colour $j$; thus we may regard the entire 
process $X_j(t)$, $t\in\ooo$, as known, and consider only its
effect on $X_i(t)$. 

We may regard $X_i(t)$ as a CB process with immigration.
Let again $0<T_1<T_2<\dots$ be the times that a ball of colour $j$ is
drawn, and 
let $\eta_1,\eta_2,\dots$ be the corresponding number (amount) of balls of
colour $i$ that are added.
In case there is only a finite number $K$ of times that a ball of colour $j$ is
drawn, we define for completeness $T_k=\infty$ for $k>K$, and pick $\eta_k$
with the correct distribution and independent of everything else.
We regard each $\eta_k$ as an immigrant coming at $T_k$. 
(Some $\eta_k$ may be 0; this is no problem.)
We can separate the descendants of each of these immigrants, and write
\begin{align}\label{j3}
  X_i(t)=\sum_{k:T_k\le t} Y_k(t-T_k)
=\sumk \indic{T_k\le t}Y_k(t-T_k)
\end{align}
where each $Y_k(t)$, conditioned on $\eta_k$, 
is a copy of the single-colour CB process in \refSS{SS11}
with $Y_k(0)=\eta_k$. Furthermore, conditionally on $(\eta_k)_k$, 
the processes $Y_k(t)$ are independent.
Recall that
the additions $\eta_1,\eta_2,\dots$  
are \iid{} copies of $\xi_{ji}$; in particular, they have expectations
$\E\eta_k=r_{ji}$. 
The times $T_k$ are stopping times; moreover, the process $Y_k$ is
independent of $T_k$ and the \gsf{} $\cF_{T_k}$.

The following lemma contains much of the technical parts of our argument.
It makes an assumption  on the growth of $X_j(t)$
that later will be justified by an induction argument.

We define, for  $\mmi\in\sQ$,
\begin{align}\label{j4}
\tX_\mmi(t)&:=t^{-\kk_\mmi}e^{-\glx_\mmi t} X_\mmi(t),
\qquad 0<t<\infty,
\\ \label{j4*}
\tXX_\mmi&:= \sup_{t\ge0}\bigcpar{(t+1)^{-\kk_\mmi}e^{-\glx_\mmi t} X_\mmi(t)}.
\end{align}
We use powers of $t$ in \eqref{j4} and of $t+1$ in \eqref{j4*} for technical
convenience below (and we therefore use a notation with ${}^{**}$ instead of
${}^*$); 
the difference is not important since we are mainly
interested in limits as \ttoo.

\begin{lemma}\label{L12}
Let\/ $i\in\sQ$.
Suppose that 
there is exactly one colour $j\in\sQ$ such that $j\to i$, 
and that $X_i(0)=0$.
Suppose further that
\begin{align}\label{j1}
 \tX_j(t) \asto \cX_j
\qquad \text{as }\ttoo,
\end{align}
for some random variable $\cX_j\ge0$, and
\begin{align}\label{j2}
  \bignorm{\tXX_j}_2 <\infty.
\end{align}
Then
\begin{align}\label{ji1}
 \tX_i(t) \asto \cX_i \qquad \text{as }\ttoo,
\end{align}
for some $\cX_i\ge0$ and
\begin{align}\label{ji2}
  \bignorm{\tXX_i}_2 <\infty.
\end{align}
Furthermore, \as,
if\/ $\cX_j>0$, then $\cX_i>0$.

Moreover, if $\gl_i\le\glx_j$, and thus $\glx_i=\glx_j$, then
\begin{align}\label{ji3}
  \cX_i=
  \begin{cases}
    \frac{a_j r_{ji}}{\glx_j-\gl_i} \cX_j, & \gl_i<\glx_j,
\\
    \frac{a_j r_{ji}}{\kk_i} \cX_j, & \gl_i=\glx_j.
  \end{cases}
\end{align}
\end{lemma}

\begin{proof}
Using \eqref{j3}, and recalling $Y_k(0)=\eta_k$, 
we make the decomposition
\begin{align}
e^{-\gl_i t} X_i(t)
&=\sumk  \indic{T_k\le t}e^{-\gl_i  T_k}\bigpar{e^{-\gl_i(t-T_k)}Y_k(t-T_k)-Y_k(0)}
\label{jz1}\\&\hskip4em
+ \sumk  \indic{T_k\le t}e^{-\gl_i T_k}\xpar{\eta_k-r_{ji}}
\label{jz2}\\&\hskip4em
+ r_{ji}\Bigpar{\sumk
 \indic{T_k\le t}e^{-\gl_i T_k}-\intot e^{-\gl_i s} a_j X_j(s)\dd s}
\label{jz3}\\&\hskip4em
+ r_{ji}a_j \intot e^{-\gl_i s} X_j(s)\dd s
\label{jz4}\\&\label{jzz}
=: Z_1(t) + Z_2(t) + r_{ji} Z_3(t)+ a_jr_{ji} Z_4(t)
,\end{align}
say.
We consider the four processes $Z_\ell(t)$ in \eqref{jzz} separately
(initially, at least).
Moreover, 
recalling \eqref{glx} and \eqref{kk},
we consider sometimes separately the  three cases:
\begin{romenumerate}
  
\item \label{gli>}
$\gl_i > \glx_j$: then $\glx_i=\gl_i$ and $\kk_i=0$.
\item \label{gli=}
$\gl_i = \glx_j$:
then $\glx_i=\gl_i$, and $\kk_i= \kk_j+1$. 
\item \label{gli<}
$\gl_i < \glx_j$:
then $\glx_i=\glx_j$, and $\kk_i= \kk_j$. 
\end{romenumerate}

We define, in analogy with \eqref{j4}--\eqref{j4*} (but note the different
exponents $\gl_i-\glx_i$),
for $\ell=1,2,3,4$,
\begin{align}\label{j4z}
\tZ_\ell(t)&:=t^{-\kk_i}e^{(\gl_i-\glx_i) t} Z_\ell(t),
\\ \label{j4z*}
\tZZ_\ell&:= \sup_{t\ge0}\bigcpar{(t+1)^{-\kk_i}e^{(\gl_i-\glx_i) t} |Z_\ell(t)|}.
\end{align}
Then \eqref{jzz} yields
\begin{align}\label{jxz}
  \tX_i(t)&= \tZ_1(t)+\tZ_2(t)+ r_{ji} \tZ_3(t) + a_jr_{ji}\tZ_4(t),
\\\label{jxz*}
  \tXX_i &\le \tZZ_1+\tZZ_2+ C \tZZ_3 + C \tZZ_4.
\end{align}
We will prove, for $\ell=1,\dots,4$ and some $\cZ_\ell$,
\begin{align}\label{jtz}
  \tZ_\ell(t)&\asto \cZ_\ell
\qquad\text{as }\ttoo,
\\\label{jtz*}
\norm{\tZZ_\ell}_2&<\infty,
\end{align}
and then \eqref{ji1}--\eqref{ji2} follow from \eqref{jxz}--\eqref{jxz*}.

We treat $Z_1(t),\dots,Z_4(t)$ in \eqref{jz4} in reverse order,
partly because $Z_4(t)$ will turn out to be the main term 
(sometimes at least).

\resetsteps
\stepx{$Z_4$} 
By \eqref{jz4} and \eqref{j4*},
\begin{align}\label{j41}
0\le  Z_4(t)&\le \tXX_j\intot (s+1)^{\kk_j}e^{(\glx_j-\gl_i) s} \dd s
\notag\\&
\le
  \begin{cases}
    C \tXX_j= C(t+1)^{\kk_i} \tXX_j, & \gl_i>\glx_j,
\\
    (t+1)^{\kk_j+1} \tXX_j = (t+1)^{\kk_i} \tXX_j
, & \gl_i=\glx_j,
\\
    C(t+1)^{\kk_j}e^{(\glx_j-\gl_i)t} \tXX_j, & \gl_i<\glx_j
.  \end{cases}
\end{align}
It follows that in all three cases \ref{gli>}--\ref{gli<},
$\tZZ_4\le C\tXX_j$. In particular, \eqref{j2} implies
\begin{align}\label{j42}
 \bignorm{\tZZ_4}_2<\infty.
\end{align}

For \eqref{jtz} we treat the three cases \ref{gli>}--\ref{gli<} separately,
in each case using \eqref{j4z} and \eqref{j4}--\eqref{j4*}.
First, in case \ref{gli>}, we let \ttoo{} in \eqref{jz4} and \eqref{j41} and
conclude that
\as{}
\begin{align}\label{j43}
\tZ_4(t)=Z_4(t)\to Z_4(\infty):=\intoo e^{-\gl_i s}X_j(s)\dd s \le C \tXX_j<\infty.
\end{align}
Hence
\begin{align}\label{j44}
  \tZ_4(t)\asto \cZ_4=Z_4(\infty).
\end{align}
Note that if $\cX_j>0$, then $X_j(t)>0$ for large $t$, and thus 
$\cZ_4=Z_4(\infty)>0$ by
\eqref{j43}.

In case \ref{gli=}, we have for $t\ge1$,
using the change of variables $s=xt$,
\begin{align}\label{j45}
  \tZ_4(t)&
=t^{-\kk_i}Z_4(t)
= t^{-\kk_j-1}\intot s^{\kk_j}\tX_j(s)\dd s
\notag\\&
= t^{-\kk_j-1}\intoi s^{\kk_j}\tX_j(s)\dd s
+ \int_{1/t}^1 x^{\kk_j}\tX_j(xt)\dd x
\notag\\&
\to
 0
+ \int_{0}^1 x^{\kk_j}\cX_j\dd x
=(\kk_j+1)\qw\cX_j
=\kk_i\qw\cX_j
,\end{align}
by \eqref{j1} and dominated convergence, which applies since
\eqref{j4}--\eqref{j4*} show that 
 $\sup_{s\ge1}\tX_j(s)\le
2^{\kk_j}\tXX_j<\infty$ a.s.

In case \ref{gli<}, similarly, now with $s=t-u$,
\begin{align}\label{j46a}
  \tZ_4(t)&
= t^{-\kk_j}e^{(\gl_i-\glx_j)t}\intot s^{\kk_j}e^{(\glx_j-\gl_i)s}\tX_j(s)\dd s
\notag\\&
= \intot
  \Bigpar{\frac{t-u}t}^{\kk_j}e^{-(\glx_j-\gl_i)u}\tX_j(t-u)\dd u
\notag\\&
\to
\intoo e^{-(\glx_j-\gl_i)u}\cX_j\dd u
=
(\glx_j-\gl_i)\qw\cX_j,
\end{align}
using dominated convergence again, 
justified by, for $t\ge1$,
\begin{align}\label{j46b}
  \Bigpar{\frac{t-u}t}^{\kk_j}\tX_j(t-u)
\le
  \Bigpar{\frac{t-u+1}t}^{\kk_j}\tXX_j
\le 2^{\kk_j}\tXX_j<\infty
.\end{align}

We have thus shown \eqref{jtz} and \eqref{jtz*} for $\ell=4$ in all cases,
with $\cZ_4>0$ when $\cX_j>0$.
Furthermore,  in cases \ref{gli=} and \ref{gli<}, we have
\begin{align}\label{j46=}
  \cZ_4 =
  \begin{cases}
    \kk_i\qw\cX_j, & \gl_i=\glx_j,
\\
(\glx_j-\gl_i)\qw\cX_j, & \gl_i<\glx_j.
  \end{cases}
\end{align}

\stepparti{$Z_3$} 
We have, recalling \eqref{j7},
\begin{align}\label{j31}
  Z_3(t)
= \intot e^{-\gl_i s} \dd \tN(s),
\end{align}
which is a local martingale since $\tN(t)$ is;
in fact it is a martingale since 
\begin{align}\label{j32}
  Z_3^*(t)\le N(t) + a_j\intot e^{-\gl_i s} X_j(s)\dd s 
\le N(t) + t X_j(t),
\end{align}
and thus $\E Z_3^*(t)<\infty$ by \refL{LtN}, noting that the assumption
\eqref{j2} implies $\E X_j(t)<\infty$ for all $t\ge0$.
Since $Z_3(t)$ has locally bounded variation and jumps only at the times
$T_k$, 
with $\gD Z_3(T_k)=e^{-\gl_i T_k}$,
its quadratic variation is 
\begin{align}\label{j33}
  [Z_3,Z_3]_t &= 
\sumk \indic{T_k\le t}e^{-2\gl_i T_k}
= \intot e^{-2\gl_i s}\dd N(s).
\end{align}
The draws $T_k$ occur with rate $a_jX_j(t)$, 
and it follows by \eqref{j33} and \refL{L9+}
that, recalling \eqref{j4*},
\begin{align}\label{j35}
\E |Z_3(t)|^2&
=\E  [Z_3,Z_3]_t 
= a_j\E\intot e^{-2\gl_is}X_j(s)\dd s
\notag\\&
\le a_j\E\tXX_j\intot (s+1)^{\kk_j}e^{(\glx_j-2\gl_i)s}\dd s
.\end{align}
The rest of the argument will be the same for $Z_2$ and $Z_1$, so we first
consider them.

\stepparti{$Z_2$} 
Each term in the sum \eqref{jz2} is a martingale, since $T_k$ is a stopping
time, and $\eta_k-r_{ji}$ has mean 0 and 
is independent of the \gsf{} $\cF_{T_k}$.
Hence, the sum $Z_2(t)$ is at least a local martingale (since stopping at
any $T_k$ yields a finite sum and thus a martingale).
Since $Z_2(t)$ jumps only at the times $T_k$, and is constant in between,
its quadratic variation is given by
\begin{align}\label{j21}
  [Z_2,Z_2]_t 
=\sumk \indic{T_k\le t}e^{-2\gl_i T_k}(\eta_k-r_{ji})^2.
\end{align}
Hence, again using the independence between $\eta_k$ and $T_k$,
\begin{align}\label{j22}
\E  [Z_2,Z_2]_t &
=\sumk \E\bigsqpar{ \indic{T_k\le t}e^{-2\gl_i T_k}}
\E \bigsqpar{(\eta_k-r_{ji})^2}
\notag\\&
=\Var[\eta_1]\E\sumk { \indic{T_k\le t}e^{-2\gl_i T_k}}.
\end{align}
We recognize the final sum from \eqref{j33}, and conclude
\begin{align}\label{j23}
\E  [Z_2,Z_2]_t &
=C \E  [Z_3,Z_3]_t. 
\end{align}

\stepparti{$Z_1$}\label{stepZ1}
We write the sum in \eqref{jz1} as $Z_1(t)=\sumk Z_1\kkk(t)$, with
\begin{align}\label{j11}
Z\kkk_1(t):= 
\indic{t\ge T_k}e^{-\gl_i  T_k}\bigpar{e^{-\gl_i(t-T_k)}Y_k(t-T_k)-Y_k(0)}
.\end{align}
It is easily seen that
each $Z\kkk_1(t)$ 
is a martingale,
since $T_k$ is a stopping time
and $e^{-\gl_i t}Y_k(t)-Y_k(0)$ is a martingale starting at 0,
which furthermore is independent of $\cF_{T_k}$.
Hence, for every finite $m\ge1$, the finite sum
\begin{align}\label{j12}
\mmZ_1(t):=\sum_{k=1}^m Z\kkk_1(t)  
\end{align}
is a martingale, and thus
$Z_1(t\bmin T_m)=\mmZ_1(t\bmin T_m)$ is a martingale.
Consequently, $Z_1$ is a local martingale.
Furthermore, conditioned on all $T_k$ and $\eta_k$, the 
processes $Z_1\kkk(t)$ are independent  and thus \as{}
they jump at different times. (Note that  $Z_1\kkk(0)=0$.)
Hence, by \eqref{nn4},
\begin{align}\label{j5a}
  [Z_1,Z_1]_t = \sumk [Z_1\kkk,Z_1\kkk]_t
\end{align}
and thus, using \eqref{j11} and \eqref{nn5} again together with \eqref{lm3b},
\begin{align}\label{j5b}
\E \bigsqpar{  [Z_1,Z_1]_t\mid (T_k,\eta_k)\xoo}& = 
\sumk \E\bigsqpar{[Z_1\kkk,Z_1\kkk]_t\mid (T_k,\eta_k)\xoo}
\notag\\&
= \sumk\indic{T_k\le t} e^{-2\gl_i  T_k}
  \Var\bigsqpar{e^{-\gl_i(t-T_k)}Y_k(t-T_k)\mid T_k,\eta_k}
\notag\\&
\le \sumk\indic{T_k\le t} e^{-2\gl_i T_k}C\eta_k.
\end{align}
This yields, by taking the expectation,
\begin{align}\label{j6}
\E[Z_1,Z_1]_t
\le 
C \E \sumk\indic{T_k\le t} e^{-2\gl_i T_k}.
\end{align}

This is the same sum as in \eqref{j33}, and we conclude
\begin{align}\label{j16}
\E  [Z_1,Z_1]_t &
\le 
C \E  [Z_3,Z_3]_t. 
\end{align}

\stepx{$Z_1,Z_2,Z_3$, final part}\label{stepZ123b}
Let $\ell\in\set{1,2,3}$.
 In all three cases, 
$Z_\ell(t)$ is a local martingale such that,
by \eqref{j23}, \eqref{j16}, and \eqref{j35}, 
for $t\in\ooo$,
\begin{align}\label{j51}
\E  [Z_\ell,Z_\ell]_t 
\le C\E  [Z_3,Z_3]_t 
\le C\E\tXX_j\intot (s+1)^{\kk_j}e^{(\glx_j-2\gl_i)s}\dd s
.\end{align}
In particular, \eqref{j51} shows that 
$\E [Z_\ell,Z_\ell]_t <\infty$ for every $t<\infty$, 
and thus $Z_\ell(t)$ is a square
integrable martingale,
and $\E|Z_\ell(t)|^2=\E[Z_\ell,Z_\ell]_t$.

We now consider another partition into three separate cases for $\glx_j$ and
$\gl_i$.

\pfcasex{$\glx_j<2\gl_i$} \label{step5-i'}
Then we can let \ttoo{} in \eqref{j51} and obtain 
$\E[Z_\ell,Z_\ell]_\infty<\infty$. Consequently, $Z_\ell(t)$ is an $L^2$-bounded
martingale, and thus 
\begin{align}\label{j52}
Z_\ell(t)\asto Z_\ell(\infty)<\infty,
\qquad \text{as \ttoo}, 
\end{align}
and
$Z_\ell^*(\infty):=\sup_{t\ge0}|Z_\ell(t)|\in L^2$.
Furthermore (as always), $\glx_i\ge\gl_i$, and thus 
\eqref{j4z*} implies $\tZZ_\ell\le Z^*_\ell(\infty)$; consequently,
$\norm{\tZZ_\ell}_2<\infty$. 

We consider two subcases, recalling the cases discussed at the beginning of
the proof:
\begin{enumerate}
 \renewcommand{\labelenumi}{\textup{(\roman{casex}${}'$\alph{enumi})}}%
 \renewcommand{\theenumi}{\labelenumi}%
\item \emph{Case \ref{gli>}, $\gl_i>\glx_j$.}
Then $\glx_i=\gl_i$ and $\kk_i=0$;
hence \eqref{j4z}
yields $\tZ_\ell(t)=Z_\ell(t)$ and thus \eqref{j52} yields 
$\tZ_\ell(t)\to\cZ_\ell:=Z_\ell(\infty)$ a.s.

\item \emph{Case \ref{gli=} or \ref{gli<}, $\gl_i\le\glx_j$.}
Then $\glx_i>\gl_i$ or $\kk_i\ge1$ and thus
\eqref{j4z} and \eqref{j52} yield
$\tZ_\ell(t)\to \cZ_\ell:=0$ a.s.
\end{enumerate}

\pfcasex{$\glx_j\ge2\gl_i$ and $\glx_j>0$}\label{step5-ii'} 
Let, for $n\ge1$,
\begin{align}\label{j53}
  \tZcl(n)
&:=\sup_{n-1\le t\le n}(t+1)^{-\kk_i}e^{(\gl_i-\glx_i) t} |Z_\ell(t)|
.\end{align}
It follows from  Doob's inequality \eqref{nn6} and \eqref{j51} that,
with $\bC:=C\E\tXX_j$,
\begin{align}\label{j54}
\E \tZcl(n)^2&
\le C n^{-2\kk_i}e^{2(\gl_i-\glx_i) n}
\E Z^*_\ell(n)^2
\notag\\&
\le C n^{-2\kk_i}e^{2(\gl_i-\glx_i) n}
\E[Z_\ell,Z_\ell]_{n}
\notag\\&
\le \bC n^{-2\kk_i}e^{2(\gl_i-\glx_i) n}
\int_0^{n} (s+1)^{\kk_j}e^{(\glx_j-2\gl_i)s}\dd s
\notag\\&
\le \bC n^{-2\kk_i}e^{2(\gl_i-\glx_i) n}n^{\kk_j+1}e^{(\glx_j-2\gl_i)n}
\notag\\&
= \bC n^{\kk_j+1-2\kk_i}e^{(\glx_j-2\glx_i) n}
.\end{align}
We have always $\glx_i\ge\glx_j$, and thus, in the present case, 
$\glx_j-2\glx_i\le -\glx_j<0$.
Hence \eqref{j54} yields
\begin{align}\label{j55}
\E\sumn \tZcl(n)^2&
=\sumn\E \tZcl(n)^2
<\infty
.\end{align}
Consequently, \as{} $\tZcl(n)\to0$ as \ntoo, which by \eqref{j4z} and
\eqref{j53}
means $\tZ_\ell(t)\to0$ as \ttoo.
Moreover, \eqref{j4z*} implies
\begin{align}\label{j56}
\bigpar{\tZZ_\ell}^2\le 
\sumn  \tZcl(n)^2,
\end{align}
and thus \eqref{j55} implies also $\norm{\tZZ_\ell}_2<\infty$.

\pfcasex{$\glx_j=\gl_i=0$}\label{step5-iii'}
In this case we have $\glx_i=0$ and $\kk_i=\kk_j+1$.  Let, for $n\ge1$,
\begin{align}\label{j3c1}
  \tZccl(n)
&:=\sup_{2^{n-1}\le t\le 2^n}t^{-\kk_i}e^{(\gl_i-\glx_i) t} |Z_\ell(t)|
=\sup_{2^{n-1}\le t\le 2^n}t^{-\kk_j-1} |Z_\ell(t)|
.\end{align}
Similarly to \eqref{j54}, 
it follows from  Doob's inequality and \eqref{j51} that
\begin{align}\label{j3c2}
\E \tZccl(n)^2&
\le \bC 2^{-2(\kk_j+1)n}\int_0^{2^n} (s+1)^{\kk_j}\dd s
\le \bC 2^{-(\kk_j+1)n}
.\end{align}
Hence,
\begin{align}\label{j38b}
\E\sumn \tZccl(n)^2&
=\sumn\E \tZccl(n)^2
<\infty
,\end{align}
and the rest of the argument is as in the preceding case,
now using
\begin{align}\label{j56b}
\bigpar{\tZZ_\ell}^2\le 
{Z_\ell^*(1)}^2 +
\sumn  \tZccl(n)^2
\end{align}
and noting that $\E {Z_\ell^*(1)}^2 <\infty$ by \eqref{nn6} and
\eqref{j51}. 

We have shown that \eqref{jtz} and \eqref{jtz*} hold for $\ell\le3$ in all
cases,
with $\cZ_\ell=0$ except in the case \ref{gli>}.

\stepx{Conclusion}\label{stepVI}
We have shown that \eqref{jtz}--\eqref{jtz*} hold for every $\ell$;
consequently, \eqref{ji1}--\eqref{ji2} hold by \eqref{jxz}--\eqref{jxz*}.

Moreover, in cases \ref{gli=} and \ref{gli<}, $\cZ_\ell=0$ for $\ell=1,2,3$,
and thus $\cX_i=a_jr_{ji}\cZ_4$; this 
yields \eqref{ji3} by  \eqref{j46=}, which in particular  shows
that $\cX_i>0$ when $\cX_j>0$ in these cases.

It remains to show $\cX_i>0$ when $\cX_j>0$
in the case \ref{gli>}, \ie, when $\gl_i>\glx_j$.
In this case $\glx_i=\gl_i$ and $\kk_i=0$, 
and thus by \eqref{j4} and \eqref{j3}, 
\begin{align}\label{j61}
\tX_i(t)=e^{-\gl_i t}X_i(t)
=\sumk e^{-\gl_i t}\indic{T_k\le t}Y_k(t-T_k)
.\end{align}
Moreover, if $\cX_j>0$, then $\liminf_\ttoo X_j(t)>0$ and thus \as{} there
is an infinite number of draws of colour $j$, and thus all $T_k$ are finite.
Let $K$ be the (random) smallest $k$ such that $\eta_k>0$; such $k$ exist
\as{} since 
$\E\eta_k=r_{ji}>0$ by assumption.
Then \eqref{j61} implies
\begin{align}\label{j62}
\tX_i(t)
\ge e^{-\gl_i T_K}\indic{t\ge T_K}e^{-\gl_i (t-T_K)}Y_K(t-T_K),
\end{align}
which 
\as{} has a strictly positive limit
by conditioning on $K$ and applying  \refL{LM}  to $Y_K$.
We have already shown that the limit $\cX_i$ in \eqref{ji1} exists \as, and
\eqref{j62} then shows $\cX_i>0$ a.s.
\end{proof}

If $\gl_i\le\glx_j$, then $\cX_i$ is determined by $\cX_j$, see
\eqref{ji3}. 
On the contrary, if $\gl_i>\glx_j$, then  $\cX_i$
is not determined by $\cX_j$, as shown in the following lemma.

\begin{lemma}\label{LND}
If\/ $\gl_i>\glx_j$
  in \refL{L12},
then the distribution of $\cX_i$
conditioned on $\cX_j$ is non-degenerate.
In fact, then
the conditional distribution $\cL(\cX_i\mid\cX_j)$ is \as{} absolutely
continuous. 
\end{lemma}

\begin{proof}
Since $\gl_i>\glx_j$, we have \eqref{j61}. Condition on the entire process
$(X_j(t))_{t\ge0}$ and on all $T_k$ and $\eta_k$;
then  the terms in the sum in \eqref{j61}
are independent, and, as \ttoo, each term converges \as{} 
by \refL{LM} to
a limit that by 
\refL{LM<} has an \abscont{} distribution 
when $\eta_k>0$.
Since \as{} $\eta_k>0$ for infinitely many $k$,
it follows that 
the limit is \as{} (conditionally) \abscont.
Since $(X_j(t))_t$ determines $\cX_j$, 
this implies (by \refL{LX})
that the distribution \as{} is \abscont{} also if we condition
only on $\cX_j$.
\end{proof}

\subsection{The general case for a single colour}\label{SSLG}

We have in the preceding subsections considered two special cases. We
consider now a single colour $i$ in a general triangular urn.
We continue to use the notations \eqref{j4}--\eqref{j4*}.
Recall also that $\sP_i:=\set{j\in\sQ:j\to i}$.

\begin{lemma}\label{LG}
Let $i\in\sQ$, and assume that for every $j\in\sP_i$,
we have
\begin{align}\label{g1}
 \tX_j(t) \asto \cX_j
\qquad \text{as }\ttoo,
\end{align}
for some random variable $\cX_j\ge0$, and
\begin{align}\label{g2}
  \bignorm{\tXX_j}_2 <\infty.
\end{align}
Then
\begin{align}\label{g3}
 \tX_i(t) \asto \cX_i \qquad \text{as }\ttoo,
\end{align}
for some $\cX_i\ge0$ and
\begin{align}\label{g4}
  \bignorm{\tXX_i}_2 <\infty.
\end{align}
Furthermore, \as,
if\/ $\cX_j>0$ for every $j\in\sP_i$,
then $\cX_i>0$.
\end{lemma}

\begin{proof}
If $\sP_i=\emptyset$, we are in the situation of \refL{LM}. Moreover, in this
case $X_i(0)>0$ by our standing assumption \ref{A3}.
Consequently, in this case the result follows from \refL{LM}.
(Note that $\glx_i=\gl_i$ and $\kk_i=0$ by \eqref{glx} and \eqref{kk}.)

In general,
  we separate the balls of colour $i$ according to their original reason for
existing. 
Formally, we split the colour $i$ and replace it by several
``subcolours'' (or shades); we define one subcolour labelled $i_0$, 
and an additional subcolour $i_j$ for each $j\in \sP_i$.
These subcolours have the same replacement vector $\bxi_i$ as $i$, with the
modification that new balls of colour $i$ always get the same subcolour as the
drawn ball. 
Also, balls of colour $i$ that are added when a ball of some 
other colour $j$ is drawn get the subcolour $i_j$.
Furthermore, all balls of colour $i$ at time 0 get subcolour $i_0$.
In other words, $i_0$ is used for
descendants of the balls with colour $i$ in the urn at the beginning, and
$i_j$ are used for balls of colour $i$ that eventually descend from a ball
of colour $j$. Note that $i_0$ is minimal, while $\sP_{i_j}=\set j$.

We thus have 
\begin{align}\label{ca1}
  X_i(t) = X_{i_0}(t)+ \sum_{j\in \sP_i} X_{i_j}(t).
\end{align}
Moreover, in the modified urn with subcolours instead of colour $i$,
the subcolour $i_0$ is of the type in \refSS{SS11} (possibly with $x_0=0$), and
each subcolour ${i_j}$ ($j\in \sP_i$)
is of the type in \refSS{SS12}.
Hence, \refLs{LM} and \ref{L12} apply, using our assumptions
\eqref{g1}--\eqref{g2}.
It follows from \eqref{glx} that
\begin{align}\label{gli0}
  \glx_{i_0}&=\gl_i,
\\\label{glij}
  \glx_{i_j}&=\gl_i\bmax\glx_j,\qquad j\in\sP_i.
\end{align}
In particular, 
for every $j\in\sP_i\cup\set0$, 
we have $\glx_{i_j}\le\glx_i$.
Furthermore, if $\sP_i\neq\emptyset$, then
\begin{align}\label{g5}
  \glx_i=\gl_i \bmax \max\set{\glx_j:j\in \sP_i}
=
\max\set{\glx_{i_j}:j\in \sP_i}.
\end{align}
Moreover, \eqref{kk} yields $\kk_{i_0}=0$, and implies also that if
$\glx_{i_j}=\glx_i$, then $\kk_{i_j}\le \kk_i$.
Hence, 
for every $j\in\sP_i\cup\set0$, 
and for all large $t$ (at least),
\begin{align}  \label{g55}
t^{-\kk_{i}}e^{-\glx_it} \le t^{-\kk_{i_j}}e^{-\glx_{i_j}t}.
\end{align}
Consequently, it follows from
\eqref{ca1} and the definition \eqref{j4} that,
at least for large $t$,
\begin{align}\label{g6}
  \tX_i(t) \le \tX_{i_0}(t)+\sum_{j\in \sP_i}\tX_{i_j}(t),
\end{align}
and furthermore, using also \eqref{lm1} and \eqref{ji1},
\begin{align}\label{g7}
\tX_i(t)\asto\cX_i:= 
\sum_{\nuj\in \sP_i\cup\set0}\cX_{i_\nuj}
  \indic{(\glx_{i_\nuj},\kk_{i_\nuj})=(\glx_i,\kk_i)}.  
\end{align}
This shows \eqref{g3}. 

Similarly, 
$(t+1)^{-\kk_{i}}e^{-\glx_it} \le C(t+1)^{-\kk_{i_j}}e^{-\glx_{i_j}t}$
for every $j\in\sP_i\cup\set0$ and every $t\ge0$,
and thus \eqref{ca1} and \eqref{j4*} imply
\begin{align}\label{g66}
  \tXX_i \le C\tXX_{i_0}+C\sum_{j\in \sP_i}\tXX_{i_j}.
\end{align}
Hence,
\eqref{g4} follows from \eqref{lm3a} and \eqref{ji2}.

Finally, assume $\cX_j>0$ for every $j\in \sP_i$.
If $\sP_i=\emptyset$, then, as remarked above, 
$\cX_i>0$ \as{} by \ref{A3} and \refL{LM}.
On the other hand,  if $\sP_i\neq\emptyset$, 
then $\cX_{i_j}>0$ for every $j\in\sP_i$ by \refL{L12}.
By \eqref{g5} and \eqref{kk}, there exists some $j\in \sP_i$ such that
$(\glx_{i_j},\kk_{i_j})=(\glx_i,\kk_i)$, and thus \eqref{g7} implies
\begin{align}
  \cX_i\ge \cX_{i_j}>0
\end{align}
a.s., which completes the proof.
\end{proof}

\begin{remark}\label{Rg7}
  If $\gl_i>\glx_j$ for every $j\in\sP_i$, then 
$\glx_{i_j}=\gl_i=\glx_i$ and $\kk_{i_j}=0=\kk_i$
for every $j\in \sP_i\cup\set0$, and thus
  \eqref{g7} yields 
\begin{align}\label{gg7}
\cX_i= \sum_{\nuj\in \sP_i\cup\set0}\cX_{i_\nuj}.
\end{align}

On the other hand, suppose that
$\gl_i\le \glx_j$ for some $j\in\sP_i$. Then either
$\glx_i>\gl_i=\glx_{i_0}$, or $\gl_i=\glx_i=\glx_j$ for some $j\in \sP_i$, 
and in the
latter case $\kk_i\ge 1+\kk_j>\kk_{i_0}=0$. Hence,  in both cases,
$(\glx_{i_0},\kk_{i_0})\neq(\glx_i,\kk_i)$, and thus
the sum in
\eqref{g7} is really only over (some) $\nuj\in \sP_i$.
\end{remark}

\section{The main theorem for continuous time}\label{STC}
\begin{theorem}\label{TC}
Let $(X_{i}(t))_{i\in\sQ}$ be a continuous-time
triangular \Polya{} urn satisfying the 
conditions \refAA. 
Then, for every colour $i\in\sQ$,
\begin{align}\label{tc1}
  t^{-\kk_i}e^{-\glx_i t} X_i(t) \asto \cX_i
\qquad \text{as }\ttoo,
\end{align}
for some random variable $\cX_i$ with $\cX_i>0$ a.s.
\end{theorem}

\begin{proof}
  We may choose a total order $<$ of the colours such that \eqref{q2} holds.
Taking the colours in this order, we see by \refL{LG} and induction that
\eqref{g3} and \eqref{g4} hold for every $i\in\sQ$, with  $\cX_i>0$ a.s.
Recalling \eqref{j4}, we see that \eqref{g3} is the same as \eqref{tc1}.
\end{proof}

The limits $\cX_i$ are strictly positive.
We will see in \refT{Tac} that they are
non-degenerate when $\glx_i>0$;
however, in general some of them will be
linear combinations of others.
We discuss dependencies between the limits in \refS{Sdep}.

\section{Proof of the main theorem for discrete time}
\label{SpfT1}

\begin{proof}[Proof of \refT{T1}]
We modify the urn by adding a new colour, 0 say, to $\sQ$;
let $\sQp:=\sQ\cup\set0$ be the new set of colours.
All activities $a_i$ and replacements
$\xi_{ij}$ with $i,j\in\sQ$ remain unchanged.
Balls of colour 0 have activity $a_0:=0$, and thus they are never drawn; in
accordance with \ref{A00} we define  $\xi_{0j}:=0$ for every $j\in\sQp$.
We further let $\xi_{i0}:=1$ for every $i\in\sQ$ with $a_i>0$. 
For $i\in\sQ$ with $a_i=0$ we let $\xi_{i0}:=0$, again in accordance with
\ref{A00}.
We let $X_{00}=X_0(0):=0$, so there are initially no balls of colour 0.
Note that the extended urn too satisfies \refAA.

Since balls of colour 0 never are drawn, we may ignore them and recover the
original urn. In other words,
the new urn will be the original urn with ``dummy balls'' of colour 0
added.
We may assume that the old and new urn are coupled in
this way, which means that for $i\in\sQ$, the number of balls of colour $i$
at any time  is the same in both urns, so we may unambiguously use
the notations $X_{ni}$ and $X_i(t)$ for both urns. 
Moreover, there are  initially no balls of colour $0$, but we add exactly 1
ball of colour 0 every time a ball is drawn. Hence,
in the discrete-time version, the number of dummy balls at time
$n$ equals $n$; thus $X_{n0}=n$.  
In the continuous-time model, $X_0(t)$ equals the number of
draws up to time $t$ , and in particular, cf.\ \eqref{ct},
\begin{align}\label{hh2}
  X_0(\TT_n)=X_{n0}=n.
\end{align}

The new urn is obviously also triangular, with $i\to0$ for every $i\in\sQ$
with $a_i>0$.
We have $\gl_0=0$.
Since every $\gl_i\ge0$, and we have $i\to 0$ when $\gl_i>0$ 
(and thus $a_i>0$ by \ref{A00}), it follows from \eqref{glx} and
\eqref{hgl} that
\begin{align}
  \label{hh1}
\glx_0=\max\set{\gl_j:j\in\sQ} = \hgl.
  \end{align}
If $\hgl>0$, then 
is further easily seen from \eqref{kk} and \eqref{hkk} that
\begin{align}\label{kk01}
\kk_0=
\max\set{\kk_i:i\in\sQ \text{ and }\glx_i=\hgl}
=\hkk.
\end{align}
If $\hgl=0$, which means that $\gl_i=0$ for every $i$, we have by
\eqref{kk} and \eqref{hkk0} instead
\begin{align}\label{kk0}
  \kk_0&=\max\bigset{\kk:\exists i_1\prec i_2\prec \dots\prec i_{\kk+1}=0 
\text{ in }\sQp}
\notag\\&
=
\max\bigset{\kk:\exists i_1\prec i_2\prec \dots\prec i_{\kk}
\text{ with }a_{i_\kk}>0
\text{ in }\sQ}
\notag\\&
=
1+\max\bigset{\kk_j: j\in\sQ \text{ with }a_{j}>0}
\notag\\&
=\hkko
.\end{align}

We apply \refT{TC} to the new urn.
In particular, 
taking $i=0$ in \eqref{tc1} yields,
using \eqref{hh2} and \eqref{hh1}--\eqref{kk0}
and  recalling that $\TT_n\to\infty$ a.s., 
\begin{align}\label{hh3}
  \TT_n^{-\kk_0}e^{-\hgl \TT_n} n 
=  \TT_n^{-\kk_0}e^{-\hgl \TT_n} X_0(\TT_n) 
\asto \cX_0,
\qquad \text{as }\ntoo.
\end{align}
Recall that $0<\cX_0<\infty$ a.s.

\pfCaseY1{$\hgl>0$}
Consider first the case $\hgl>0$; then $\kk_0=\hkk$ by \eqref{kk01}.
It follows from \eqref{hh3} that \as, as \ntoo,
\begin{align}\label{hh4}
\log n-\hgl \TT_n-\hkk \log \TT_n
\to \log\cX_0
\end{align}
which yields, since $\TT_n\to\infty$,
\begin{align}\label{hh5}
\frac{\log n}{\TT_n} &\to \hgl,
\end{align}
and thus
\begin{align}\label{hh5b}
   \log\log n - \log \TT_n &\to \log\hgl.
\end{align}
Using \eqref{hh5b} in \eqref{hh4} yields, \as,
\begin{align}\label{hh6}
  \TT_n 
= \frac{1}{\hgl}\bigpar{\log n-\hkk\log\log n+\hkk\log\hgl-\log\cX_0+o(1)}
.\end{align}
Now let $i\in\sQ$. Then \eqref{hh6} yields, as \ntoo, \as,
\begin{align}\label{hh7}
  \TT_n^{\kk_i}e^{\glx_i \TT_n} 
\sim \hgl^{-\kk_i+\hkk\glx_i/\hgl}
\cX_0^{-\glx_i/\hgl}n^{\glx_i/\hgl} (\log n)^{\kk_i-\hkk\glx_i/\hgl} 
.\end{align}
Consequently, taking $t=\TT_n$ in \eqref{tc1} yields, as \ntoo, 
using the notation \eqref{gam},
\begin{align}\label{hh8}
  \frac{X_{ni}}{ n^{\glx_i/\hgl} \log^{\gam_i}n}
= \frac{X_{i}(\TT_n)}{ n^{\glx_i/\hgl} (\log n)^{\kk_i-\hkk\glx_i/\hgl}}
\asto \hcX_i
\end{align}
with
\begin{align}\label{hcx}
  \hcX_i:= \hgl^{-\kk_i+\hkk\glx_i/\hgl}\cX_0^{-\glx_i/\hgl}\cX_i
= \hgl^{-\gam_i}\cX_0^{-\glx_i/\hgl}\cX_i
.\end{align}
This shows \eqref{t1b}.

\pfCaseY2{$\hgl=0$}
In the case $\hgl=0$, we have instead $\kk_0=\hkko$ by \eqref{kk0}.
Moreover,  \eqref{hh3} now yields that as \ntoo, \as,
\begin{align}\label{hh10}
  \TT_n \sim \cX_0^{-1/\hkko}n^{1/\hkko}
\end{align}
and thus \eqref{tc1} yields
\begin{align}\label{hh11}
  \frac{X_{ni}}{n^{\kk_i/\hkko}}
= \frac{X_{i}(\TT_n)}{n^{\kk_i/\hkko}}
\asto \hcX_i
:= \cX_0^{-\kk_i/\hkko}\cX_i.
\end{align}
This shows \eqref{t1c}, and completes the proof of \refT{T1}.
\end{proof}

The limits $\hcX_i$ are all strictly positive.
We return to the question whether they are degenerate in \refS{Sdeg}.

\begin{remark}\label{Rel}
  Suppose that $\hgl>0$, so that \eqref{hh8} holds. As a sanity check, we
  note that \eqref{hh8} and \eqref{hcx} imply, 
recalling \eqref{hgl} and \eqref{hkk},
that 
if we define
\begin{align}\label{el2}
\sQx:=\set{i\in\sQ:\glx_i=\hgl \text{ and } \kk_i=\hkk},
\end{align}
which is nonempty,
then
  \begin{align}\label{el1}
    \frac{X_{ni}}n\asto
    \begin{cases}
      \hcX_i = \cX_i/\cX_0 & \text{if } i\in\sQx,
\\
0 &\text {otherwise}.
    \end{cases}
  \end{align}
Hence, the total number of balls grows linearly, as should be expected;
moreover, the distribution of colours in the urn is asymptotically
concentrated on $\sQx$.

Note also that \eqref{g7} 
for the dummy colour 0 yields,  
recalling \eqref{hh1}--\eqref{kk01} and using \eqref{ji3}, 
\begin{align}\label{el3}
  \cX_0 = \sum_{j\in\sQx}\cX_{0_j}
=\sum_{j\in\sQx}\frac{a_j}{\glx_j}\cX_j
=\hgl\qw \sum_{j\in\sQx}a_j\cX_j
.\end{align}
By \eqref{el1} and \eqref{el3}, we have for the total activity in the urn  
\begin{align}\label{el4}
  \frac{1}n\sum_{i\in\sQ}a_iX_{ni} 
\asto  \frac{1}{\cX_0}\sum_{i\in\sQx}a_i\cX_i
=\hgl.
\end{align}
\end{remark}

\begin{remark}\label{REL0}
If we instead suppose $\hgl=0$, which implies $\gl_i=\glx_i=0$ for every $i$,
then \eqref{hh11} similarly shows that 
if we now define
  \begin{align}\label{el02}
\sQx:=\set{i: \kk_i=\hkko},
  \end{align}
then \eqref{el1} still holds.
Note that it is quite possible that $\sQx$ is empty; 
this happens precisely when $\hkk_0>\hkk$, which by \eqref{kk0} happens when
some $i$ with $\kk_i=\hkk$ has activity $a_i>0$.
(Such a colour cannot have any descendants, since all colours have the same
$\glx_j=0$; hence   we must have $\bxi_i=0$ a.s., which means that such colours
will be drawn, but nothing happens to the urn at these draws.)
In this case the total
number of balls is \as{} $o(n)$, and the colour distribution is
asymptotically concentrated on the colours $i$ such that
$\kk_i=\hkk=\hkko-1$.

On the other hand, if $\sQx\neq\emptyset$, which means that $\hkko=\hkk$, 
and thus by \eqref{kk0} $a_i=0$ for every $i$ with $\kk_i=\hkk$,
let $\sQxm:=\set{j:\kk_j=\hkk-1}$.
Then \eqref{g7} and \eqref{ji3} yield, 
since $j\to0$ if and only if $a_j>0$, and then $\kk_{0_j}=\kk_j+1$,
\begin{align}\label{el03}
  \cX_0 = \sum_{j\in\sP_0\cap\sQxm}\cX_{0_j}
=\sum_{j\in\sQxm}\frac{a_j}{\hkk}\cX_j
=\hkk\qw \sum_{j\in\sQxm}a_j\cX_j.
\end{align}
\end{remark}

\section{Dependencies between the limits}\label{Sdep}
The limits $\cX_i$ in \refT{TC} are non-degenerate except in extreme cases, 
as shown below,
but there are frequently linear dependencies between them.
To explore this, we introduce more terminology.

We say that a colour $i$ is a \emph{leader} if $\gl_j<\gl_i$ for every
$j\prec i$,
and a \emph{follower} otherwise.
(In particular, a minimal colour $i$ is a leader.)
We have, recalling \eqref{glx},
\begin{align}\label{lead}
i\text{ is a leader}&
\iff \Bigpar{j\prec i \implies \gl_j<\gl_i}
  \iff \Bigpar{j\to i \implies \glx_j <\gl_i},
\\\label{follow}
i\text{ is a follower}&
\iff \Bigpar{\exists j\prec i \text{ with } \gl_j\ge\gl_i}
  \iff \Bigpar{\exists j\to i \text{ with } \glx_j \ge\gl_i}.
\end{align}
By \eqref{glx}--\eqref{kk}, 
\begin{align}\label{lead2}
i \text{ is a leader} 
\iff 
\glx_i=\gl_i \text{ and } \kk_i=0
.\end{align}

If $i$ is a follower,
let $\sA_i$ be the set
of ancestors of $i$ that have maximal $\gl_j$, 
\ie, $\gl_j=\glx_i$,
and such that furthermore there is a path from $j$ to $i$ with the
maximum number $\kk_i+1$ of colours $\ell$ with $\gl_\ell=\glx_i$.
(Recall \eqref{glx} and \eqref{kk}.) Thus
\begin{align}\label{sA}
  \sA_i:=\set{j\prec i: 
\exists i_1=j\prec i_2\prec \dots\prec i_{\kk_i+1}\preceq i
\text{ with } \gl_{i_1}=\dots=\gl_{i_{\kk+1}}=\glx_i}
.\end{align}
Note that $\sA_i$ is non-empty for every follower $i$.
Moreover, it is easily seen from \eqref{lead}--\eqref{follow} and \eqref{sA}
that every $j\in\sA_i$ is a leader.
We may say that $i$ follows the leaders in $\sA_i$; note that a follower may
follow several leaders.
For completeness, we define $\sA_i:=\set i$ when $i$ is a leader.
Note that, by \eqref{lead2} and \eqref{sA}, 
\begin{align}\label{glai}
\gl_j=  \glx_j=\glx_i
\qquad \text{for every $j\in\sA_i$}.
\end{align}

We show first that the variables $\cX_i$ are determined by the ones for
leaders $i$.
\begin{lemma}\label{LNDB}
  If\/ $i\in \sQ$,
then $\cX_i$ is a linear combination
\begin{align}\label{lndb}
  \cX_i = \sum_{k\in \sA_i} c_{ik}\cX_k
\end{align}
with strictly positive coefficients $c_{ik}$.
\end{lemma}
\begin{proof}
  Note first that \eqref{lndb} is trivial when $i$ is a leader, with
$c_{ii}=1$.
We may thus suppose that $i$ is a follower.

By induction on the colour $i$, we may assume that the formula
\eqref{lndb} holds for every colour that is an ancestor of $i$, 
and in particular for every  $\nuj\in\sP_i$.

By \refR{Rg7}, we only have to consider $\nuj\in \sP_i$ in \eqref{g7}.
Let 
\begin{align}\label{oc3}
\sP_i':=\bigset{ j\in \sP_i:(\glx_{i_\nuj},\kk_{i_\nuj})=(\glx_i,\kk_i)},   
\end{align}
so that the sum in \eqref{g7} really is over $j\in \sP_i'$.
By  \eqref{lead2}, since we now assume that $i$ is a follower,
either $\glx_i>\gl_i$ or $\kk_i>0$ (or both).
Hence,
if $j\in \sP_i'$, then either $\glx_{i_j}=\glx_i>\gl_i$ or $\kk_{i_j}=\kk_i>0$;
since $\glx_{i_j}=\gl_i\lor \glx_j$,
it follows  in both cases (using \eqref{kk} in the latter)
that $\glx_j=\glx_{i_j}\ge\gl_i$ and thus
 by \eqref{ji3}
\begin{align}\label{lndb2}
\cX_{i_j}=c'_{ij} \cX_j  
\end{align}
for some constant $c'_{ij}>0$.
Moreover, we have 
$\glx_j=\glx_j\bmax \gl_i=\glx_{i_j}=\glx_i$
and it follows from \eqref{oc3} and \eqref{kk} that
\begin{align}
  \label{oc4}
\kk_i=\kk_{i_j}=\kk_j+\indic{\gl_i=\glx_i}.
\end{align}
 If $j\in\sP'_i$ is a leader, so $\gl_j=\glx_j=\glx_i$ and $\kk_j=0$ by
 \eqref{lead2}, 
it follows from \eqref{sA} and \eqref{oc4} that  $\nuj\in\sA_i$. 
Similarly, if $\nuj\in\sP'_i$ is a follower, then it follows 
from \eqref{sA} and \eqref{oc4} that 
if $k\in\sA_\nuj$, then a chain from $k$ to $j$ of the (maximal) type 
in \eqref{sA} can be
extended to a chain from $k$ to $i$ of the same type, and thus $k\in\sA_i$;
consequently, if $\nuj\in\sP'_i$ is a follower, then $\sA_\nuj\subseteq \sA_i$. 
The result \eqref{lndb}, with some $c_{ij}\ge0$, now follows by \eqref{g7},
  \eqref{lndb2}, and induction.

Finally, if $k\in \sA_i$, 
let $i_1=k, \dots,i_{\kk_i+1}\preceq i$ be as in \eqref{sA}, choose a path
in $\sQ$ from $k$ to $i$ that contains all $i_\ell$, and let $j$ be the last
colour in this path before $i$. 
Then it follows from \eqref{sA} that $k\in\sA_j$, and thus by induction
$c_{jk}>0$, which implies $c_{ik}\ge c'_{ij}c_{jk}>0$.
\end{proof}

Motivated by \eqref{lndb}, we turn to considering $\cX_i$ for leaders $i$.
We first consider a trivial exceptional case.
\begin{lemma}\label{LL0}
  Let $i$ be a colour with $\gl_i=0$. Then $i$ is a leader if and only if
  $i$ is minimal, \ie, there is no colour $j$ with $j\to i$.
In this case $\cX_i=X_i(0)$ is a deterministic positive constant.
\end{lemma}
\begin{proof}
We have already remarked that a minimal colour is a leader.   
Conversely, 
if $i$ is a leader with $\gl_i=0$, then \eqref{lead} shows that $j\to i$
is impossible; hence $i$ is minimal.
In this case, no balls of colour $i$ are added by draws of balls
  of other colours. Furthermore,
since $\gl_i=0$, also no balls of colour $i$ are added by
  draws of colour $i$. Consequently, $X_i(t)=X_i(0)$ for all $t\ge0$,
and \eqref{g1} holds trivially with $\tX_i(t)=X_i(t)$ and $\cX_i=X_i(0)$.
\end{proof}

We next show that except for the case in \refL{LL0}, the distribution of
$\cX_i$ for a leader $i$ is absolutely continuous, and thus non-degenerate.

\begin{lemma}\label{LNDA}
  \begin{thmenumerate}
    
  \item \label{LNDA1}
  Let $i$ be a leader with $\gl_i>0$, and let $\sE$ be a set of colours 
  that contains neither $i$ nor any descendant of $i$, \ie,
if $j\in \sE$ then $j\not\succeq i$. Then 
the conditional distribution $\cL\xpar{\cX_i\mid \cX_j, j\in\sE}$ is
\as{} \abscont.
In particular, the distribution of $\cX_i$ is \abscont.

\item \label{LNDA2}
More generally, let $\sL$ be a (non-empty) set of leaders with $\gl_{i}>0$ for
every $i\in\sL$, and 
suppose that $\sL$ is an anti-chain in $\sQ$, \ie,  $i\not\prec j$ when 
$i,j\in\sL$.
Let further $\sE$ be a set of colours such that if $j\in\sE$,
then $j\not\succeq i$ for every $i\in\sL$.
Then the joint conditional distribution 
$\cL\xpar{(\cX_{i})_{i\in\sL}\mid  \cX_j, j\in\sE}$ is \as{} \abscont.
(This is a distribution in $\bbR^{|\sL|}$.)
  \end{thmenumerate}
\end{lemma}
\begin{proof}
It suffices to consider \ref{LNDA2}. 
The conclusion 
(conditional absolute continuity) is preserved if we reduce $\sE$ to a
smaller set, see \refL{LX} and \refR{RLX}. 
Thus  we may assume that $\sE$ is maximal, \ie, 
\begin{align}\label{oc68}
\sE=\set{j\in\sQ:  j\not\succeq i, \forall i\in\sL}.
\end{align}
Note that, 
if $i\in\sL$ and $j\prec i$, then we cannot have $j\succeq k$ for any
$k\in\sL$, since this would imply $k\prec i$, contradicting the assumption
that $\sL$ is an antichain.
Consequently, by \eqref{oc68}, if $i\in\sL$ and $j\prec i$, then $j\in\sE$.

We argue as in the proof of \refL{LND}, and condition on the entire
processes $(X_j(t))_{t\ge0}$, $j\in\sE$, and also on the times of all
draws of a color $j\in\sE$, and on all replacement vectors $\bxi_j\nn$
($n\ge1$,  $j\in\sE$).

Let  $i\in\sL$. If $j\in\sP_i$, so $j\to i$, then $\glx_j<\gl_i=\glx_i$ by the
assumption that $i$ is a leader.
We use again the decomposition \eqref{ca1}, and note that each 
$e^{-\gl_i t}X_{i_j}(t)$
($j\in \sP_i$) can be written as a sum as in \eqref{j61} with (conditionally)
independent terms, and it follows as in the proof of \refL{LND} that each
$\cX_{i_j}$ has (conditionally) an \abscont{} distribution.
Furthermore, $\cX_{i_0}$ is independent of conditioning on colours $j\in
\sE$, and if $X_i(0)>0$, then its distribution too is \abscont{} by
\refL{LM<}
(since we assume $\gl_i>0$).
Moreover (still conditionally), all $\cX_{i_j}$ and $\cX_{i_0}$ ($i\in\sL$,
$j\in \sP_i$) are independent. 

Since $i$ is a leader,
\refR{Rg7} shows that \eqref{gg7} holds. We have shown that (conditionally)
the  terms in the sum in \eqref{gg7} are independent, and all are \abscont{}
except $\cX_{i_0}$ when $X_i(0)=0$. 
Thus our assumption \ref{A3} implies that the sum contains at least one
\abscont{} summand,
and hence $\cX_i$ is (conditionally) \abscont.

Moreover, this argument shows that $\cX_i$ for different $i\in\sL$ are
(conditionally) independent, and since each has an \abscont{} distribution,
their joint distribution is (conditionally) \abscont.

Finally, \refL{LX} (again with \refR{RLX}) shows that
the same holds if we instead condition on 
$\cX_j$, $j\in\sE$.
\end{proof}

\begin{theorem}\label{Tac}
  Let $\sL$ be a set of leaders such that $\gl_i>0$ for every $i\in\sL$.
Then the joint distribution $\cL\bigpar{(\cX_i)_{i\in\sL}}$
is \abscont.
\end{theorem}
\begin{proof}
  Order the elements of the set $\set{\gl_i:i\in\sL}$ as
  $\gl_{(1)}<\dots<\gl_{(r)}$. Let  
  \begin{align}
\sL_\ell&:=\set{i\in\sL:\gl_i=\gl_{(\ell)}}, \qquad 1\le \ell\le r,    
\\
\sL_{\le\ell}&:=\bigcup_{k:=1}^\ell\sL_k, \qquad 0\le \ell\le r.
  \end{align}
(Thus, $\sL_{\le0}=\emptyset$.)
Let $\ell\in\setr$.
If $i,j\in\sL_\ell$, then \eqref{lead} shows that $j\not\prec i$, and thus
$\sL_\ell$ is an antichain.
Furthermore, if $i\in\sL_\ell$ and $j\in\sL_{\le\ell-1}$, then
$\glx_j=\gl_j<\gl_i$, and thus $i\not\preceq j$.
Consequently, \refL{LNDA}\ref{LNDA2} shows that the conditional distribution
of $(\cX_i)_{i\in\sL_\ell}$ given $(\cX_j)_{j\in\sL_{\le\ell-1}}$ is a.s.\ \abscont.

It now follows from \refL{LY} by induction that the distribution of
$(\cX_i)_{i\in\sL_{\le\ell}}$ is \abscont{} for $\ell=1,\dots,r$.
Taking $\ell=r$ yields the result, since $\sL_r=\sL$.
\end{proof}

\begin{theorem}\label{TD}
Let $i\in\sQ$ be any colour.
\begin{romenumerate}
  
\item \label{TD1}
If $\glx_i=0$, then $\cX_i$ is a deterministic positive constant.
\item \label{TD2}
If $\glx_i>0$, then $\cX_i$ has an \abscont{} distribution.
\end{romenumerate}
\end{theorem}

\begin{proof}
  \pfitemref{TD1}
Every $k\in\sA_i$ is a leader with $\gl_k=\glx_i=0$. Thus \refL{LL0} shows
that $\cX_k$ is deterministic for $k\in\sA_i$. Consequently, \eqref{lndb} 
shows that $\cX_i$ too is deterministic.
\pfitemref{TD2}
Every $k\in\sA_i$ is a leader with $\gl_k=\glx_i>0$. Thus \refT{Tac} shows
that the joint distribution of $(\cX_k)_{k\in\sA_i}$ is \abscont.
It follows from \eqref{lndb} that the distribution of $\cX_i$ is \abscont.
(See \refL{LZ}.)
\end{proof}

\begin{corollary}
If $\glx_i>0$, then
  the coefficients $c_{ik}$ in \refL{LNDB} are uniquely determined.
\end{corollary}

\begin{proof}
Suppose, on the contrary, that \eqref{lndb} holds a.s.\ for two different
sets of coefficients $(c_{ik})_k$  and $(c_{ik}')_k$. 
Let $b_k:=c_{ik}-c'_{ik}$.
Then $\sum_{k\in\sA_i} b_k \cX_k=0$ \as, and thus $(\cX_k)_{k\in\sA_i}$
\as{} lies in the certain hyperplane orthogonal to $(b_k)$.
However, this contradicts \refT{TD} which says that the distribution of
$(\cX_k)_{k\in\sA_i}$ is \abscont.
(See also \refL{LZ}.)
This contradiction proves the claim.
\end{proof}

The coefficients $c_{ik}$ in \eqref{lndb} can be found by the recursive
procedure in the proof of \refL{LNDB}.
We proceed to show that they also can be found as eigenvectors of suitable
submatrices of
the weighted mean replacement matrix $(a_ir_{ij})_{i,j\in \sQ}$.
In the sequel, we let $c_{ik}$ be the coefficient given by the inductive
proof of \refL{LNDB}; this determines $c_{ik}$ uniquely for $k\in\sA_i$
also when $\glx_i=0$.
We further define $c_{ik}:=0$ if $k\notin \sA_i$.

Let $\nu$ be a leader, and let 
\begin{align}\label{sD}
\sD_\nu&:=\set{i:\glx_i=\gl_\nu},
\\\label{sDk}
  \sD_\nu^\kk&:=\set{i\in\sD_\nu:\kk_i=\kk},
\qquad \kk=0,1,\dots.
\end{align}
(The sets $\sD_\nu^\kk$ are empty from some $\kk$ on; we only consider $\kk$ with
$\sD_\nu^\kk\neq\emptyset$. Note also that $\sD_\nu=\sD_{\nu'}$ if $\nu'$ is
another leader with $\gl_{\nu'}=\gl_\nu$.)
Note that if $c_{i\nu}>0$, then $\nu\in\sA_i$, and thus $i\in\sD_\nu$.



Let $i\in\sD_\nu^\kk$, and suppose first that $\gl_i<\glx_i=\gl_\nu$. 
Then it
follows from \refR{Rg7} that in \eqref{g7}, we only have to sum over
$j\in\sP_i$. Moreover, if $j\in\sP_i$ and $\glx_{i_j}=\glx_i$, then
$\glx_j=\glx_{i_j}=\glx_i$ and $\kk_{i_j}=\kk_j$. Hence, in \eqref{g7} we
only have to sum over $j\in\sP_i\cap\sD_\nu^\kk$.
Since $j\in\sP_i\iff j\to i\iff r_{ji}>0$ and $j\neq i$ by \eqref{itoj},
it follows from \eqref{g7}
and \eqref{ji3}
that
\begin{align}\label{eq2}
  \cX_i = \sum_{j\in\sD_\nu^\kk\cap\sP_i}\cX_{i_j}
= \sum_{j\in\sD_\nu^\kk\setminus\set i}\frac{a_jr_{ji}}{\glx_j-\gl_i}\cX_{j}
.\end{align}
Consequently, the coefficient $c_{i\nu}$ in \eqref{lndb} is given by
\begin{align}\label{eq2c}
  c_{i\nu}
= \sum_{j\in\sD_\nu^\kk\setminus\set i}\frac{a_jr_{ji}}{\glx_j-\gl_i}c_{j\nu}
.\end{align}

Since $j\in\sD_\nu^\kk$ implies $\glx_j=\gl_\nu$, \eqref{eq2c} yields
\begin{align}\label{eq3}
(\gl_\nu-\gl_i) c_{i\nu}
= \sum_{j\in\sD_\nu^\kk\setminus\set i}{a_jr_{ji}}c_{j\nu}
,\end{align}
for every $i\in\sD_\nu^\kk$ with $\gl_i<\glx_i$.

On the other hand, suppose that $i\in\sD_\nu^\kk$ with $\gl_i=\glx_i=\gl_\nu$.
If $j\in\sD_\nu\setminus\set i$ and $r_{ji}>0$, then $\glx_j=\gl_\nu=\glx_i$
and $j\to i$, 
and thus $\kk_i\ge \kk_j+1$, since a maximal chain in the definition
\eqref{kk} of $\kk_j$ can be extended by $i$. Hence, if $j\in\sD_\nu^\kk
\setminus\set i$, then $r_{ji}=0$. 
It follows that \eqref{eq3} holds trivially
for $i\in\sD_\nu^\kk$ with $\gl_i=\glx_i$.

Consequently, \eqref{eq3} holds for every $i\in\sD_\nu^\kk$, and thus,
recalling $\gl_i=a_ir_{ii}$,
\begin{align}\label{eq4}
\sum_{j\in\sD_\nu^\kk}{a_jr_{ji}}c_{j\nu}
=(\gl_\nu-\gl_i)c_{i\nu}+a_ir_{ii}c_{i\nu} = \gl_\nu c_{i\nu},
\qquad i\in\sD_\nu^\kk
.\end{align}

We summarize, and elaborate, this as a lemma.
We say that $i\in\sQ$ is a \emph{subleader} if $\gl_i=\glx_i$ but $i$ is not
a leader;
recalling \eqref{lead2} we thus have
\begin{align}\label{sublead}
i \text{ is a subleader} 
\iff 
\glx_i=\gl_i \text{ and } \kk_i\ge1
.\end{align}

\begin{lemma}\label{LC}
  For any leader $\nu$, and any $\kk\ge0$ such that
  $D_\nu^\kk\neq\emptyset$, 
we have \eqref{eq4}, and thus
$(c_{i\nu})_{i\in\sD_\nu^\kk}$ is a left eigenvector of the
triangular matrix 
$(a_ir_{ij})_{i,j\in\sD_\nu^\kk}$ for its largest eigenvalue $\gl_\nu$.

For $\kk=0$,
the value $c_{i\nu}$ for a leader $i\in\sD_\nu^0$  is determined by
$c_{i\nu}=\gd_{i\nu}$ (i.e., $1$ when $i=\nu$ and $0$ otherwise),
and these values for the
leaders determine the eigenvector $(c_{i\nu})_{i\in\sD_\nu^\kk}$ uniquely.

For $\kk\ge1$, 
the value $c_{i\nu}$ for a subleader $i\in\sD_\nu^\kk$ is determined
recursively from the values $c_{j\nu}$ for $j\in\sD_\nu^{\kk-1}$ by
\begin{align}\label{subl}
  c_{i\nu}=\sum_{j\in\sD_\nu^{\kk-1}}\frac{a_jr_{ji}}{\kk}c_{j\nu},
\end{align}
and these values for the
subleaders determine the eigenvector $(c_{i\nu})_{i\in\sD_\nu^\kk}$ uniquely.
\end{lemma}

\begin{proof}
  We have shown that \eqref{eq4} holds, and thus
  $(c_{i\nu})_{i\in\sD_\nu^\kk}$ 
is a left eigenvector.

Since the matrix $A_\kk:=(a_jr_{ij})_{i,j\in\sD_\nu^\kk}$ is triangular, its
eigenvalues are its diagonal elements $a_jr_{jj}=\gl_j$, $j\in\sDnk$.
The definition \eqref{kk} of $\kk_i$ implies that if $\sDnk\neq\emptyset$,
then there exists $i\in\sDnk$ with $\gl_i=\glx_i$, i.e.\ a leader (if
$\kk=0$) or a subleader (if $\kk\ge1$). Since $\gl_i\le\glx_i=\gl_\nu$ for every
$i\in\sDnk$, it follows that $\glx_\nu$ is the largest eigenvalue, and that
its multiplicity equals the number of leaders or subleaders in $\sDnk$.

Let $\sLk$ be the set of leaders or subleaders in $\sDnk$, and consider
the projection $\Pi_\kk:\bbR^{{\sDnk}}\to\bbR^{\sLk}$ mapping
$(x_i)_{i\in\sDnk}\mapsto(x_i)_{i\in\sLk}$. Let $V_\kk$ be the left eigenspace of 
the matrix $A_\kk$ for its largest eigenvalue $\gl_\nu$. 
The recursion in the proof of \refL{LNDB} and the calculations
\eqref{eq2}--\eqref{eq3} show, more generally, that any vector  
$(x_i)_{i\in\sLk}$ can be extended to a vector 
$(x_i)_{i\in\sDnk}$ that belongs to $V_\kk$. In other words, the projection
$\Pi_\kk$ maps $V_\kk$ onto $\bbR^{\sLk}$.  These spaces
have the same dimension $|\sLk|$, and thus
$\Pi_\kk:V_\kk\to\bbR^{\sLk}$ is a bijection.
Consequently, an eigenvector is determined by its values for leaders or
subleaders. 

It is trivial that $c_{i\nu}=\gd_{i\nu}$ for a leader $i$.

For $\kk\ge1$, let $i$ be a subleader in $\sDnk$, so $\gl_i=\glx_i$.
Since $\kk_i=\kk>0$, it follows from \refR{Rg7} that $\gl_i=\glx_j$ for some
$j\in\sP_i$, and furthermore that $\kk_j+1\le \kk_i$ for every
such $j$. 
Since then also $\glx_{i_j}=\glx_j=\glx_i$ and $\kk_{i_j}=\kk_j+1$, it
follows that 
in \eqref{g7} we only have to sum over
$j\in\sD_\nu^{\kk-1}$, which together with \eqref{ji3} yields \eqref{subl}.
\end{proof}

\begin{remark}
It is well known that
the eigenvalues and eigenvectors of the
weighted mean replacement matrix $(a_ir_{ij})_{i,j\in \sQ}$
are important for the asymptotics of \Polya{} urns in general;
see for example 
\cite[Section V.9.3]{AN} and
\cite{SJ154} 
for irreducible urns.
\refL{LC}, where we consider eigenvectors of certain submatrices, is inspired
by a special case in \cite{BoseDM}, see \refE{EBose}.
\end{remark}

\section{Degenerate limits in discrete time}\label{Sdeg}
Consider now the limits $\hcX_i$ in \refT{T1} for the discrete-time urn.

\begin{theorem}\label{TEL}
Let $i\in\sQ$.
  \begin{romenumerate}
  \item\label{TEL0}     
If\/ $\glx_i=0$, then $\hcX_i$ is a positive constant.
  \item\label{TEL<}     
If\/ $0<\glx_i<\hgl$, then $\hcX_i$ has an \abscont{} distribution.
  \item\label{TEL=}     
If\/ $\glx_i=\hgl>0$, then $\hcX_i$ is either constant or 
has an \abscont{} distribution; both alternatives are possible.
  \end{romenumerate}
\end{theorem}

\begin{proof}
Consider the urn with an added dummy colour 0 as in \refS{SpfT1}.

  \pfitemref{TEL0}
In this case, \refT{TD} shows that $\cX_i$ is a positive constant.
If $\hgl>0$, then
\eqref{hcx} shows that  $\hcX_i=\hgl^{-\kk_i}\cX_i$ is a constant.
If $\hgl=0$, then \eqref{hh1} and \refT{TD} show that also $\cX_0$ is a
positive constant; thus \eqref{hh11} shows that $\hcX_i$ is a constant.

\pfitemref{TEL<}
By \eqref{hh1}, we have $\glx_0=\hgl>0$. 
By \refL{LNDB}, $\cX_i$ is a linear combination \eqref{lndb} of $\cX_k$ for
leaders $k$ with $\glx_k=\glx_i$, and $\cX_0$ is a similar linear combination,
now for leaders $k$ with $\glx_k=\glx_0=\hgl$.
The two sets of leaders are disjoint, and thus \refT{Tac} implies, using
\refL{LZ}, that the  distribution of
$(\cX_i,\cX_0)$ is \abscont{} in $\bbR^2$.
It is then easily seen from \eqref{hcx} that the distribution of $\hcX_i$ is
\abscont.

\pfitemref{TEL=}
In this case, \refL{LNDB} shows that both $\cX_i$ and $\cX_0$ are linear
combinations \eqref{lndb} of $\cX_k$ for 
leaders $k$ with $\glx_k=\hgl$.
If the two vectors of coefficients of these linear combinations are
proportional, then $\cX_i$ and $\cX_0$ are proportional; since now
\eqref{hcx} shows that $\hcX=c\cX_i/\cX_0$ for some constant $c$, it follows
that $\hcX$ is constant.
On the other hand, if the vectors of coefficients are not proportional, then
\refL{LZ} shows that the distribution of
$(\cX_i,\cX_0)$ is \abscont{} in $\bbR^2$, and as in \ref{TEL<}, it follows
easily from \eqref{hcx} that the distribution of $\hcX_i$ is \abscont.
\end{proof}

An example where $\glx_i=\hgl$ and $\hcX_i$ is \abscont{} is given by the
classical \Polya{} urn, see \refE{Eclassical}. Many examples with degenerate
$\hcX_i$ are provided by the following theorem;
see also the examples in \refS{Sex}. 

\begin{theorem}\label{Tone}
    Suppose that there is only one leader $\nu$ with $\gl_\nu=\hgl$.
(This is equivalent to $\gl_\nu=\hgl$ and that $i\succeq \nu$ for
every colour $i$ with $\gl_i=\hgl$.)
Then $\hcX_i$ is deterministic for every $i$ with $\glx_i=\hgl$.  
\end{theorem}
\begin{proof}
If $\hgl=0$, this follows from \refT{TEL}\ref{TEL0}.

Suppose now $\hgl>0$, and
consider the urn with an added dummy colour 0 as in \refS{SpfT1}.
Then 
$\glx_0=\hgl$ by \eqref{hh1}.
Furthermore, 
$\gl_\nu=\hgl>0$ and thus $a_\nu>0$; 
hence $\nu\to 0$ and consequently
$\nu$ is still the only leader 
$k$ with $\gl_k=\hgl$.
In particular, $\sA_0=\set\nu$.
\refL{LNDB} 
shows that $\cX_i=c_{i\nu}\cX_\nu$ for every $i\in\sQ$ with
$\glx_i=\hgl$, and also that $\cX_0=c_{0\nu}\cX_\nu$. Hence, 
the result follows from
\eqref{hcx}.
\end{proof}


Even when the limits $\hcX_i$ are not deterministic, there are
generally 
linear dependencies between them just as for $\cX_i$.
In fact, in the main case $\hgl>0$,
we have the following analogue of \refLs{LNDB} and \ref{LC}. 
\begin{lemma}  \label{LhC}
If\/ $\hgl>0$, then for every $i\in\sQ$
\begin{align}\label{hlndb}
  \hcX_i = \sum_{k\in \sA_i} \hc_{ik}\hcX_k,
\end{align}
where 
\begin{align}\label{hc}
\hc_{ik}= \hgl^{-\kk_i}c_{ik}>0.  
\end{align}
Moreover, \refL{LC} holds also for $\hc_{ik}$, with
the only difference that \eqref{subl} is replaced by
\begin{align}\label{hsubl}
  \hc_{i\nu}=\sum_{j\in\sD_\nu^{\kk-1}}\frac{a_jr_{ji}}{\kk\hgl}\hc_{j\nu}.
\end{align}
\end{lemma}
\begin{proof}
The expansion \eqref{hlndb}--\eqref{hc} follows from \eqref{lndb} and  
\eqref{hcx},
since $k\in\sA_i$ implies $\glx_k=\glx_i$ and $\kk_k=0$.
The final claim follows by \eqref{hc} and \refL{LC}.
\end{proof}

We leave the corresponding result for the less interesting case $\hgl=0$ to
the reader.

\section{Urns with subtractions}\label{S-}

We have so far assumed \ref{A+}: $\xi_{ij}\ge0$ for all $i,j\in\sQ$.
In many applications, there are  urns with subtractions, where we allow
$\xi_{ij}<0$. In the present paper, we are only interested in cases where the
urn allows an infinite number of drawings according to the rules in
\refS{S:intro}; in other words we want \eqref{urn} to make sense as
probabilities for all $n$, and thus we require that the urn is such
that $X_{ni}\ge0$ for all $n$, 
and also  $\sum_i a_i X_{ni}\neq 0$.
Such urns are often called \emph{tenable}.
(See \eg{} \cite{Mahmoud} for a discussion and examples.)
A standard way to ensure that
$X_{ni}\ge0$ without assuming $\xi_{ii}\ge0$ is to assume that $X_{0i}$
and every $\xi_{ji}$ always  is an integer, 
so that $X_{ni}$ is an integer,
and that $\xi_{ii}\ge-1$ while $\xi_{ji}\ge0$ for
$j\neq i$.
Typically, this is done for all $i$. (Thus $\bX_n\in\bbZgeo^q$; we may then
regard the urn process as drawings without replacement followed by adding
$\xi_{ij}+\gd_{ij}\ge0$ balls of each colour $j$.) However, we will be
more flexible and allow a combination with this assumption
for some colours $i$, and
$\xi_{ji}\ge0$ (as in earlier sections) for the others. We therefore assume:
\begin{PQenumerate}{\Anude${}'$}
\item \label{A-5}
For each colour $i\in\sQ$, we have either (or both)
\begin{enumerate}
 \renewcommand{\theenumii}{(\alph{enumii})}%
 \renewcommand{\labelenumii}{\theenumii}%
  \item\label{A-5a} %
$\xi_{ji}\ge 0$ \as{} for all $j\in\sQ$, 
or
\item\label{A-5b} 
$\xi_{ii}\in\bbZgeo\cup\set{-1}$ \as{} and
$\xi_{ji}\in\bbZgeo$ \as{} for all $j\neq i$, 
and $X_i(0)\in\bbZgeo$.
\end{enumerate}  
\end{PQenumerate}
Note that \ref{A+} is \ref{A-5a} for every $i\in\sQ$.
Note also that \ref{A-5} implies 
$\xi_{ji}\ge0$ \as{} for all $i,j\in\sQ$ with $i\neq j$.
We  let
\begin{align}\label{sQm}
    \sQm:=\bigset{i\in\sQ:\P(\xi_{ii}<0)>0}
=\bigset{i\in\sQ:\P(\xi_{ii}=-1)>0}
\end{align}
denote the set of colours where \ref{A-5b} but not \ref{A-5a} applies.
Note that \ref{A00} implies that $a_i>0$ for every $i\in\sQm$.

\begin{remark}\label{RA-5}
  More generally, we may assume that, for a given $i$,
$\xi_{ii}$ may take some
negative value $-b$, provided 
it does not take any other negative value, and $X_i(0)$ and all $\xi_{ji}$
($j\in\sQ$) \as{} are multiples of $b$.
This case is easily reduced to the case $b=1$ in \ref{A-5} by dividing
all $X_{ni}$, $X_i(t)$, and $\xi_{ji}$ by $b$ and multiplying the activity
$a_i$ by $b$. 
(See \refE{Epref-} for an example where we, however, use
a slightly different alternative.)
For convenience, and thus without real loss of generality, we will assume
\ref{A-5}, where the only allowed negative replacement is $-1$.
\end{remark}

\begin{remark}\label{RBP}
  If $i\in\sQm\cap\sQmin$, then $X_i(t)$ 
is not influenced by the other colours, 
and \ref{A-5b} implies that
$X_i(t)$ is integer-valued and a classical \ctime{}
Markov branching process of the
type studied in \eg{} \cite[Chapter III]{AN}.
(As noted above, we have $a_i>0$ by \ref{A00}.)
Recall that \cite[Section III.7]{AN}
in the subcritical and critical cases $\gl_i\le0$,
this process \as{} dies out, \ie, $X_i(t)=0$ for all sufficiently large $t$,
while in the supercritical case $\gl_i>0$,
the process survives for ever with positive probability 
(assuming $\E\xi_{ii}\log\xi_{ii}<\infty$, as we do); 
this probability
is strictly less than 1, since the process always may die out when
$\P(\xi_{ii}=-1)>0$.
\end{remark}

We continue to use the definitions above, in particular
\eqref{gli}--\eqref{gam}. 
Note that we now may have $r_{ii}<0$, and thus $\gl_i<0$;
we may also have $\gl_i=r_{ii}=0$ without $\xi_{ii}=0$ a.s.

We will see that the results in the previous sections
are valid with minor changes also if we
replace \ref{A+} by \ref{A-5},
at least under some further technical assumptions 
\ref{A-6}--\ref{A-8} given below.
However, in order to get more general results,
these  will be assumed only when needed.
Hence,
we  assume throughout this section \refAAXZ{}, with 
the following further assumptions
added explicitly when needed.
\begin{Aenumerate}
\item \label{A-6}
   $\sum_{i\in\sQ} a_i X_i(t)>0$ \as{} for all $t\ge0$.
\item\label{A-7} 
If $i\in\sQm$, then $\glx_i>0$. 
\item\label{A-8} 
If $i\in\sQ$ is a minimal colour, then $\xi_{ii}\ge0$ a.s.\
(i.e., $i\notin\sQm$).
\end{Aenumerate}

In other words, \ref{A-7} says
that if $i\in\sQm$, 
then either $\gl_i>0$ or there exists $j\in\sQ$ with  $j\prec i$ and $\gl_j>0$
(or both).
In particular, if $i$ is a minimal colour, then either $\gl_i>0$ or
$\xi_{ii}\ge0$ a.s.

Note that \ref{A-8} implies that $X_i(t)\ge X_i(0)$ when $i$ is minimal;
since \ref{A0}--\ref{A3} imply that there exists a minimal $i$ with $a_i>0$
and $X_i(0)>0$, \ref{A-8} together with \ref{A0}--\ref{A3} imply \ref{A-6}.

\subsection{Some motivation for \ref{A-6}--\ref{A-8}}
The assumption \ref{A-5} implies that the continuous-time urn 
$\bX(t)=\xpar{X_i(t)}_{i\in\sQ}$
is well-defined for $t\in\ooo$, with $X_i(t)\ge0$ for all $i\in\sQ$. 
However, the urn might possibly reach a state where $X_i(t)=0$ for all $i$
with $a_i>0$, and thus $\sum_i a_i X_i(t)=0$; such a state is absorbing and
no more draws will be made, and then the total number of draws is finite
and the discrete-time urn is not well defined for all $n$.
For results on the discrete-time urn, we therefore further assume
\ref{A-6}.
Since $\sum_i a_iX_i(t)$ is constant between the jumps,
\ref{A-6} implies that waiting times $\TT_{n+1}-\TT_n$ are finite \as,
and thus there is \as{} an infinite
number of draws; hence $\TT_n<\infty$ for every $n$, and the discrete-time
urn $(\bX_n)\xoo$ is well defined by \eqref{ct}.

Next, we note that a colour $i$ with $\glx_i<0$ will die out:
\begin{lemma}\label{L<0}
  If $\glx_i<0$, then \as{} $X_i(t)=0$ for all large $t$, and thus
  $X_{ni}=0$ for all large $n$.
\end{lemma}
\begin{proof}
  Since every $j\to i$ has $\glx_j\le \glx_i<0$, we can use induction and
  assume that the statement holds for every $j\to i$.
Then there exists a.s.\ a (random) time $T$ such that no colour $j$ with
$j\to i$ exists for times $t\ge T$. Hence, for $t\ge T$, $X_i(t)$ evolves 
as a single colour urn with only colour $i$; 
we have $\gl_i\le\glx_i<0$, and thus, as in \refR{RBP}, 
this process is a subcritical Markov branching process, which dies out a.s.
\end{proof}

Hence, if $\glx_i<0$, we can wait until colour $i$ and all its ancestors have
disappeared, and we may then regard the urn as restarted at that time; 
then we have an urn with fewer colours. Consequently, we may without loss of
generality assume $\glx_i\ge0$ for every colour $i$.
Note that it follows from the definition \eqref{glx} that
\begin{align}\label{glge0}
  \glx_i\ge0 \text{ for every colour } i
\iff
\gl_j\ge0 \text{ for every minimal colour }j
.\end{align}

Moreover, also the case $\glx_i=0$ may be problematic when $i\in\sQm$,
and we will actually make a stronger assumption than \eqref{glge0}. There are
two reasons.

First, suppose that $i\in\sQm$ is a minimal colour (and thus $\glx_i=\gl_i$).
Then, as said in \refR{RBP}, 
$X_i(t)$ is a 
Markov branching process which dies out also in the critical case 
$\gl_i=\glx_i=0$,
so as above we may in this case 
wait until colour $i$ has disappeared and consider
an urn with fewer colours.

Secondly, and more importantly,
if $i\in\sQm$ is not minimal and $\glx_i=0$, there are examples where
(after normalization)
$X_i(t)$  converges in distribution but \emph{not} a.s.,
and similarly for $X_{ni}$; see \refEs{E+-}
and \ref{E--}. 
(In \refE{E--}, $X_{ni}$ does not even converge in distribution.)

We therefore exclude these cases and assume \ref{A-7}.
Note that if $i\notin\sQm$, then $\glx_i\ge\gl_i\ge0$.
Hence, \ref{A-7} implies that 
$\glx_i\ge0$ for every $i$, and thus
both conditions in \eqref{glge0} hold.
(So we do not have to assume this explicitly.)

Even with these assumptions,
one  complication remains.
If $i\in\sQm$ is a minimal colour, then, as noted in \refR{RBP}, $X_i(t)$ is a
branching process, and even in the supercritical case $\gl_i>0$, 
it dies out with positive probability.
We will accept this complication, but note that it may be eliminated by 
the additional assumption \ref{A-8}, which we only assume when needed.

\subsection{Main results for urns with subtractions}

With the assumptions above, our main theorems for discrete and continuous
time still hold:

\begin{theorem}\label{T1-}
Let $(X_{i}(t))_{i\in\sQ}$ be a discrete-time
triangular \Polya{} urn satisfying the 
conditions \refAAZZ. 
Then the conclusions of \refT{T1} hold.
\end{theorem}

We do not explicitly assume \ref{A-6} in \refT{T1-}, but it is implicit
since it follows from \ref{A-8} and the other assumptions, as noted above.

Without \ref{A-8}, we get a more complicated partial result
(where $\cX_0$ is as in \refS{SpfT1}); it is particularly useful
in cases where $\cX_0>0$ \as\ for some other reason, for example by
\refL{LBB5} below in the balanced case.
(See \refE{Epref-}.)

\begin{theorem}\label{T1--}
Let $(X_{i}(t))_{i\in\sQ}$ be a discrete-time
triangular \Polya{} urn satisfying the 
conditions \refAAZ. 
Then the conclusions of \refT{T1} hold
\as{} on the event $\set{\cX_0>0}$
(which has positive probability),
except that $\hcX_i=0$ is possible, 
with $0\le\P(\hcX_i=0\mid\cX_0>0)<1$.
\end{theorem}

\begin{theorem}  \label{TC-}
Let $(X_{i}(t))_{i\in\sQ}$ be a continuous-time
triangular \Polya{} urn satisfying the 
conditions \refAAZC. 
Then the conclusions of \refT{TC} hold, 
except that $\cX_i=0$ is possible, 
with $0\le\P(\cX_i=0\mid\cX_0>0)<1$;
moreover, we have $\P(\cX_i>0\;\forall i\in\sQ)>0$.

If furthermore \ref{A-8} holds, then the all conclusions of \refT{TC} hold.
\end{theorem}

\begin{remark}\label{RcX=0}
  In \refTs{T1--} and \ref{TC-}, if the limit $\hcX_i=0$ or $\cX_i=0$, then
  the result \eqref{t1b}, \eqref{t1c} or \eqref{tc1} does not give the
  correct growth of $X_{ni}$ or $X_i(t)$. In such cases, it might be
  possible to find more precise results using our methods on suitable
  subsets of the colours;
it seems that this can be done very generally, at least on a
case-by-case basis. We give one example in \refE{EcX=0} but do not attempt
to state any general theorem.
\end{remark}

\begin{remark}\label{Rdep-}
Also the results in \refS{Sdep}--\ref{Sdeg} still hold
(with the same proofs)
under the assumptions \refAAZZ, as the reader might verify.

If we do not assume \ref{A-8} and assume only \refAAZ, then the results on
absolute continuity do not hold as stated, since typically the limits may be
0 with positive probability. (We conjecture that these results might hold
conditioned on the variables being non-zero, but we have not pursued this.)
The remaining results in \refS{Sdep}--\ref{Sdeg} still hold.
\end{remark}

Note that \ref{A+} implies \ref{A-5}, \ref{A-7}, and \ref{A-8}
(with $\sQm=\emptyset$), and thus \refAA{} imply also \ref{A-6};
hence these theorems (strictly) extend \refTs{T1} and \ref{TC}.

We prove \refTs{T1-}--\ref{TC-} in the remainder of  this section.
As in \refS{S1}, we first study a single colour, and as there we split the
discussion into several subsections.

First note that  we used that $X_j(t)$ is
increasing in \refL{LtN} and its proof.
This is no longer necessarily true, but 
we may replace $X_j(t)$ by $X^*_j(t)$ with no further
consequences. Hence, \refLs{LtN} and \ref{L9+} hold with the assumption 
modified to $\E X_j^*(t)<\infty$.

\subsection{A colour not influenced by others}\label{ZSS11}
Consider the situation in \refSS{SS11}, where $\xi_{ji}=0$ for $j\neq i$,
\ie, $i\in\sQmin$.
If $i\in\sQm$,
then $X_i(t)$ is, as noted in \refR{RBP}, a  Markov branching process.

\refL{LM} extends to this case, with some modifications.

\begin{lemma}\label{ZLM}
Suppose that \refAAXZ{} hold and that $i\in\sQmin$.
\begin{alphenumerate}[-10pt]
  
\item\label{ZLMa} 
If\/ 
$i\notin\sQm$ (\ie, $\xi_{ii}\ge 0$ a.s.),
then \refL{LM} still holds.

\item\label{ZLMb} 
If\/ 
$i\in\sQm$, 
then  \refL{LM} holds with the following modifications:
\end{alphenumerate}
  \begin{romenumerate}[0pt]
    
  \item\label{ZLM+} 
If\/ $\gl_i>0$,
the only difference is that $\cX_i=0$  is possible with positive
probability. If\/ $x_0>0$, then $0<\P(\cX_i=0)<1$.
  \item\label{ZLM0} 
If\/ $\gl_i=0$, then $X_i(t)$ is a martingale with
    \begin{align}\label{zlm01}
      X_i(t)&\asto \cX_i=0,
\qquad\text{as \ttoo},
\\\label{zlm00}
\E\bigsqpar{X_i(t)}&= x_0,
\\\label{zlm02}
\Var\bigsqpar{X_i(t)}&= Cx_0t,
\\\label{zlm03}
\E \bigpar{\sup_{0\le t\le u}X_i(t)}^2&\le Cx_0^2+Cx_0 u,
\qquad\text{ for every $u<\infty$}
.    \end{align}
Furthermore, for every $\gd>0$,
\begin{align}\label{zlm04}
\E \bigpar{\sup_{ t\ge0}\cpar{e^{-\gd t}X_i(t)}}^2&<\infty.
\end{align}

  \item\label{ZLM-} 
If\/ $\gl_i<0$, then
    \begin{align}\label{zlm-1}
      X_i(t)&\asto \cX_i=0,
\qquad\text{as \ttoo},
\\\label{zlm-2}
\Var\bigsqpar{X_i(t)}&\le C x_0 e^{\gl_i t} 
,
\\\label{zlm-3}
\E \bigpar{\sup_{0\le t<\infty}X_i(t)}^2&\le Cx_0^2.
    \end{align}
  \end{romenumerate}
\end{lemma}

\begin{proof}
\pfitemref{ZLMa}
We thus assume \ref{A+} = \ref{A-5a} for the colour $i$,
and then the proof of \refL{LM} still holds.
(Recall that for simplicity we had
\refAA{} as standing assumptions in \refS{S1},
including \ref{A+} for all colours;
however, 
as remarked before \refL{LM}, 
this is needed only for the colour $i$ in
that subsection, and in particular for \refL{LM} and its proof.)

\pfitemref{ZLMb}
Now assume \ref{A-5b}.
By \ref{A00}, we  also have $a_i>0$.

We argue as in the proof of \refL{LM}, with the following differences.
Note that $x_0$ is assumed to be an integer, 
and the case $x_0=0$ is trivial; we thus may assume $x_0\ge1$.
It is still true that $M(t):=e^{-\gl_i t}X_i(t)$ is a martingale, and
\eqref{lm6} holds if $\gl_i\neq0$.

Now, however, as said in \refR{RBP}, with positive probability,
$X_i(t)=0$ for all large $t$;
moreover, in the critical and subcritical cases \ref{ZLM0} and \ref{ZLM-},
this happens a.s.

We study the three cases separately:

  \pfitemref{ZLM+}
The proof of \refL{LM} still holds, except for the final part yielding
\eqref{lm11}. 

By \refR{RBP}, 
with positive probability
$X_i(t)$ dies out and thus $\cX_i=0$. 
Moreover,
by \eqref{lm2},
$\E\cX_i=x_0>0$.
Consequently, $0<\P(\cX_i=0)<1$. 

\pfitemref{ZLM0}
By \refR{RBP}, 
\as{} $X_i(t)=0$ for all large $t$, which gives \eqref{zlm01}.

 For $\gl_i=0$, \eqref{cb3} says that
$X_i(t)=M(t)$ is a martingale, which implies \eqref{zlm00}.
Furthermore, \eqref{lm6} in this case yields
\begin{align}\label{zlmgb}
  \E X_i(t)^2 = \E M(t)^2=\E[M,M]_t = x_0^2+a_i\gb x_0 t,
\end{align}
which gives \eqref{zlm02} and, together with Doob's inequality, \eqref{zlm03}.
To prove \eqref{zlm04}, we note that \eqref{zlm03} implies
\begin{align}\label{zlm03-4}
\E \Bigpar{\sup_{ t\ge0}\cpar{e^{-\gd t}X_i(t)}}^2&
\le \E\sumno e^{-2\gd n}\Bigpar{\sup_{n\le t\le n+1}X_i(t)}^2
\le \E\sumno e^{-2\gd n}O(n+1)
<\infty.
\end{align}

\pfitemref{ZLM-}
As for \ref{ZLM0}, 
\refR{RBP} shows  that
\as{} $X_i(t)=0$ for all large $t$, and thus \eqref{zlm-1} holds. 
Since $\gl_i<0$, we obtain from \eqref{lm6}
\begin{align}\label{zlm-4}
  \E M(t)^2 = \E [M,M]_t \le x_0^2 + C x_0 e^{-\gl_i t} 
\end{align}
and thus
\begin{align}\label{zlm-5}
  \Var \bigsqpar{X_i(t)} = e^{2\gl_i t}\Var \bigsqpar{M(t) }
\le C x_0 e^{\gl_i t} 
,\end{align}
showing \eqref{zlm-2}. Furthermore, \eqref{zlm-4} and Doob's inequality
yield, for any $m\ge1$,
\begin{align}
  \E \bigsqpar{\sup_{m-1\le t\le m}X_i(t)^2}
&\le  e^{2\gl_i (m-1)}\E\bigsqpar{\sup_{m-1\le t\le m}M(t)^2}
\le C e^{2\gl_i m}\E \bigsqpar{M(m)^2}
\notag\\&
\le C e^{\gl_i m}x_0^2.
\end{align}
Since $\gl_i<0$, the sum over all $m\ge1$ is $\le C x_0^2$, which implies
\eqref{zlm-3}. 
\end{proof}

\begin{remark}\label{Rdie}
In fact, it follows from \cite[Theorem III.7.2]{AN} (using \ref{A2}) that
if $i\in\sQmin\cap\sQm$, then
$\cX_i=0$ occurs \as{} exactly when 
$X_i(t)=0$ for large $t$ so that the branching process dies out in finite time.
Thus,  $\P(\cX_i=0)$ equals the
probability that the \ctime{} branching process $X_i(t)$ dies out.
Considering only the times that a ball of colour $i$ is drawn, we obtain an
embedded  
random walk with \iid{} increments distributed as $\xi_{ii}$;
thus $\P(\cX_i=0)$ also equals the probability that this random walk hits 0.
(It is easy to see that this also equals the probability of extinction for a 
\GWp{} with offspring distribution $1+\xi_{ii}$ started with $x_0$ individuals.)
\end{remark}

\subsection{A colour only produced by one other colour}\label{ZSS12}
Consider the situation in \refSS{SS12}, where 
there is a single colour $j$ such that $j\to i$.

We used that $X_j(t)$ is increasing in \eqref{j32};
this is no longer necessarily true, but
we may replace $X_j(t)$ by $X^*_j(t)$ there with no further
consequences. 

\refL{L12} holds with the following minor modifications.
We assume $\glx_j\ge0$, since otherwise $j$ eventually becomes extinct 
by \refL{L<0}, 
and after that $X_i(t)$ evolves as in \refL{ZLM}.
(We are not really interested in this case, as discussed earlier.)
Note that we also exclude the case $\glx_i=\glx_j=0$.
(For good reasons, see \refEs{E+-} and \ref{E--}).

\begin{lemma}\label{ZL12}
Suppose that \refAAXZ{} hold. Let $i\in\sQ$, and
suppose that 
there is exactly one colour $j\in\sQ$ such that $j\to i$, 
and that $X_i(0)=0$.
Suppose also that one of the following holds:
\begin{alphenumerate}
\item\label{ZL12a} 
$i\notin\sQm$ (\ie, $\xi_{ii}\ge 0$ a.s.), and $\glx_j\ge0$.

\item\label{ZL12b} 
$i\in\sQm$, 
$\glx_i>0$, and 
$\glx_j\ge0$.
\end{alphenumerate}
Then \refL{L12} still holds, i.e.,
if 
\eqref{j1}--\eqref{j2} hold,
then we have the conclusions 
\eqref{ji1}--\eqref{ji3} and $\cX_j>0\implies \cX_i>0$ a.s.
\end{lemma}

\begin{proof}
\pfitemref{ZL12a}
Recall again that we had 
\refAA{} as standing assumptions in \refS{S1},
including \ref{A+} for all colours;
however, it is easily verified that 
the proof of \refL{L12} does not use this for other
colours than $i$, except
in  \eqref{j32}, where we now replace $X_j(t)$ by $X^*_j(t)$ as discussed
above, and to see that $\glx_j\ge0$.
Hence, in the present setting where we assume \refAAXZ, and also explicitly
assume $\glx_j\ge0$, 
if $\xi_{ii}\ge0$ a.s.,
then the proof of \refL{L12} still holds.

\pfitemref{ZL12b}
We now assume \ref{A-5b} for $i$.
Most of the proof of \refL{L12} remains the same.
The main difference in the proof comes in
\refSteps{stepZ1}--\ref{stepZ123b}, where
we used \refL{LM} in \eqref{j5b}, but now instead use \refL{ZLM}\ref{ZLMb}.
We consider three cases:
\begin{PXenumerate}
  
\item 
\emph{$\gl_i>0$}: Then \eqref{j5b} 
still holds by \refL{ZLM},
and thus all estimates
in \refSteps{stepZ1}--\ref{stepZ123b} hold.

\item 
$\gl_i=0$: Then \eqref{j5b} is replaced by, using \eqref{zlm02},
\begin{align}\label{zj5b}
\E \bigsqpar{  [Z_1,Z_1]_t\mid (T_k,\eta_k)\xoo}&  
= \sumk\indic{T_k\le t}
  \Var\bigsqpar{Y_k(t-T_k)\mid T_k,\eta_k}
\notag\\&
= \sumk\indic{T_k\le t} C(t-T_k)\eta_k
.\end{align}
Thus
\begin{align}\label{zj6}
\E   [Z_1,Z_1]_t&  
= C\E\sumk\indic{T_k\le t}(t-T_k) 
\le C t \E N(t).
\end{align}
Using \refL{L9+} (or \refL{LtN}),
it follows that \eqref{j51} for $\ell=1$ is replaced by (with $\gl_i=0$)
\begin{align}\label{zj51}
\E  [Z_1,Z_1]_t &
\le C t \E\intot X_j(s)\dd s
\le C\E\tXX_jt\intot (s+1)^{\kk_j}e^{(\glx_j-2\gl_i)s}\dd s
,\end{align}
with an extra factor $t$.
We assume $0<\glx_i=\gl_i\bmax\glx_j$ and $\gl_i=0$, and thus
$\glx_j>0=\gl_i$, so we are in Case \ref{step5-ii'} of \refStep{stepZ123b};
\eqref{j54} now holds with an  insignificant extra factor $n$
and thus \eqref{j55}--\eqref{j56} still hold and the conclusions 
of \refStep{stepZ123b}
remain the same.

\item 
$\gl_i<0$: Then we obtain instead of \eqref{j5b}, now using
\eqref{zlm-2},
\begin{align}\label{zzj5b}
\E \bigsqpar{  [Z_1,Z_1]_t\mid (T_k,\eta_k)\xoo}& 
\le \sumk\indic{T_k\le t} e^{-2\gl_i T_k}Ce^{-\gl_i(t-T_k)}\eta_k
\notag\\&
=C \sumk\indic{T_k\le t} e^{-\gl_i t -\gl_i T_k}\eta_k
.\end{align}
Thus, by \refL{L9+}\ref{L9+b},
\begin{align}\label{zzj6}
\E {[Z_1,Z_1]_t}&
\le C \E\sumk\indic{T_k\le t} e^{-\gl_i t -\gl_i T_k}\eta_k
=Ce^{-\gl_i t}\E\intot e^{-\gl_i s}X_j(s)\dd s
.\end{align}
It follows that \eqref{j51} for $\ell=1$ is replaced by
\begin{align}\label{zzj51}
\E  [Z_1,Z_1]_t &
\le C\E\tXX_je^{-\gl_i t}\intot (s+1)^{\kk_j}e^{(\glx_j-\gl_i)s}\dd s
.\end{align}
However, 
for $t\le1$, \eqref{zzj51} is trivially equivalent to \eqref{j51},
and for $t\ge1$, \eqref{zzj51} implies,
using $\glx_j-\gl_i\ge-\gl_i>0$,
\begin{align}\label{zzj52}
\E  [Z_1,Z_1]_t &
\le C\E\tXX_j(t+1)^{\kk_j}e^{(\glx_j-2\gl_i)t}
\notag\\&
\le C\E\tXX_j\int_{t-1}^t (s+1)^{\kk_j}e^{(\glx_j-2\gl_i)s}\dd s
.\end{align}
Hence \eqref{j51} holds for all $\ell\le3$ also in the present setting.
By our assumption $\glx_i=\gl_i\bmax\glx_j>0$ we have $\glx_j>0$, 
so we are again in Case \ref{step5-ii'} of \refStep{stepZ123b};
\eqref{j54}--\eqref{j56} still hold and the conclusions 
of \refStep{stepZ123b}
remain the same.
\end{PXenumerate}

In all cases, the conclusions \eqref{ji1}--\eqref{ji3}  of \refStep{stepVI}
follow as for \refL{L12}.

Finally, as for \refL{L12}, if $\gl_i\le\glx_j$, then \eqref{ji3} shows that
$\cX_j>0\implies\cX_i>0$. If $\gl_i>\glx_j$ and $\cX_j>0$, then we use again
\eqref{j61}, where still a.s.\ all $T_k$ are finite.
Let $\cY_k:=\lim_\ttoo e^{-\gl_i t}Y_k(t)$. 
In contrast to \refL{L12}, $\cY_k=0$ is now possible also if $\eta_k>0$, but
$0<\P(\cY_k=0)<1$ by \refL{ZLM}. We now let $K$ be the smallest $k$ such that 
$\cY_k>0$, and conclude as in \refS{S1} that $\cX_i>0$.
\end{proof}

\begin{remark}
In \refL{ZL12}\ref{ZL12b}
we have excluded the case
$\gl_i=\glx_j=0$, since then $\glx_i=0$. In this case
we  are in Case \ref{step5-iii'} of \refStep{stepZ123b}
of the proof of \refL{L12}.
As mentioned above, there is now an extra factor $t$ in \eqref{j51}, and
therefore
we obtain instead of \eqref{j3c2}
\begin{align}\label{zj3c2}
\E \tZccl(n)^2&
\le \bC 2^{-\kk_jn}
.\end{align}
Hence,  if furthermore $\kk_j\ge1$, then \eqref{j38b} still holds, and the rest
of the proof of \refL{L12} applies.
Consequently, \refL{ZL12}\ref{ZL12b} holds also in the case $\gl_i=\glx_j=0$ and
$\kk_i\ge1$. 

\refE{E+-} gives an example with $\gl_i=\glx_j=0=\kk_i$
where the conclusion of \refL{ZL12} does not hold.
For simplicity we have in \refL{ZL12}\ref{ZL12b}  excluded all cases
with $\gl_i=\glx_j=0$ (and thus $\glx_i=0$).

Similarly, \refE{E--}
gives an example with $\gl_i<0=\glx_j=\kk_i$
where the conclusion of \refL{ZL12} does not hold.
We have excluded all cases 
with $\gl_i<\glx_j=0$; it seems likely that \refL{ZL12}\ref{ZL12b} 
might hold also in
this case if $\kk_i\ge1$, but we have not investigated this further.
\end{remark}

We need also a coarse estimate in the case excluded in \refL{ZL12}\ref{ZL12b}.
\begin{lemma}\label{ZL1200}
Suppose that \refAAXZ{} hold and that $i\in\sQ$ is  such that 
$\P(\xi_{ii}=-1)>0$.
Assume that 
there is exactly one colour $j\in\sQ$ such that $j\to i$, 
and that $X_i(0)=0$.
Assume also that $\glx_i=\glx_j=0$.

Assume further that
\eqref{j1}--\eqref{j2} hold.
Then, for every $\gd>0$,
\begin{align}\label{zl1200a}
  e^{-\gd t}X_i(t)\asto0,
\end{align}
and
\begin{align}\label{zl1200b}
\Bignorm{\sup_{t\ge0} \bigcpar{e^{-\gd t}X_i(t)}}_2<\infty.  
\end{align}
\end{lemma}
\begin{proof}
We have $\gl_i\le\glx_i=0$ and, by \ref{A00}, $a_i>0$. 
Thus $r_{ii}=\gl_i/a_i\le0$.
Furthermore, $\xi_{ii}\ge-1$ a.s., and thus
$r_{ii}=\E\xi_{ii}\ge-1$. 
Choose $p\in(0,1)$ such that $0<r_{ii}+2p<\gd/a_i$.
 
 Modify the urn by increasing $\xi_{ii}$ to $\chxi_{ii}:=\xi_{ii}+2\zeta$,
  where $\zeta\in\Be(p)$ 
is independent of
  $\xi_{ii}$;
  all other $\xi_{k\ell}$ and initial conditions $X_k(0)$ remain the same.
We denote quantities for the new urn by adding $\widecheck{\ }$.
The modification does not affect any ancestors of $i$; thus
$\chX_j(t)=X_j(t)$ and $\chglx_j=\glx_j=0$.
On the other hand, by our choice of $p$,
\begin{align}\label{chgl}
\chgl_i=a_i\E\chxi_{ii} = a_i(\E\xi_{ii}+2\E\zeta)
=a_i(r_{ii}+2p)
\in(0,\gd).
\end{align}
Consequently, $\chglx_i=\chgl_i\bmax\chglx_j=\chgl_i>0$
and thus \refL{ZL12}\ref{ZL12b} applies to the modified urn.
(The other conditions of \refL{ZL12}\ref{ZL12b} obviously hold.)

The modification adds extra balls of colour $i$, and these may get
descendants of colour $i$ and they may disappear again, but we may separate
these extra balls of colour $i$ from the original ones, and thus couple the
two versions such that $\chX_i(t)\ge X_i(t)$ for all $t\ge1$.
Consequently, noting that 
 $\chglx_i=\chgl_i$ and $\widecheck\kk_i=0$,
\refL{ZL12} yields, from \eqref{ji1}--\eqref{ji2} and \eqref{j4}--\eqref{j4*}, 
\begin{align}
\limsup_\ttoo e^{-\chgl_i t}X_i(t)
\le
\lim_\ttoo e^{-\chgl_i t}\chX_i(t) = \chcX_i<\infty
\quad\text{a.s.},
\end{align}
and 
\begin{align}
\Bignorm{\sup_{t\ge0} \bigcpar{e^{-\chgl_i t}X_i(t)}}_2
\le
\Bignorm{\sup_{t\ge0} \bigcpar{e^{-\chgl_i t}\chX_i(t)}}_2
<\infty.  
\end{align}
The results follow, since $\chgl_i<\gd$ by \eqref{chgl}.
\end{proof}

\subsection{The general case for a single colour}\label{ZSSLG}

\begin{lemma}\label{ZLG}
Suppose that \refAAZC{} hold.
Then Lemma \ref{LG} still holds,
except for the last sentence.

The last sentence remains valid if $i\notin\sQm$, and also if $i\notin\sQmin$.
(In particular, it is valid if  \ref{A-8} holds.)
In the remaining case, when $i\in\sQm$ and is minimal, then $0<\P(\cX_i>0)<1$.
\end{lemma}

\begin{proof}
Recall that the assumption \ref{A-7} implies that $\glx_j\ge0$ for every
$j\in\sQ$.

We follow the proof of \refL{LG} and split the colour $i$ into subcolours
$i_0$ and   $i_j$, $j\in\sP_i$
(also if $i$ is minimal so $\sP_i=\emptyset$); 
then \eqref{ca1} holds, and also
\eqref{gli0}--\eqref{g6} and \eqref{g66}. 

If $i\notin\sQm$, 
then the proof is exactly as for \refL{LG}, now using
\refLs{ZLM}\ref{ZLMa} and \ref{ZL12}\ref{ZL12a}
instead of \refLs{LM} and \ref{L12}.

Assume thus in the sequel of the proof that instead 
$i\in\sQm$.
(Thus, \ref{A-5b} holds.) 
Then
by \ref{A-7}, we have $\glx_i>0$.
However, it is possible that $\gl_{i_0}=\gl_i\le0$ or that
$\glx_{i_j}=\gl_i\bmax\glx_j=0$ for some $j\in\sP_i$.

For $i_0$, we now apply \refL{ZLM}\ref{ZLMb}.
If $\gl_i>0$, then 
\refL{ZLM}\ref{ZLMb}\ref{ZLM+}
shows that $X_{i_0}$ can be treated as before in \eqref{g6} and \eqref{g66}.
If $\gl_i\le0$, we replace, as we may, $\tX_{i_0}(t)$ by $e^{-\glx_i t}X_{i_0}(t)$
in the estimate \eqref{g6}, cf.\ \eqref{ca1} and \eqref{j4}.
By \eqref{zlm01} and \eqref{zlm-1}, this term contributes 0 to the limit
\eqref{g3}.
 Similarly, in \eqref{g66} we replace $\tXX_{i_0}$ 
by $\sup_{t\ge0}\cpar{e^{-\glx_i t}X_{i_0}(t)}$,
which gives a finite contribution to \eqref{g4} by
\eqref{zlm04} and \eqref{zlm-3}.

Similarly,
for each $i_j$ ($j\in\sP_i$), we  now apply 
\refL{ZL12}\ref{ZL12b} if $\glx_{i_j}>0$;
otherwise, i.e., if $\gl_i\le0$ and $\glx_j=0$,
we apply \refL{ZL1200} (with $\gd:=\glx_i>0$).
In the latter case we replace
$\tX_{i_j}(t)$ by $e^{-\glx_i t}X_{i_j}(t)$
in the estimate \eqref{g6} and use \eqref{zl1200a},
and we replace
$\tXX_{i_j}$ by $\sup_{t\ge0}\cpar{e^{-\glx_i t}X_{i_j}(t)}$
 in  \eqref{g66} and use \eqref{zl1200b}.

This proves, in all cases, that \eqref{g3} and \eqref{g4} hold.
Moreover, for each $j\in\sP_i\cup\set0$ such that we apply 
\refL{ZLM}\ref{ZLMb}\ref{ZLM0},\ref{ZLM-} or \refL{ZL1200} to $i_j$,
the contribution to the limit \eqref{g3} is 0,
and so is the contribution to the \rhs{} of \eqref{g7} since
in these cases $\glx_{i_j}\le0<\glx_i$ by \eqref{gli0}--\eqref{glij}.
Hence, \eqref{g7} still holds.

It remains to consider the final sentence in \refL{LG}.
First, if $\sP_i\neq\emptyset$, 
then it follows as in the proof of \refL{LG} that if
$\cX_j>0$ for every $j\in\sP_i$, then $\cX_i>0$.
In the remaining case, $i$ is minimal. 
Note that then $X_i(0)>0$ by \ref{A3}.
If $i$ is minimal and $i\notin\sQm$, 
then \as{} $\cX_i>0$  by \refL{LM}.
On the other hand,
if $i$ is minimal and $i\in\sQm$, then 
\refL{ZLM}\ref{ZLMb}\ref{ZLM+}
shows that $\P(\cX_i=0) \in(0,1)$.
\end{proof}

\subsection{Proofs of \refTs{T1-}--\ref{TC-}}\label{Spf-}

\begin{proof}[Proof of \refT{TC-}]
As in \refS{STC}, now using \refL{ZLG}; we use again induction on the colour
$i$ and conclude that \eqref{g3} and \eqref{g4} hold for every $i\in\sQ$,
although now $\cX_i=0$ is possible.

 The processes $X_i(t)$ for $i\in\sQmin$
are independent, and thus so are their limits $\cX_i$.
Each is strictly positive with positive probability, by
\refL{ZLM} or \ref{ZLG}, and thus
there is a positive probability that $\cX_i>0$ for every minimal $i$.
Moreover, if \ref{A-8} holds, then this probability is 1.
Finally, it follows by \refL{ZLG} by induction on the colour $i$ 
that  \as{} if $\cX_i>0$ for every minimal $i$, then
$\cX_i>0$ for every  $i$.
\end{proof}

\begin{proof}[Proof of \refTs{T1-} and \ref{T1--}]
As in \refS{SpfT1}, now using \refT{TC-} and considering only the event
$\set{\cX_0>0}$. 
Note that \refT{TC-}, applied to the extended urn, shows that
with positive probability $\cX_0>0$ and $\cX_i>0$ for all $i\in\sQ$;
furthermore, in \refT{T1-}, where we assume \ref{A-8}, this holds a.s.

Note also that since we assume \ref{A-6}, we have \as{} $\TT_n<\infty$ for
every $n$, and thus the 
\dtime{} process $\bX_n$ is well-defined.
\end{proof}

\section{Random vs non-random replacements}\label{Smean}

Consider a \Polya{} urn $\cU$
with a random replacement matrix $(\xi_{ij})_{i,j\in\sQ}$.
For simplicity we study only the case when all $\xi_{ij}\ge0$, and we thus
assume \refAA.
Consider also another urn $\cU'$ with the same colours $\sQ$, 
the same initial vector $\bX_0$,
and the same
activities $a_i$, but non-random replacements $r_{ij}=\E\xi_{ij}$.
We thus replace the replacement matrix by its mean, and we may call $\cU'$
the  \emph{mean urn} corresponding to $\cU$.
Both urns have the same $r_{ij}$, and thus the same colour graph, $\gl_i$,
$\glx_i$, $\kk_i$, $\hgl$, $\hkk$ and $\gam_i$,
see \eqref{itoj} and \eqref{gli}--\eqref{gam}.
The new urn $\cU'$ 
also satisfies \refAA, and \refTs{T1} and \ref{TC} show that we
have the same qualitative asymptotic behaviour for both urns, with the same
normalization factors. However,
the limits $\hcX_i$ and $\cX_i$ are in
general \emph{not} the same for the two urns
$\cU$ and $\cU'$,
and thus the asymptotic
distributions may be be different; see  \refE{E2p}.

We note that the constants $c_{ik}$ in \refS{Sdep} by \refL{LC} are the same
for the two urns.
This yields one simple case. (We assume $\hgl>0$ for simplicity, and leave a
study of the case $\hgl=0$ to the reader.)
\begin{theorem}\label{Tmean}
Assume $\hgl>0$.
  If\/ $i$ is a colour such that $\hcX_i$ is deterministic for one of the two
  urns, then $\hcX_i$ is the same constant for both urns.
\end{theorem}
\begin{proof}
By \refT{TEL}, there are two cases: either $\glx_i=0$, or $\glx_i=\hgl$.

If $\glx_i=0$,
then \eqref{glai} and
\refL{LL0} show that 
for every $j\in\sA_i$, 
$\cX_j=X_j(0)$ is a constant, the same  for both urns.
Hence,  \refLs{LNDB} shows, 
since the constants $c_{ik}$ are the same for both urns,
that $\cX_i$ is the same constant for both urns.
Consequently, \eqref{hcx} shows that
$\hcX_i=\hgl^{-\kk_i}\cX_i$ is the same constant for both urns.

Assume now $\glx_i=\hgl$. If we add dummy balls to both urns as in
\refS{SpfT1}, then, by \refL{LNDB} and \eqref{hh1}, for any of the urns,
both $\cX_i$ and $\cX_0$ are linear combinations 
$\sum_{j\in J}c_{ij}\cX_j$ and $\sum_{j\in J}c_{0j}\cX_j$,
where $J$ is the set of leaders with $\glx_j=\hgl$.
The joint distribution of $(\cX_j)_{\in J}$ is \abscont{} by \refT{Tac}, and
since \eqref{hcx} yields $\hcX_i=\hgl^{-\gam_i}\cX_i/\cX_0$,
it follows that if $\hcX_i$ is deterministic in one of the two urns, then 
the vectors $(c_{ij})_{j\in J}$ and $(c_{0j})_{j\in J}$ are proportional.
Since these vectors are the same for both urns, it follows that then
$\hcX_i$ is the same constant in both urns.
\end{proof}

\section{Balanced urns}\label{Sbalanced}
Many papers on \Polya{} urns assume that urn is balanced (as defined below);
while this is quite restrictive, it is partly justified by the fact that
\Polya{} urns that appear in applications often are balanced.
We have not needed this assumption in the preceding sections, but it will be
used for some results in the following sections (in particular
\refS{Smoments}). 
The present section contains the definition and some
simple results for later use. 

In the standard case when all activities $a_i=1$, a \Polya{} urn is said to
be \emph{balanced} if we add the same total number of balls at each draw;
in other words, if the replacement matrix $(\xi_{ij})_{i,j\in\sQ}$ satisfies
$\sum_j\xi_{ij}=\gbal$ a.s., for some $\gbal\in\bbR$ and every $i\in\sQ$.
In general, with possibly different activities $a_i$, it is not the total
number of balls but their total activity that is important, and we make the
following general definition.
(\refD{DB} and \refL{LB1} -- \refR{RBT}
below apply to all \Polya{} urns, triangular or not.)

\begin{definition}\label{DB}
A \Polya{} urn is \emph{balanced} if there exists a constant $\gbal\in\bbR$
  (called the \emph{balance}) such that for every $i\in\sQ$ with $a_i>0$, we
  have
  \begin{align}
    \label{ume1}
\sum_{j\in\sQ} a_j\xi_{ij}=\gbal
\qquad\text{a.s.}
  \end{align}
\end{definition}

With $\ba:=(a_j)_{j\in\sQ}$, the vector of activities, \eqref{ume1} may be
written $\ba\cdot\bxi_i=\gbal$ for all $i$ with $a_i>0$.

Suppose that the urn is balanced as in \refD{DB}.
It follows from \eqref{xin} (since colours $i$ with $a_i=0$ will never be drawn)
that a.s.
\begin{align}
  \label{ume3}
\ba\cdot\bX_n=\sum_{j\in\sQ} a_jX_{nj}=\ba\cdot\bX_0+n\gbal.
\end{align}
Hence, the denominator in \eqref{urn} is deterministic.
This has been used in several ways in many papers by different authors.
(In particular, it makes it possible to use martingale methods
for the \dtime{} urn, 
similar to the \ctime{} arguments in the present paper;
see \refR{Rmart} and \eg{} \cite{BoseDM}.)
Here we note some simple consequences.

Note first that if the balance $\gbal<0$, then \eqref{ume3} shows that
the activity $\ba\cdot\bX_n$ becomes negative for large $n$; this is clearly
impossible and shows that the \dtime{} urn process must stop and cannot be
continued for ever. In the present paper, we are not interested in this
case, so we must have $\gbal\ge0$, and we thus assume this
in the sequel (whether it is said explicitly or not).
In this case, there are no problems.

\begin{lemma}\label{LB1}
  A balanced urn with balance $\gbal\ge0$ which satisfies \ref{A0} 
and \ref{A+} or (more generally) \ref{A-5} is well-defined for all
discrete or continuous times, and thus satisfies \ref{A-6}.
\end{lemma}
\begin{proof}
  By \eqref{ume3}, we have, for every $n$, \as{}
$\ba\cdot\bX_n\ge\ba\cdot\bX_0>0$.
Hence the \dtime{} process $\bX_n$ is well-defined, and also
$\ba\cdot\bX(t)>0$ for every $t$.
\end{proof}

We assume \ref{A0} and \ref{A+} or \ref{A-5} below; thus \ref{A-6} holds.

\begin{remark}\label{RB0}
  The case $\gbal=0$ is trivial. First, if \ref{A+} holds, so there are no
  subtractions, then the only possibility is $\xi_{ij}=0$ \as{} for all
  $i,j\in\sQ$; in other words, we draw from the urn but nothing happens. 
If we allow subtractions as in \ref{A-5b}, there may be a few
  initial draws that wipe out some colours, but nothing more will happen;
  furthermore, such urns violate \ref{A-7}.
\end{remark}

\begin{lemma}
  \label{LB}
If a \Polya{} urn is balanced, then the random processes $(\bX_n)\xoon$ and
$(\TT_n)\xoon$ are independent.
\end{lemma}

\begin{proof}
  Let $n\ge0$. 
Stop the \ctime{} process $\bX(t)$ at $\TT_n$, \ie, at the $n$th draw.
Conditionally on everything that has happened so far, 
the  waiting time $\TT_{n+1}-\TT_n$ until the next draw is an exponential
random variable with rate $\sum_i a_iX_{ni}$, which by \eqref{ume3} is the
constant $\ba\cdot\bX_0+n\gbal$.
Hence, $\TT_{n+1}-\TT_{n}$ is independent of $\bX_n$,
and thus also of the colour of the drawn ball, say $i$, 
and  of $\bX_{n+1}-\bX_n=\bxi_i\nn$, recall \eqref{xin}.
Consequently, the processes $(\TT_n)_n$ and $(\bX_n)_n$ evolve independently.
\end{proof}

\begin{lemma}  \label{LBT}
If a \Polya{} urn is balanced, then the 
distribution of the random sequence $(\TT_n)\xoon$
depends only on the balance $\gbal$ and the initial activity $\ba\cdot\bX_0$.
\end{lemma}
\begin{proof}
  As in the proof of \refL{LB}, $\TT_{n+1}-\TT_n$ is exponential distributed
with rate $\ba\cdot\bX_0+n\gbal$, and these waiting times are independent.
\end{proof}

\begin{remark}\label{RBT}
  The distribution of $(\TT_n)\xoon$ is thus the same as for the single-colour
urn with initial value $x_0:=\ba\cdot\bX_o$, activity $a_1=1$, and
replacement matrix $(\gbal)$. 
For $\gbal>0$, this is the more or less trivial case $q=1$ of
\refE{Eclassical}, and  it is well known that $(\gbal\TT_n)\xoo$
are the jump times of a Yule process started with $x_0/\gbal$ individuals.
\end{remark}

In the remaining part of the section we consider, as in the rest of the paper,
triangular urns.

\begin{lemma}\label{LBB}
  Consider a triangular balanced urn with balance $\gbal\ge0$, 
and suppose that the  urn satisfies \refAAX.
Then $\hgl=\gbal$.
Furthermore, if $\gbal>0$, then $\hkk=0$.
\end{lemma}

\begin{proof}
  Let $\sQ':=\set{i\in \sQ:a_i>0}$.
Then $\sQ'\neq\emptyset$ by \ref{A0}.
There are three possibilities for a colour $i\in\sQ$:
\begin{romenumerate}
  
\item \label{LBB1}
If $i\notin\sQ'$, then $a_i=0$ and thus $\gl_i=0$.
\item \label{LBB2}
If $i\in\sQ'$ and $i$ is maximal in $\sQ'$ for the order $\prec$, \ie,
$i\not\to j$ for every $j\in\sQ'$, 
then $\xi_{ij}=0$ \as{} when $j\in\sQ'\setminus\set i$, and
$a_j=0$ for every $j\notin\sQ'$; hence $a_j\xi_{ij}=0$ \as{} for all colours
$j\neq i$. Consequently,
\eqref{ume1} yields
\begin{align}
  a_i\xi_{ii} = \sum_{j\in\sQ} a_j\xi_{ij} = \gbal
\end{align}
a.s., and thus, taking the expectation,
\begin{align}
\gl_i=a_ir_{ii}=a_i\E\xi_{ii}=\gbal.  
\end{align}
\item \label{LBB3}
If $i\in\sQ'$ is not maximal in $\sQ'$, then there exists $j\in\sQ'$ with
$i\to j$ and thus $r_{ij}>0$.
Taking expectations in \eqref{ume1} yields
\begin{align}
  \gl_i=a_ir_{ii}=\gbal-\sum_{j\neq i} a_jr_{ij}
\le \gbal - a_jr_{ij}<\gbal.
\end{align}
\end{romenumerate}
It follows that $\hgl:=\max_i\gl_i=\gbal$.
Furthermore, if $\gbal>0$, then the maximum is attained precisely in Case
\ref{LBB2}, \ie, for $i$ that are maximal in $\sQ'$.
It follows that for every such $i$ we have $\gl_j<\gbal$ for all $j\in\sP_i$,
and thus $\kk_i=0$. Hence, $\hkk=0$.
\end{proof}

\begin{remark}\label{RBB}
The proof shows also that
if a balanced triangular urn has $\gbal>0$ and all activites $a_i=1$,
then $\xi_{ii}=r_{ii}=\gl_i=\gbal$ \as{} for every $i$ that is maximal in $\sQ'$, 
but
$\E\xi_{ii}=r_{ii}=\gl_i<\gbal$ for all other colours $i$.
\end{remark}

For urns with subtractions as in \refS{S-}, there are further
simplifications when the urn is balanced.

\begin{lemma}\label{LBB5}
  Consider a triangular balanced urn with balance $\gbal\ge0$,
and suppose that the
  urn satisfies \refAAZC.
Then, with notation as in \refS{SpfT1}, $\cX_0>0$ a.s.
\end{lemma}
\begin{proof}
The case $\gbal=0$ is trivial by \refR{RB0}, so we may assume $\gbal>0$.

For convenience, denote the urn by $\cU$.
\refL{LBT} and \refR{RBT} show that $\TT_n$ are (jointly) distributed as
for the one-colour urn $\cU_1$ with replacement matrix $(\gbal)$,
activity $a=1$, and the same
initial total activity. Hence we may couple the urns such that they have the
same $\TT_n$. (Or we may simply define the content of
$\cU_1$ as being the total activity in the urn $\cU$.)
The  urn $\cU_1$ obviously satisfies \refAA, and thus \eqref{hh3} applies to
it, with $\cX_0>0$ a.s. Moreover, \eqref{hh3} applies also to the original
urn $\cU$ by the same argument in \refS{SpfT1}
(as in the proof of \refTs{T1-} and \ref{T1--} in \refSS{Spf-});
note that the two urns $\cU$ and $\cU_1$ have the same $\hgl=\gbal$ and
$\kko=0$ by \refL{LBB} and \eqref{kk01}. 
Consequently, the two urns $\cU$ and $\cU_1$ have the same  $\cX_0$.
\end{proof}
In fact, in \refL{LBB5}, $\cX_0$ has a Gamma distribution by 
the proof and \eqref{jw0}.

\section{The drawn colours}\label{Sdraw}
We have so far studied $\bX_n$ and $\bX(t)$, the numbers of balls of each
colour in the urn. It is also of interest to study the number of times each
colour is drawn. 
(See \refEs{Epref}--\ref{Eprefk} for an application.)
For $i\in\sQ$,
we denote the number of times that a ball of colour $i$ is drawn up to time $n$
in the \dtime{} urn by $N_{ni}$, and the number of times up to time $t$ in
the \ctime{} urn by $N_i(t)$; thus
\begin{align}\label{tn0}
  N_{ni}=N_i(\TT_n),
\qquad i\in\sQ,\; n\ge0.
\end{align}

We state first a \ctime{} and then a \dtime{}  result;
both are similar to the results for $\bX(t)$ and $\bX_n$ earlier,
but note that exponents change when $\glx_i=0$.
(Proofs are given later in this section.)

\begin{theorem}  \label{TNC}
Let $(X_{i}(t))_{i\in\sQ}$ be a continuous-time
triangular \Polya{} urn satisfying 
either
\refAA, or (more generally) \refAAZC.
Let\/ $i\in\sQ$.
\begin{romenumerate}
\item \label{TNCa}
If $\glx_i>0$, then, as \ttoo,
\begin{align}\label{tnc1}
  t^{-\kk_i}e^{-\glx_i t} N_i(t) \asto \cN_i:=\frac{a_i}{\glx_i}\cX_i
.
\end{align}

\item\label{TNCb}
If $\glx_i=0$, then, as \ttoo,
\begin{align}\label{tnc2}
  t^{-\kk_i-1} N_i(t) \asto \cN_i:=\frac{a_i}{\kk_i+1}\cX_i
.
\end{align}
\end{romenumerate}
 Furthermore, if \ref{A+} or \ref{A-8} holds, 
then $\cN_i>0$ a.s.
\end{theorem}

\begin{theorem}\label{TN}
Let $(X_{i}(t))_{i\in\sQ}$ be a discrete-time
triangular \Polya{} urn 
and suppose that it satisfies \refAA.
Alternatively, suppose that the urn satisfies 
\refAAZ,  and also either satisfies \ref{A-8} or is balanced.
Let $i\in\sQ$.
\begin{romenumerate}
  
\item\label{TNa}
If $\glx_i>0$, then as \ntoo,
\begin{align}\label{tna}
  \frac{N_{ni}}{ n^{\glx_i/\hgl}\log^{\gam_i}n}
\asto \hcN_i
:=\frac{a_i}{\glx_i}\hcX_i
.\end{align}

\item \label{TNb}
If $\glx_i=0$ and $\hgl>0$, then as \ntoo,
\begin{align}\label{tnb}
  \frac{N_{ni}}{\log^{\kk_i+1}n}
\asto \hcN_i
:=\frac{a_i}{(\kk_i+1)\hgl}\hcX_i
.\end{align}

\item \label{TNc}
If $\glx_i=\hgl=0$, then  as \ntoo,
\begin{align}\label{tnc}
  \frac{N_{ni}}{n^{(\kk_i+1)/\hkko}}
\asto \hcN_i
:=\frac{a_i}{\kk_i+1}\cX_0^{-1/\hkko}\hcX_i
.\end{align}
\end{romenumerate}
 Furthermore, if \ref{A+} or \ref{A-8} holds, then $\hcN_i>0$ a.s.
\end{theorem}
\begin{remark}
  If we assume only \refAAZ{} in \refT{TN}, then the results hold on the
  event $\set{\cX_0>0}$, by the same proof and \refT{T1--}.
\end{remark}

We can also state the results as the following 
simple \as{} limit results for the ratios 
$N_{ni}/X_{ni}$ and $N_i(t)/X_i(t)$; 
in particular, if $\glx_i>0$ (the main case), these ratios
converge \as{} to a positive constant.

\begin{theorem}\label{TN2}
Let $(X_{i}(t))_{i\in\sQ}$ be a discrete-time
triangular \Polya{} urn 
and suppose that it satisfies \refAA{} or (more generally) \refAAZZ.
 Let\/ $i\in\sQ$.
\begin{romenumerate}
\item \label{TN2a}
If $\glx_i>0$, then
\begin{align}\label{tn2a}
  \lim_\ntoo \frac{N_{ni}}{X_{ni}}=
  \lim_\ttoo \frac{N_i(t)}{X_i(t)}=
\frac{a_i}{\glx_i}
\qquad\textrm{a.s.}
\end{align}
\item \label{TN2b}
If $\glx_i=0$, then, as \ttoo,
\begin{align}\label{tn2b}
  \frac{N_i(t)}{tX_i(t)}\asto \frac{a_i}{\kk_i+1}.
\end{align}
\item \label{TN2c}
If $\glx_i=0$ and $\hgl>0$, then, as \ntoo,
\begin{align}\label{tn2c}
  \frac{N_{ni}}{X_{ni}\log n}\asto \frac{a_i}{(\kk_i+1)\hgl}.
\end{align}
\item \label{TN2d}
If $\glx_i=0$ and $\hgl=0$, then, as \ntoo,
\begin{align}\label{tn2d}
 \frac{N_{ni}}{X_{ni}n^{1/\kko}}\asto \frac{a_i}{\kk_i+1}\cX_0^{-1/\hkko}.
\end{align}
\end{romenumerate}
\end{theorem}
Note that the limit in all cases is a strictly positive constant,
in case \ref{TN2d} by \refT{TD} (and \refR{Rdep-}).

\begin{proof}[Proof of \refT{TNC}]
We use (as in the proof of the corresponding result in \cite{SJ154})
dummy balls  similarly to the proof in \refS{SpfT1}, but now we use one
dummy ball for each colour.

It suffices to consider one colour at a time, so we fix $i\in\sQ$;
we assume $a_i>0$ since otherwise $N_{ni}=N_i(t)=0$ \as{} and the results
are trivial.
Denote the urn by $\cU$.
We consider one new colour, which we denote by $\iota$, and let
$\sQp:=\sQ\cup\set\iota$ be the new set of colours.
Balls of colour $\iota$ have activity $a_\iota:=0$ and are thus never drawn,
and we let $\xi_{\iota j}:=0$ for all $j\in\sQp$;
we further let $\xi_{i\iota}:=1$ and $\xi_{j\iota}:=0$ for every $j\neq i$,
and we start with $X_{0\iota}:=0$.

Consequently, the new urn, which we denote by $\cUp$,
differs from the old one $\cU$ only in that one
additional ball of colour $\iota$ is added each time a ball of colour $i$ is
drawn.
We may thus assume that the two urns are coupled such that they have the
same $X_j(t)$  for all $j\in\sQ$ and $t\ge0$, and then 
\begin{align}
  \label{qk}
N_i(t)=X_\iota(t) \qquad\text{and}\qquad N_{ni}=X_{n\iota}.
\end{align}
The new urn $\cUp$ is also triangular and satisfies \refAA, or \refAAZC, if
$\cU$ does.
In $\cUp$ we have $\gl_\iota=0$ and,
using \eqref{glx} and \eqref{kk}, 
\begin{align}\label{jkk0}
\glx_\iota&=\glx_i,
\\
\label{jkk}
  \kk_\iota&=
  \begin{cases}
    \kk_i, & \glx_i>0,
\\
\kk_i+1, & \glx_i=0.
  \end{cases}
\end{align}

\pfitemref{TNCa}
By \refT{TC} or \refT{TC-} applied to the new urn $\cUp$, 
we have \eqref{tc1} for
the colour $\iota$, and thus, using \eqref{qk} and \eqref{jkk0}--\eqref{jkk},
\begin{align}\label{job1}
t^{-\kk_i}e^{-\glx_i t} N_i(t)
= t^{-\kk_\iota}e^{-\glx_\iota t} X_\iota(t)
\asto\cX_\iota.
\end{align}
Furthermore, \refL{L12} or \refL{ZL12}\ref{ZL12a}
applies to $\iota\in\sQp$ (with $i$ replaced by $\iota$ and
$j$ by $i$), and thus \eqref{ji3} yields, since $\gl_\iota=0<\glx_i$,
\begin{align}\label{job2}
\cX_\iota=\frac{a_ir_{i\iota}}{\glx_i-\gl_\iota}\cX_i=\frac{a_i}{\glx_i}\cX_i.
\end{align}
Hence, \eqref{tnc1} follows, with $\cN_i:=\cX_\iota$.

\pfitemref{TN2b}
Similar, now with $\kk_\iota=\kk_i+1$ by \eqref{jkk}, and using the second
alternative in \eqref{ji3} which gives
\begin{align}\label{job20}
\cN_i:=\cX_\iota=\frac{a_ir_{i\iota}}{\kk_\iota}\cX_i=\frac{a_i}{\kk_i+1}\cX_i.
\end{align}

The final sentence follows from \refTs{TC} and \ref{TC-}, which yield
$\cX_\iota>0$ a.s.
\end{proof}

\begin{proof}[Proof of \refT{TN}]
We add a dummy colour $\iota$ as in the proof of \refT{TNC}, and note that
if the original urn $\cU$ satisfies \ref{A-8} or is balanced, then the same
holds for the new urn $\cUp$. 
(Note that dummy colours with activity 0 are ignored in \refD{DB}.)

By \refT{T1}, \refT{T1-}, or \refT{T1--} together with \refL{LBB5},
the conclusions 
of \refT{T1} hold for $\cUp$,
except that $\hcX_\iota=0$ is possible unless we have \ref{A-5} or \ref{A-8}.
(However, our assumptions yield $\cX_0>0$ \as{} in all cases.) 

We argue similarly to the proof of \refT{TNC}, now using \eqref{t1b} or
\eqref{t1c} for the dummy colour $\iota$ in $\cUp$,
together with \eqref{qk} and \eqref{jkk0}--\eqref{jkk}.
This yields  
\as{} convergence to $\hcN_i:=\hcX_\iota$ in \eqref{tna}, \eqref{tnb}, or
\eqref{tnc} (depending on $\glx_i$ and $\hgl$);
note that $\hgl\ge\glx_i$,  
that \eqref{gam} yields $\gam_\iota=\gam_i$ in \ref{TNa} and
$\gam_\iota=\kk_\iota=\kk_i+1$ in \ref{TNb},
and that \eqref{hkk0} shows that $\hkko$ is the same for $\cUp$ as for $\cU$.

Finally, the formulas for $\hcN_i$ follow from
\eqref{job2}--\eqref{job20} 
and
\eqref{hcx} or \eqref{hh11}.
\end{proof}

Alternatively, we may prove \refT{TN} from \refT{TNC} by adding a dummy
colur 0 as in \refS{SpfT1}. In any case, the proof is really based on adding
two dummy colours 0 and $\iota$ to the \ctime{} urn.

\begin{proof}[Proof of \refT{TN2}]
The results for the \ctime{} urn in \ref{TN2a} and \ref{TN2b} follow 
by comparing the results of \refT{TNC} 
with the limits for $X_{i}(t)$ in \refT{TC} or \refT{TC-},
recalling that $\cX_i>0$ \as{} as stated in \refT{TC}.
The result for $N_{ni}/X_{ni}$ in \ref{TN2a}
then follows by \eqref{tn0}  and \eqref{ct}.
Similarly, \eqref{tn2c} and \eqref{tn2d} follow from \eqref{tn2b} and
\eqref{hh6} or \eqref{hh10}.
(Alternatively, the results for $N_{ni}/X_{ni}$ follow by comparing the
results in \refT{TN} and \refT{T1}.)
\end{proof}

\begin{remark}\label{RN2}
  In \refT{TN2}, if we assume only \refAAZC{} but not \ref{A-8}, 
then the  conclusions hold on the event 
$\set{\hcX_i>0}$ for $N_i(t)$ and on
$\set{\hcX_i>0,\cX_0>0}$ for $N_{ni}$.
\end{remark}

\section{Moment convergence}\label{Smoments}

We have so far considered convergence \as, which as always implies
convergence in distribution. 
We consider in this section whether we also have convergence of moments, or
equivalently convergence in $L^p$;
recall the general fact that for a sequence of positive random variables
converging \as{} (as we have here), 
convergence of the $p$th moment is equivalent to convergence in $L^p$, for
any (real) $p>0$, see \eg{} \cite[Theorem 5.5.2]{Gut}.

Unlike earlier sections, there seems to be an important
difference between the \dtime{} and \ctime{} cases.

For a continuous-time urn,
we will see below (\refT{TMC2}) that
we always have convergence in $L^2$
in \refTs{TC} and \ref{TC-};
hence the mean and variance converge in these results.
This  extends to convergence 
in $L^p$ for any $p>2$, and thus convergence of any moment,
assuming a corresponding moment condition for the replacements $\xi_{ij}$
(\refT{TMCp}).

The situation for \dtime{} urns is more complicated.
We will only consider balanced urns, and then 
prove $L^2$-convergence, and thus convergence of mean and variance;
we also extend this to $L^p$ and higher moments under the             
corresponding moment condition for the replacements $\xi_{ij}$
(\refT{TMD}).
It seems likely that this result extends to a wider class of triangular urns.
However, it does \emph{not} hold for all triangular urns.
\refE{ED} gives a simple example where the \as{}
limit does not have a finite mean,
and thus we cannot have even $L^1$-convergence in \refT{T1}.
\begin{remark}
The counterexample in \refE{ED} is rather special (a diagonal urn); 
\cite[Theorem 1.6]{SJ169} shows that for a class of more typical 
unbalanced triangular urns, the \as{} limit in \refT{T1} has moments of all
orders. This does not prove moment convergence, but we see no reason against it,
and we conjecture that for these urns, and many others,
we have convergence in $L^p$ for all $p>0$
(\refP{Pmom2}).
\end{remark}

$L^p$-convergence or moment convergence (usually
as part of a proof of convergence in distribution by the method of moments)
have been proved earlier by different methods
for some \dtime{} \Polya{} urns, as far as we know
all of them balanced.
This includes balanced triangular urns with 
$q=2$ or 3 and
deterministic integer-valued replacements
by  \citet{F:exact}, \cite{Puy} (see \refEs{E2} and \ref{E3}),
and, for $L^2$ only, more general
balanced triangular urns with deterministic  replacements by \citet{BoseDM}
(\refE{EBose}).
Further examples where moment convergence has been shown earlier 
are discussed in
\refEs{E2p1}, 
\ref{Epref}, 
\ref{Epref-},
and
\ref{ERW}.

\subsection{Continuous-time urns}

\begin{theorem}\label{TMC2}
  In \refTs{TC} and \ref{TC-},
the limit \eqref{tc1} holds also in $L^2$.
In particular, mean and variance converge.
\end{theorem}

\begin{proof}
    Let $i\in\sQ$.
The proofs of \refTs{TC} and \ref{TC-} show that \eqref{g4} holds. 
Thus, by the definitions \eqref{j4}--\eqref{j4*}, 
\begin{align}
\sup_{t\ge1}  \bigabs{\tX_i(t)}^2
\le \bigabs{2^{\kk_i}\tXX_i}^2 \in L^1.
\end{align}
Hence, the collection $\set{ \abs{\tX_i(t)}^2:t\ge1}$ is uniformly integrable,
and consequently the \as{} convergence $\tX_i(t)\to\cX_i$ implies
convergence in $L^2$
\cite[Theorems 5.4.4 and 5.5.2]{Gut}.
\end{proof}

\refT{TMC2} extends to $L^p$ for all $p\ge2$ as follows.
The proof is based on the arguments in \refS{S1} combined with the
Burkholder--Davis--Gundy inequalities which enable us to replace $L^2$
estimates by $L^p$ estimates. The details are quite long, however, so we
postpone them to \refApp{ALp}.

\begin{theorem}\label{TMCp}
Let $p\ge2$.
In \refTs{TC} and \ref{TC-},
assume also $\xi_{ij}\in L^p$ for all $i,j\in\sQ$.
Then, for every $i\in\sQ$, 
the limit \eqref{tc1} holds also in $L^p$.
Moreover, $\tXX_i\in L^p$.
\end{theorem}

\begin{remark}\label{Rmomp}
In \refApp{ALp}, we further show that \refT{TMCp}
holds for any $p>1$, also without the $L^2$ condition \ref{A2}.
The proofs below then show that the same holds 
for all $L^p$ results in this section.
\end{remark}

\subsection{Balanced discrete-time urns}

\begin{theorem}\label{TMD}
Consider a balanced triangular urn statisfying \refAA{} or \refAAZC.
\begin{romenumerate}
  
\item \label{TMD1}
In \refTs{T1} and \ref{T1-}--\ref{T1--},
the limit \eqref{t1b} 
or \eqref{t1c}
holds also in $L^2$.
In particular, mean and variance converge.
\item \label{TMD2}
Moreover, if $p\ge2$ and $\xi_{ij}\in L^p$ for all $i,j\in\sQ$,
then, for every $i\in\sQ$, 
the limit \eqref{t1b} or \eqref{t1c} holds also in $L^p$.
Hence all moments of order $\le p$ converge.
\end{romenumerate}
\end{theorem}

As said above,
this has been shown earlier in some special cases and examples,
in particular \cite{F:exact}, 
\cite{Puy} ($q=2,3$), and
\cite{BoseDM} ($p=2$);
see also the examples in \refS{Sex}.

\begin{proof}
Part \ref{TMD1} is a special case of \ref{TMD2}, so we show only the latter.

Note first that \ref{A-6} holds by \refL{LB1}, 
and that \refL{LBB} shows that $\cX_0>0$ a.s.;
thus the conclusion \eqref{t1b}  or \eqref{t1c} of \refT{T1} holds
by \refT{T1} or  \ref{T1--}. (Or \refT{T1-} when it is applicable.)

Let $c>0$ be so small that $\P(X_0>c)>\frac12$.
Our assumptions imply (as in the proof of \refTs{T1-} and \ref{T1--})
that \eqref{hh3} holds, by the proof in \refS{SpfT1}.
This implies, by our choice of $c$,
\begin{align}\label{ume6}
\P\Bigpar{\TT_n^{-\kk_0}e^{-\hgl \TT_n} n >c}>\frac12
\end{align}
for all large $n$. By decreasing $c$ 
(or just by ignoring some small $n$ in the sequel),
we may assume that \eqref{ume6} holds for all $n\ge1$.
Let $\cE_n$ denote the event 
\begin{align}\label{cen}
  \cE_n:=
\bigset{\TT_n^{-\kk_0}e^{-\hgl \TT_n} n >c},
\end{align}
so that \eqref{ume6} reads $\P(\cE_n)>\frac12$.

Let $i\in\sQ$. 
By \refT{TMCp}, $\tXX_i\in L^p$.
Hence, by \eqref{ct} and \eqref{j4*}, 
\begin{align}\label{ume7}
Y_n:=  \lrabs{\frac{X_{ni}}{\bigpar{1+\TT_n}^{\kk_i} e^{\glx_i\TT_n}}}^p
=  \lrabs{\frac{X_{i}(\TT_n)}{\bigpar{1+\TT_n}^{\kk_i} e^{\glx_i\TT_n}}}^p
\le (\tXX_i)^p\in L^1.
\end{align}
Consequently, the sequence $Y_n$ is \ui.
It follows from this and \eqref{ume6} that the sequence of conditioned
random variables
$(Y_n\mid\cE_n)$
also is \ui.
We consider two cases:

\pfCaseY1{$\hgl>0$}
In this case, \refL{LBB} shows, using \eqref{kk01} and \eqref{gam}, 
that $\kko=\hkk=0$
and   $\gam_i=\kk_i$.
Hence, the event  $\cE_n$ means
${e^{-\hgl \TT_n} n >c}$, 
and thus $\TT_n \le t_n:=\hgl\qw (\log n+C)$.
Consequently, 
on the event $\cE_n$ we have (for $n\ge2$)
\begin{align}\label{ui1}
Z_n:=
\lrabs{\frac{X_{ni}}{ n^{\glx_i/\hgl}\log^{\gam_i} n }}^p
=
\lrabs{\frac{X_{ni}}{ n^{\glx_i/\hgl}\log^{\kk_i} n }}^p
\le
C\lrabs{\frac{X_{ni}}{t_n^{\kk_i} e^{\glx_i t_n}}}^p 
\le C Y_n
\end{align}
and thus
the uniform integrability of $(Y_n\mid\cE_n)$ 
implies uniform integrability of 
$(Z_n\mid\cE_n)$.
However, by \refL{LB} and \eqref{cen}, $X_{ni}$ is
independent of the event $\cE_n$; hence, so is $Z_n$ and thus
$(Z_n\mid\cE_n)\eqd Z_n$.
Consequently,  the sequence $Z_n$ is \ui.

In other words,
the \lhs{} of \eqref{t1b} is uniformly
$p$th power integrable. Hence, the \as{} convergence in \eqref{t1b} implies
convergence also in $L^p$.

\pfCaseY2{$\hgl=0$}
This is similar.
In this case, \eqref{cen} means $\TT_n\le C n^{1/\kko}$.
We now define 
\begin{align}\label{ui2}
Z_n:=\lrabs{\frac{X_{ni}}{n^{\kk_i/\hkko}}}^p,
\end{align}
and note again that $Z_n$ is independent of $\cE_n$.
Since $0\le \glx_i\le\hgl=0$ we have $\glx_i=0$, and 
$\hkko=\kko$ by \eqref{kk0}; hence \eqref{ui2} and \eqref{ume7} show that
on $\cE_n$, we have $Z_n\le C Y_n$.
Consequently, we have again
$Z_n\eqd(Z_n\mid\cE_n) \le (CY_n\mid\cE_n)$, and it follows
that $Z_n$ is uniformly integrable.
Hence the \as{} convergence \eqref{t1c} holds also in $L^p$.
\end{proof}

As said above,
\refE{ED} shows that \refT{TMD} does not extend to all  triangular
urns.
However, it seems likely that it extends to many unbalanced urns; we leave
this as an open problem.

\begin{problem}\label{Pmom2}
 Find more general conditions (including also some unbalanced urns)
for convergence in  $L^2$ or $L^p$ in \refTs{T1} and \ref{T1-}.
\end{problem}

\subsection{Moments for drawn colours}
The results above on convergence in $L^2$, and thus convergence of mean and
variance, apply to the number of drawn balls with a given colour, $N_{ni}$
and $N_i(t)$, since as shown in the proofs in \refS{Sdraw}, they can be
regarded as $X_{n\iota}$ and $X_\iota(t)$ for an extended urn with a dummy colour
$\iota$ added. 
Hence we obtain:
\begin{theorem}\label{TMQC} 
In \refT{TNC}, the \as{} limit \eqref{tnc1} or \eqref{tnc2} 
holds also in $L^2$.
Moreover, if $p\ge2$ and $\xi_{ij}\in L^p$ $\forall i,j\in\sQ$, then the
limit holds also in $L^p$.
\end{theorem}
\begin{proof}
  By \refT{TMC2} or \ref{TMCp} applied to $X_\iota(t)$.
\end{proof}

\begin{theorem}\label{TMQ} 
In \refT{TN}, if the urn is balanced, then
the \as{} limit \eqref{tna}, \eqref{tnb}, or \eqref{tnc} 
holds also in $L^2$.
Moreover, if $p\ge2$ and $\xi_{ij}\in L^p$ $\forall i,j\in\sQ$, then the
limit holds also in $L^p$.
\end{theorem}
\begin{proof}
  By \refT{TMD} applied to $X_{n\iota}$.
\end{proof}

\section{Rates of convergence?}\label{Smer}
For classical \Polya{} urns
(\refE{Eclassical}),
the rate of convergence for convergence in distribution in \eqref{jw2} or
\eqref{limbeta} has been studied, for several different metrics; see
\cite{SJ345} and the references there.
As noted in \cite[Remark 1.4]{SJ345}, the rate of a.s.\ convergence is
slower, and is the same as in the law of large numbers for \iid{} Bernoulli
variables, which is given by the law of iterated logarithm.

For other triangular urns,
we are not aware of any similar results on rates of convergence; 
however, \cite{F:exact} gives upper bounds  
for the rate of convergence of moments and in a local limit theorem,
for some balanced triangular urns with deterministic replacements.
(Irreducible, and thus non-triangular, balanced urns with $q=2$
and deterministic replacements
are studied in \cite{KuNeininger}.)

\begin{problem}
  Study rates of convergence in \eg{} \eqref{t1b} and \eqref{tc1},
both for the \as{} convergence and for convergence in distribution.
\end{problem}

Note that this problem
is closely related to the problem
studying fluctuations from the limit,  mentioned in \refR{Rfluct}.

\section{Examples}\label{Sex}

We consider several examples, many of which have been treated earlier from
different perspectives.
The purpose is to illustrate both the theorems above and (some of) 
their relations to earlier literature.
We generally label the colours by $1,\dots,q$, and then
assume $\xi_{ij}=0$ when $i<j$.
We assume that the initial composition $\bX_0$ is deterministic.
We write for convenience $x_i:=X_{i0}=X_i(0)$
and $\bx:=(x_i)_1^q=\bX_0$. 
We denote the total number of balls in the urn after $n$ draws by
$|\bX_n|:=\sumiq X_{ni}$.

When $q=2$ we sometimes also call the
colours \emph{white} and \emph{black} (in this order;
thus a black draw may add only black balls:
``there is no escape from a black hole'').
We then may write
$W_n:=X_{n1}$, $B_n:=X_{n2}$, $W(t):=X_1(t)$, $B(t):=X_2(t)$,
$w_0:=x_1=X_{10}$, $b_0:=x_2=X_{20}$. 

We usually describe the urns using the replacement matrix $(\xi_{ij})_{i,j=1}^q$
where the rows are the replacement vectors. 
(See \refR{Rmatrix}.)

In all our examples, all activities $a_i=1$. 
Thus
\begin{align}
\gl_{i}=r_{ii}=\E\xi_{ii}, \qquad i\in\sQ=\setq.  
\end{align}

\begin{example}\label{Eclassical}
The classical \Polya{} urn has balls of $q$ colours; when a ball is drawn it is
replaced together with a fixed number $b>0$ balls of the same colour.
Hence, the replacement matrix is deterministic and diagonal, with entries $b$
on the diagonal. This urn is obviously balanced, and we have
$\gl_i=\glx_i=b$ and $\kk_i=0$ for every colour $i$.

This urn model was studied (for $q=2$)
already by \citet{Markov1917},
\citet{EggPol} and \citet{Polya}.
See also \eg{}
\citet[Chapter 4]{JohnsonKotz} and \citet{Mahmoud}.

For this urn (as for any diagonal urn),
in the continuous-time version, the different colours evolve independently,
and each colour is version of the Yule process.
More precisely, $X_i(t)/b$ is a Yule process started with $x_i/b$
individuals, where each individual gets children at rate $b$;
thus $X_i(t/b)/b$  is a Yule process with the standard rate 1.
(This is a classical branching process if $x_i/b$ is an integer, and in
general a CB process.)
It is well-known that in this case $e^{-bt}X_i(t)/b\dto \gG(x_i/b,1)$
and thus 
\begin{align}\label{jw0}
e^{-bt}X_i(t)\asto\cX_i\in\gG(x_i/b,b);   
\end{align}
furthermore,
$\cX_1,\dots,\cX_q$ are independent, since the processes $X_i(t)$ are
independent. 
This is an example of \refT{TC}. Moreover, \eqref{hh8}--\eqref{hcx}, 
or \eqref{el1} where now $\sQx=\sQ$,
show together with \eqref{el3} that 
\begin{align}\label{jw1}
  \frac{X_{in}}n\asto\hcX_i:=b\frac{\cX_i}{\sumjq\cX_j}.
\end{align}
It follows 
that
the vector of proportions converges:
\begin{align}\label{jw2}
  \frac{\bX_{n}}{|\bX_n|}
\asto \frac{1}{b}\hbcX
:=\frac{1}{b}\bigpar{\hcX_1,\dots,\hcX_q}
=\frac{\xpar{\cX_1,\dots,\cX_q}}{\sumjq\cX_j},
\end{align}
where, as a consequence of \eqref{jw0}, the limit vector $b\qw\hbcX$
has a Dirichlet distribution
with parameter $\bx/b$.
In particular, each marginal converges \as{} to a Beta distributed variable:
\begin{equation}\label{limbeta}
\frac{X_{ni}}{|\bX_n|}\asto 
b\qw\hcX_i\sim
B\Bigpar{\frac{x_i}b,\sum_{j\neq i}\frac{x_j}{b}}.
\end{equation}

These results are all well known; the limit \eqref{limbeta} with convergence
in distribution was shown for $q=2$ already in 
\cite{Markov1917} and \cite{Polya},
and for general $q$ 
in \cite{BlackwellKendall1964} (in a special case)
and \cite{Athreya1969};
see also \cite[Section 6.3.3]{JohnsonKotz}  
and \cite{Mahmoud}.
Furthermore,
a.s.\ convergence has been shown by a number of
methods, for example in
\cite{BlackwellKendall1964} 
and \cite{Athreya1969}. 
\end{example}

\begin{example}\label{ED}
A \emph{diagonal} \Polya{} urn has $\xi_{ij}=0$ for $i\neq j$; in other
words, all added balls have the same colour as the drawn ball. 
This is a generalization of the classical \Polya{} urn in \refE{Eclassical},
but now the
diagonal elements $\xi_{ii}$ can be random, and they may have different
distributions.
A.s.\ convergence for this urn
has been shown, under weak technical
conditions, by \citet{Athreya1969};
see also \citet[Theorem 4]{Aguech}.

Consider for simplicity the case when the replacements are deterministic,
and assume to avoid trivialities that $r_{ii}=\xi_{ii}>0$ for every $i\in\sQ$.

As in \refE{Eclassical}, in the continuous-time urn, the colours evolve as
independent Yule processes, now with possibly different rates $\gl_i=\xi_{ii}$.
Hence, generalizing \eqref{jw0},
\begin{align}\label{sw0}
e^{-\gl_i t}X_i(t)\asto\cX_i\in\gG(x_i/\gl_i,\gl_i),   
\end{align}
with all $\cX_i$ independent.

Consider the simplest case: $q=2$, and assume 
$\gl_1=\ga$ and $\gl_2=\gd$ with $\ga>\gd>0$.
Then $\glx_i=\gl_i$, $\hgl=\ga$, $\kk_i=\hkk=\gam_i=0$ ($i=1,2$).
It follows from \eqref{hh8}--\eqref{hcx} and \eqref{el3} (or directly from
\eqref{sw0}) 
that 
\begin{align}\label{sw1}
n^{-\gd/\ga} X_{n2}\asto
  \hcX_2=\ga^{\gd/\ga}\frac{\cX_2}{\cX_1^{\gd/\ga}}.
\end{align}

Note that if $\cX\sim\gG(a,b)$, then its moments (for arbitrary real $r$)
are given by
\begin{align}\label{sw2}
  \E\cX^r = 
  \begin{cases}
b^r\gG(a+r)/\gG(a)<\infty, & -a<r<\infty,
\\
\infty, &r\le     -a.
  \end{cases}
\end{align}
In particular, since $\cX_1$ and $\cX_2$ are independent,
it follows from \eqref{sw0}--\eqref{sw2} that for $r>0$,
\begin{align}\label{sw3}
  \E\hcX_2^r<\infty 
  \iff   \E\cX_1^{-r\gd/\ga}<\infty 
  \iff r\gd/\ga < \xfrac{x_1}{\ga}
  \iff r < \xfrac{x_1}{\gd}.
\end{align}
Consequently, $\hcX_2$ does not have finite moments of all orders.
In particular, we cannot always have (finite) moment convergence in \eqref{sw1}.
Taking, for example, $\ga=2$, $\gd=1$, and $x_1=x_2=1$ we see that not even
the mean $\E\hcX_2$ is finite; hence we cannot have convergence in $L^1$ or
$L^2$ in \refT{T1}.

As far as we know, it is an open problem to find asymptotics of moments 
$\E X_{n2}^r$ for general $r>0$ in this (simple) example.
\end{example}

\begin{example}\label{E2}
  Consider a two-colour urn with a deterministic replacement matrix
  \begin{align}
\matrixx{\gd&\gam\\0&\ga}.     
  \end{align}
(We have chosen a notation agreeing with
\cite{SJ169}, although colours are taken in different order there and thus
the matrices are written differently.)
This urn (and special cases of it) have been studied in many papers; in
particular, \cite{SJ169} gives a detailed study of limits in distribution.
The balanced case $\ga=\gd+\gam$ with integers $\ga,\gam,\gd$
is studied by very different methods 
(generating functions)
in \cite{Puy} and
\cite{F:exact}. 
A.s.\ convergence has been shown in special cases in
\cite{Gouet89,Gouet93,BoseDM1,BoseDM} ($\ga=\gd+\gam$), and
\cite{Aguech} ($\ga\ge\gd$).

Suppose that $\gd>0$, $\gam>0$, $w_0=x_1>0$, and $b_0=x_2\ge0$;
suppose also either $\ga\ge0$, or $\ga=-1$ together with
$\gam\in\bbZ_+$ and $x_2\in\bbZgeo$.
Then the urn satisfies 
\refAA{} if $\ga\ge0$,
and \refAAZZ{} for all  $\ga$.
Hence,
\refTs{T1} and \ref{TC} apply if $\ga\ge0$, and 
\refTs{T1-} and \ref{TC-} apply for any $\ga$;
consequently, the conclusions of 
\refTs{T1} and \ref{TC} hold for all cases.

We have $\gl_1=\gd$ and $\gl_2=\ga$.
Furthermore, since $\gam>0$, we have $1\to2$.
(Thus 1 is the only minimal colour.)
Hence, $\glx_1:=\gd$, $\glx_2:=\ga\bmax\gd$ and thus
$\hgl=\glx_2=\ga\lor\gd$; 
furthermore $\kk_1=0$ while
$\kk_2=1$ when $\ga=\gd$ and $\kk_2=0$ otherwise.
We consider several cases.

\xcase{$\ga<\gd$}\label{E2<} 

Then $\glx_1=\glx_2=\hgl=\gd>0$; 
furthermore, $\kk_1=\kk_2=0=\hkk$, and \eqref{gam} yields
$\gam_1=\gam_2=0$. Consequently, \refT{T1}\ref{T1+} or \refT{T1-}
yields
\begin{align}\label{e20}
\frac{X_{ni}}n\asto\hcX_i, \qquad i=1,2. 
\end{align}
Furthermore, 1 is the only leader, and thus \refT{Tone} shows that 
$\hcX_1$ and $\hcX_2$
are constants. To find them, we can use \refL{LC}. By \refL{LNDB}
(simplifying the notation), $\cX_i=c_i\cX_1$, where obviously $c_1=1$, 
and \refL{LC} gives the eigenvalue equation
\begin{align}\label{e2a}
\xpar{c_1,c_2}\matrixx{\gd&\gam\\0&\ga}
= \gd \xpar{c_1,c_2},
\end{align}
i.e., $\gam+\ga c_2=\gd c_2$, 
with the solution
$c_2=\xqfrac{\gam}{\gd-\ga}$. In other words,
\begin{align}\label{e2b}
  \cX_2=\frac{\gam}{\gd-\ga}\cX_1.
\end{align}
This follows also directly from \refL{L12} and \eqref{ji3}.

If we add a dummy colour 0 as in \refS{SpfT1}, then  \eqref{el3} and
\eqref{e2b}  yield
\begin{align}\label{e2c}
  \cX_0=\gd\qw(\cX_1+\cX_2)=\frac{1}{\gd}\frac{\gam+\gd-\ga}{\gd-\ga}\cX_1.
\end{align}
Hence, by \eqref{hcx},
\begin{align}\label{e2d}
  \hcX_1&=\frac{\cX_1}{\cX_0}=\frac{\gd(\gd-\ga)}{\gam+\gd-\ga},
\\\label{e2e}
  \hcX_2&=\frac{\cX_2}{\cX_0}=\frac{\gd\gam}{\gam+\gd-\ga}.
\end{align}
Consequently, as \ntoo{} we have
\begin{align}\label{e2f}
\frac{X_{n1}}{n}&\asto \hcX_1= \frac{\gd(\gd-\ga)}{\gam+\gd-\ga},
\\\label{e2g}
\frac{X_{n2}}{n}&\asto\hcX_2=\frac{\gd\gam}{\gam+\gd-\ga}.
\end{align}
This is in agreement with \cite[Theorem 1.3(i)-(iii) and Lemma 1.2]{SJ169}, 
which give the asymptotic distribution of the difference $X_{ni}-n\hcX_i$
divided by the correct normalization factor,
which implies
(and is much more precise than) 
convergence in probability in \eqref{e2f}--\eqref{e2g}.
(The normalization factor is 
$n\qq$ for $\ga<\gd/2$, $(n\log n)\qq$ for $\ga=\gd/2$ and
$n^{\ga/\gd}$ for $\gd/2<\ga<\gd$. Moreover,
the distribution is
asymptotically normal for $\ga\le\gd/2$, but not for $\gd/2<\ga<\gd$.
See \cite{SJ169} for details.)

\xcase{$\ga=\gd$}\label{E2=}

Then $\glx_1=\glx_2=\hgl=\gd>0$; 
furthermore, $\kk_1=0$ and $\kk_2=1=\hkk$, and thus
\eqref{gam} yields
$\gam_1=-1$ and $\gam_2=0$. Consequently, 
\refT{T1}\ref{T1+}  yields
\begin{align}\label{e22a}
  \frac{X_{n1}}{n/\log n}&\asto \hcX_1,
\\\label{e22b}
  \frac{X_{n2}}{n}&\asto \hcX_2.
\end{align}
Furthermore, also in this case, 1 is the only leader, and thus \refT{Tone}
shows that  $\hcX_1$ and $\hcX_2$ are constants. 
However, unlike \refCase{E2<}, $\gl_2=\glx_2$ and thus 2 is now a subleader.
Again,  \refL{LNDB} shows that $\cX_2=c_2\cX_1$, where \eqref{subl} in
\refL{LC} immediately yields
\begin{align}\label{e22c}
  c_2=\frac{a_1r_{12}}{\kk_2}c_1=\gam.
\end{align}
Furthermore, \eqref{el2} now yields $\sQx=\set{2}$, and thus \eqref{el3}
yields
\begin{align}\label{e22d}
\cX_0=\hgl\qw\cX_2
=\gd\qw\cX_2
.\end{align}
Consequently, recalling \eqref{hcx}, the limits in \eqref{e22a}--\eqref{e22b}
are
\begin{align}\label{e22e}
  \hcX_1&=\hgl\frac{\cX_1}{\cX_0}
=\gd^2\frac{\cX_1}{\cX_2}
=\frac{\gd^2}{c_2}
=\frac{\gd^2}{\gam},
\\\label{e22f}
\hcX_2&=\frac{\cX_2}{\cX_0}=\gd.
\end{align}
This is in agreement with the result on the asymptotic distribution
in \cite[Theorem 1.3(iv) and Lemma 1.2]{SJ169}, 
which implies convergence in probability in \eqref{e22a}--\eqref{e22b}.

\xcase{$\ga>\gd$}\label{E2>}

Then $0<\glx_1=\gd<\glx_2=\ga=\hgl$;
furthermore, $\kk_1=\kk_2=0=\hkk$, and thus
$\gam_1=\gam_2=0$. Consequently, 
\refT{T1}\ref{T1+}  yields
(see also \cite{Aguech})
\begin{align}\label{e23a}
  \frac{X_{n1}}{n^{\gd/\ga}}&
\asto \hcX_1,
\\\label{e23b}
  \frac{X_{n2}}{n}&
\asto \hcX_2.
\end{align}
In this case, both colours 1 and 2 are leaders; hence \refT{Tac} shows that
$\cX_1$ and $\cX_2$ are \abscont, also jointly.
Furthermore, \refT{TEL}\ref{TEL<} shows that $\hcX_1$ is \abscont,
while \refT{Tone} shows that $\hcX_2$ is deterministic.
We have again $\sQx=\set2$, which by \eqref{el3} now yields
\begin{align}\label{e23d}
\cX_0=\hgl\qw\cX_2
=\ga\qw\cX_2
.\end{align}
Hence, \eqref{hcx} yields
\begin{align}\label{e23e}
  \hcX_1&=\frac{\cX_1}{\cX_0^{\gd/\ga}}
=\ga^{\gd/\ga}\frac{\cX_1}{\cX_2^\gda}
,\\\label{e23f}
\hcX_2&=\frac{\cX_2}{\cX_0}=\ga.
\end{align}
However, the formula \eqref{e23e} for $\hcX_1$ does not seem to be of
much use to find the distribution of $\hcX_1$.
This distribution was found by other methods in 
\cite[Theorem 1.3(v)]{SJ169}, which yields convergence in distribution of
$X_{n1}/n^{\gd/\ga}$; the \as{} convergence in \eqref{e23a} is a stronger
result, and the distribution of $\hcX_1$ is 
thus the limit distribution found in 
\cite{SJ169}.
This limit distribution is characterized in \cite{SJ169}, but no simple form
is known in general; it is shown in \cite[Theorem 1.6]{SJ169} that
$\hcX_1$ has moments of all orders, and a complicated integral formula is
given for these moments.
Except in the balanced case below, it is, as far as we know, an open problem
whether moments converge in \eqref{e23a} (to these limits) or not.

\xcase{$\ga=\gd+\gam$} \label{E2b}

The balanced case of the two-colour urn studied here
is $\ga=\gd+\gam$; by our assumption $\gam>0$, this is a special case of
Case \ref{E2>}, and thus \eqref{e23a}--\eqref{e23f} hold.
In this special case, $\hcX_1$ can be characterized by its moments, for
which there is a simple formula,
see
(for integers $\ga,\gam,\gd,r$)
\cite[Theorem 2.9]{Puy}, 
\cite[Proposition 17]{F:exact},
and (for the general case)
\cite[Theorem 1.7]{SJ169}:
\begin{align}\label{e24a}
    \E\hcX_1^r
=
\gd^r\frac{\Gamma\bigpar{(x_1+x_2)/\ga}\Gamma(x_1/\gd+r)}
 {\Gamma(x_1/\gd)\Gamma\bigpar{(x_1+x_2+r\gd)/\ga}},
\qquad r>0. 
\end{align}
It follows that the \mgf{} $\E e^{t\hcX_1}$ is finite for all real $t$,
and thus the moments \eqref{e24a} (even for integer $r$)  
determine the distribution.
Moreover,
\refT{TMD} shows that all moments converge in \eqref{e23a} (to the limits
\eqref{e24a}); this was earlier shown in
\cite{F:exact,Puy} 
in the case that the replacements $\gd,\gam,\ga$ are integers.

If we further assume $x_2=0$, so we start only with white balls, then
$\hcX_1$ has a density function that can be expressed using the density function
of a Mittag-Leffler distribution with parameter $\gda$, or a $\gda$-stable
distribution, see \cite[Theorem 1.8]{SJ169}
and \cite[Proposition 16]{F:exact}. 

The \ctime{} processes $W(t)$ and $B(t)$ 
are studied by \cite{ChenMahmoud},
which includes 
our \refT{TC} for this balanced urn.
\end{example}

\begin{example}\label{E2X}
Generalizing \refE{E2},
consider a general triangular
two-colour urn $\cU$ with random replacement matrix
\begin{align}
\matrixx{\xi_{11}&\xi_{12}\\0&\xi_{22}}.   
\end{align}
Such urns have been studied by \citet{Aguech},
who proved (among other results) 
the existence of a.s.\ limits under some assumptions
(including our \ref{A+}, $\E\xi_{22}\ge\E\xi_{11}$, and an unnecessary
independence assumption).
We extend this result as follows.

Assume $\xi_{11},\xi_{12},\xi_{22}\in L^2$ (Condition \ref{A2}), 
and let $\gd:=\E\xi_{11}$, $\gam:=\E\xi_{12}$, $\ga:=\E\xi_{22}$,
so that $\smatrixx{\gd&\gam\\0&\ga}$ is the mean replacement matrix.
Suppose that \ref{A-5} holds, that $\xi_{11}\ge0$ a.s.,
and that, as in \refE{E2},
$\gd>0$, $\gam>0$, $w_0=x_1>0$, and $b_0=x_2\ge0$.
Then the urn satisfies 
\refAA{} if $\xi_{22}\ge0$ a.s.,
and \refAAZZ{} in any case.
Hence, \refTs{T1-} and \ref{TC-} apply, and thus the conclusions of \refTs{T1}
and \ref{TC} hold for all $\ga$.

As in \refS{Smean}, we
let $\cU'$ denote
the mean urn  with replacement matrix $\smatrixx{\gd&\gam\\0&\ga}$; 
this urn is of the type in \refE{E2}.
As discussed in \refS{Smean}, all parameters  $\gl_i,\glx_i,\hgl,\dots$
are the same for $\cU$ and $\cU'$, and thus all results are qualitatively
the same for the two urns.
Hence, all results in \refE{E2} hold for the urn $\cU$ with random
replacements too, except any assertions on the precise distributions
of the limits (including the moment formula \eqref{e24a}, which as shown in
\refE{E2p} below is \emph{not}  valid in general for random replacements).
Note, however, that by \refT{Tmean},
in all cases in \refE{E2} with a deterministic limit $\hcX_i$ for the mean
urn, we have the same limit for the urn $\cU$.
\end{example}

\begin{example}\label{E2p}
Consider a two-colour urn $\cU$ with the random replacement matrix 
\begin{align}\label{e2p}
\matrixx{\xi&1-\xi\\0&1},   
\end{align}
where $\xi\in\Be(p)$ for some $p\in(0,1)$.
In other words, if we draw a white ball, we add another ball that is white
with probability $p$, and otherwise black; if we draw a black ball we always
add another black ball. (As usual, we also always return the drawn ball.)
Note that this urn is balanced, in spite of replacements being random.

This urn appears in several applications, see \refEs{E2p1} and \ref{ERW}
for two of them.

This urn is a special case of \refE{E2X}, and 
a.s.\ convergence for the urn follows from 
\refT{T1}.
Moment convergence follows from \refT{TMD}.

The mean replacement matrix is
\begin{align}\label{e2pmean}
(r_{ij})_{i,j=1}^2=\matrixx{p&1-p\\0&1}.  
\end{align}
By \refS{Smean}, the asymptotic behaviour of the urn $\cU$ is qualitatively the
same as for the mean urn $\cU'$ with the replacement matrix \eqref{e2pmean},
which is an instance of \refE{E2}, more precisely the balanced
\refCase{E2b}.
In particular, since \eqref{e23f} shows that $\hcX_2=\ga=1$ is constant for
the mean urn $\cU'$,
the same holds for the original urn $\cU$ by \refT{Tmean}.

We may also relate the urn $\cU$ to the mean urn $\cU'$ in another, more
direct, way.
Conditioned on the contents $(X_{n1},X_{n2})$ of the urn at time $n$,
the probability that the next added ball is white is, letting $\zeta$ be the
colour of the drawn ball,
\begin{align}\label{e2p1}
  \P\bigpar{\zeta=1\mid X_{n1},X_{n2}}\cdot p
+   \P\bigpar{\zeta=2\mid X_{n1},X_{n2}}\cdot 0
=\frac{pX_{n1}}{X_{n1}+X_{n2}}.
\end{align}
Hence, if we define
\begin{align}\label{e2p2}
  Y_{n1}&:=pX_{n1},
\\\label{e2p3}
  Y_{n2}&:=(1-p)X_{n1}+X_{n2},
\end{align}
and note that $Y_{n1}+Y_{n2}=X_{n1}+X_{n2}$, we see from \eqref{e2p1} that
we may
regard the \emph{added} ball in $\cU$ as the \emph{drawn} ball in an urn
with composition $\bY_n=(Y_{n1},Y_{n2})$. Adding a white ball to $(X_{n1},X_{n2})$
(i.e., increasing $X_{n1}$ by 1) means by \eqref{e2p2}--\eqref{e2p3} adding
$(p,1-p)$ to $(Y_{n1},Y_{n2})$, while adding a black ball to $(X_{n1},X_{n2})$
means adding $(0,1)$ to $(Y_{n1},Y_{n2})$. 
Consequently, the stochastic process $(\bY_n)_{n\ge0}$ describes a \Polya{}
urn with the 
replacement matrix 
$\smatrixx{p&1-p\\0&1}$,
which is the same as \eqref{e2pmean} for the mean urn $\cU'$ above. 
Note, however, that the
initial conditions now are, by \eqref{e2p2}--\eqref{e2p3},
\begin{align}\label{um1}
  y_1=px_1,\qquad y_2=(1-p)x_1+x_2.
\end{align}
By \refE{E2}, we have
\begin{align}\label{e2p4}
  Y_{n1}/n^{p}\asto \hcY_1
\end{align}
where the limit by \eqref{e24a} has moments, recalling \eqref{um1},
\begin{align}\label{e2p5}
    \E\hcY_1^r
=
p^r\frac{\Gamma\xpar{y_1+y_2}\,\Gamma(y_1/p+r)}
 {\Gamma(y_1/p)\,\Gamma\xpar{y_1+y_2+rp}}
=
p^r\frac{\Gamma\xpar{x_1+x_2}\,\Gamma(x_1+r)}
 {\Gamma(x_1)\,\Gamma\xpar{x_1+x_2+rp}}.
\end{align}
Hence, by \eqref{e2p2}, we have in the urn $\cU$ 
with replacements \eqref{e2p}
\begin{align}\label{e2p4x}
  X_{n1}/n^{p}\asto \hcX_1
\end{align}
with 
$\hcX_1=p\qw\hcY_1$, and thus
\begin{align}\label{e2p6}
    \E\hcX_1^r
=
p^{-r}\E\hcY_1^r=
\frac{\Gamma\xpar{x_1+x_2}\,\Gamma(x_1+r)}
 {\Gamma(x_1)\,\Gamma\xpar{x_1+x_2+rp}},
\qquad r\ge0.
\end{align}

By comparing \eqref{e2p4x} with \eqref{e24a} for the mean urn, we see that
the means 
$\E\hcX_1$ are the same for the two urns, while for the second moment,
\eqref{e2p6} and \eqref{e24a} yield, for the urn $\cU$ and its mean urn
$\cU'$, respectively,
\begin{align}
\E  \hcX_1^2 = x_1(x_1+1)\frac{\gG(x_1+x_2)}{\gG(x_1+x_2+2p)},
&&&
\E  \hcX_1^2 = x_1(x_1+p)\frac{\gG(x_1+x_2)}{\gG(x_1+x_2+2p)}.
\end{align}
The variance is thus larger for the urn $\cU$ with random replacement.
(Perhaps not surprisingly.)
This shows that an urn and its mean urn in general have different asymptotic
distributions, 
although the qualitative behaviour is the same as shown in \refS{Smean}.
\end{example}

\begin{example}\label{E2p1}
One example where the  urn in \refE{E2p} with replacement matrix
\eqref{e2p} appears is that
it describes the size of the root cluster
for  bond percolation (with parameter $p$)
on the random recursive tree
(with vertices in the root cluster coloured
white and all other vertices black, and
the initial vector $(1,0)$).
(See the argument in the generalization \refE{Epref} below.
The root cluster is studied by other methods in
\cite{BaurBertoin2015}, \cite{Kursten2016}, \cite{Baur}, \cite{DesmaraisHW}.)
In this case we obtain
\eqref{e2p4x} where
\eqref{e2p6} yields
\begin{align}\label{ML}
\E\hcX_1^r = \frac{\gG(r+1)}{\gG(1+rp)},
\qquad r\ge0,  
\end{align}
which means that $\hcX_1$ has a 
Mittag-Leffler distribution;
this was proved by \citet{BaurBertoin2015} (using other methods).
\end{example}

\begin{example}\label{Epref}
\citet{Baur} and
\citet{DesmaraisHW} considered (among other things)
the root cluster in bond percolation on a preferential attachment tree,
generalizing \refE{E2p1}.
The tree is defined as follows, for a real parameter $\ga$.
Construct the rooted tree $\cT_n$ with $n$ vertices
recursively, starting with  $\cT_1$ being just the root and adding vertices
one by one;  
each new vertex is attached
to a parent $v$ chosen among existing vertices with probability proportional
to $\ga d(v)+1$, where $d(v)$ is the current outdegree of $v$.
We also  perform bond percolation, and let each edge by \emph{active}
with probability $p\in(0,1)$, independently of all other edges.
(Note that both active and passive edges are counted in the outdegree.)
A vertex is active if it is connected to the root by a path of active edges.
Let $Z_n$ be the number of active vertices in $\cT_n$.

The case $\ga=0$ gives the random recursive tree in \refE{E2p1}.
We consider here the case $\ga\ge0$ (as assumed in \cite{Baur}), 
and study the modifications for $\ga<0$
in the following example.

We model this process by an urn 
with two colours, where each vertex $v$ contributes $\ga d(v)+1$ balls,
which are white if the urn is active and black otherwise.
To find the parent of the next vertex corresponds to drawing a ball from the
urn; if the ball is white then the parent is active and the new vertex
becomes active with probability $p$ (and otherwise passive); if the ball is
black then the parent is passive and the new vertex always becomes passive.
Since the outdegree of the parent increases by 1, we add $\ga$ balls of the
same colour as the drawn ball, plus one ball for the new vertex which is
white if the new vertex is active, i.e., with probability $p$ if the drawn
ball is white, and otherwise black.
Hence, the urn is a  \Polya{} urn with random replacement matrix
\begin{align}\label{e2q}
\matrixx{\ga+\xi&1-\xi\\0&\ga+1},   
\end{align}
where $\xi\in\Be(p)$. We start with 1 active vertex, and thus the urn starts
with a single white ball, i.e., $\bx=(1,0)$.

If $\ga=0$ (the random recursive tree), we get again \eqref{e2p}, as
discussed in \refE{E2p1}. 
In general, unlike \refE{E2p1}, the number of active vertices $Z_n$ is not
directly reflected by the contents of the urn.
However, if $N_{n1}$ is the number of drawn white balls, then this is the
number of vertices that get an active parent;
hence $N_{n1}$ is the total outdegree of the active vertices, and therefore
the number of white balls in the urn is
\begin{align}\label{qa}
  W_n=X_{n1} = \ga N_{n1} + Z_n.
\end{align}

The replacement matrix \eqref{e2q} shows that the 
urn is of the type in \refE{E2X}; furthermore, it is balanced with balance
$\ga+1$. We have $\glx_1=\gl_1=\E\xi_{11}=\ga+p$ and $\hgl=\gl_2=\ga+1$;
further $\kk_1=\kk_2=0$.
\refT{T1} yields
\begin{align}\label{qb}
  \frac{X_{n1}}{n^{(\ga+p)/(\ga+1)}}\asto\hcX_1.
\end{align}
Moreover, \refT{TN2} yields
\begin{align}\label{qc}
\frac{ N_{n1}}{X_{n1}}\asto \frac{1}{\glx_1}=\frac{1}{\ga+p},
\end{align}
and thus \eqref{qa} yields
\begin{align}\label{qd}
\frac{Z_n}{X_{n1}}
= \frac{X_{n1}-\ga N_{n1}}{X_{n1}}\asto 1-\frac{\ga}{\ga+p}
=\frac{p}{\ga+p}.
\end{align}
Hence, \eqref{qb} yields
\begin{align}\label{qe}
\frac{Z_n}{n^{(\ga+p)/(\ga+1)}}
\asto
\cZ:=\frac{p}{\ga+p}\hcX_1,
\end{align}
where the limit is in $(0,\infty)$ a.s.
The limits \eqref{qb} and \eqref{qe} hold also in $L^r$ for any $r<\infty$
by \refTs{TMD} and \ref{TMQ}; hence all moments converge.

This complements
\citet[Proposition 4.1]{Baur}, who shows $L^2$-convergence in
\eqref{qe}
and gives the first two moments of the
limit
(by different but related methods),
and 
\citet{DesmaraisHW} who
prove  convergence of all moments in \eqref{qe}
and give a recursion for the moments of the limit.
(The distribution of the limit is not known explicitly.)
\end{example}

\begin{example}\label{Epref-}
In \refE{Epref}, we assumed $\ga\ge0$. However, the results extend easily
to the case $\ga<0$. 
(This case was included in \cite{DesmaraisHW}.)
In this case, as is well known, we must have $\ga=-1/d$
with 
$d\ge 1$ an integer; the random tree $\cT_n$ then is a 
random $d$-ary recursive tree \cite[Section 1.3.3]{Drmota}.
(The case $d=1$ is trivial, and we assume $d\ge2$.)
In this case the replacements \eqref{e2q} do not satisfy \ref{A-5}, since
$\ga+\xi$ may take the non-integer negative value $-1/d$. 
This is can be remedied as in \refR{RA-5} 
by multiplying the number of white balls by $d$ and adjusting the activity
$a_1$, but in the present case we find it simpler to multiply the number of
balls of both colours by $d$ and keep the activities $a_i=1$. 
This gives the replacement matrix
\begin{align}\label{e2qd}
\matrixx{d\xi-1&d-d\xi\\0&d-1},   
\end{align}
with the initial state $\bx=(d,0)$.
In \eqref{qa}, we have to replace $X_{n1}$ by $X_{n1}/d$, which yields
\begin{align}\label{jepp}
  Z_n=\frac{X_{n1}+N_{n1}}d.
\end{align}

We have $\glx_1=\gl_1=\E\xi_{11}=dp-1$, so \ref{A-7} requires $dp>1$.
(In fact, it is easily seen that if $dp\le1$, then  $Z_n\asto Z_\infty<\infty$
\cite{DesmaraisHW}.)
Assuming $dp>1$, 
this  urn satisfies \refAAZ.
It does not satisfy \ref{A-8}, but \refL{LBB5} shows that $\cX_0>0$ a.s.,
and thus we obtain by \refT{T1--}
the \as{} limit \eqref{qe} again
 (with $\hcX_1$ replaced by $\hcX_1/d$); 
this can be written
\begin{align}\label{qe-}
\frac{Z_n}{n^{(dp-1)/(d-1)}}
\asto
\cZ:=\frac{p}{dp-1}\hcX_1.
\end{align}
Moment convergence, earlier shown by 
\cite{DesmaraisHW},
follows from \refTs{TMD} and \ref{TMQ}.
\end{example}

\begin{example}\label{Eprefk}
 In \refE{Epref}, we studied the number $Z_n$ of vertices in the preferential
 attachment tree $\cT_n$ such that the path to the root contains only active
 vertices. More generally, let $\Zk_n$ be the number of vertices such that
 this path contains exactly $k\ge0$ passive edges.
For fixed $k$, this can be treated similarly, with an urn 
with $q\ge k+2$ colours $1,\dots,q$,
where vertices
with $j$ passive edges on the path to the root are represented by colour
$\max(j+1,q)$.
For example, for $q=3$, this leads to an urn with replacement
matrix 
\begin{align}\label{e3q}
\matrixx{\ga+\xi&1-\xi&0\\
0&\ga+\xi&1-\xi\\
0&0&\ga+1},   
\end{align}
where $\xi\in\Be(p)$ as above. (It does not matter whether we write this
with the same $\xi$ on both rows or not; recall \refR{Rmatrix}.)
For a general $q\ge2$ we have 
$\xi_{ii}=\ga+\xi$ and $\xi_{i,i+1}=1-\xi$ for $1\le i\le q-1$,
$\xi_{qq}=\ga+1$, 
and all other $\xi_{ij}=0$.
The urn is balanced with balance $\gbal=\ga+1$.
We assume for simplicity $\ga\ge0$; the case $\ga<0$ can be treated as in
\refE{Epref-}.

We find $\gl_1=\dots =\gl_{q-1}=\ga+p$ and $\gl_{q}=\ga+1$, and thus
$\glx_1=\dots=\glx_{q-1}=\ga+p$, $\hgl=\glx_q=\ga+1$, 
$\kk_i=i-1$ for $1\le i\le q-1$,
and $\kk_q=0$.
We have in analogy with \eqref{qa}, $\Zx{k}_n=X_{n,k+1}-\ga N_{n,k+1}$
for $k\le q-2$.
\refT{T1} applies and shows together with \refT{TN2}
as in \refE{Epref},
an \as{} limit. Since $\kk_{k+1}=k$, we now get
\begin{align}\label{qe2}
\frac{\Zx{k}_n}{n^{(\ga+p)/(\ga+1)}\log^k n}
\asto\cZx{k}:=\frac{p}{\ga+p}\hcX_{k+1}.
\end{align}
Furthermore, 
we obtain from \eqref{hsubl} and induction,
since $r_{i,i+1}=\E(1-\xi)=1-p$
for every $i\le q-1$, 
\begin{align}
  \cZx{k}=\frac{1}{k!}\parfrac{1-p}{\ga+1}^k\cZx0
,\end{align}
where $\cZx0$ equals $\cZ$ in \eqref{qe}.
Consequently, the numbers $\Zx{k}_n$ for different $k$ are asymptotically
proportional.
\end{example}

\begin{example}\label{ERW}
  Another example where the  urn in \refE{E2p} with replacement matrix
  \eqref{e2p} appears is for elephant random walks with delays.
In a standard elephant random walk (ERW), the elephant takes steps 
$Y_n\in\set{\pm1}$;
after an initial step $Y_1$, the elephant (which remembers the entire walk)
chooses one of the preceding steps, uniformly at random, and then randomly
either (with probability $p$) repeats it, or (with probability $q$)
takes a step in the opposite direction. (Here $q=1-p$.)
As noted by \citet{BaurBertoin2016} 
(to which we refer for details and further references), 
this may be modelled by a (non-triangular)
\Polya{} urn, with one white ball for each step $+1$ and one black ball for
each step $-1$ taken so far; the replacement matrix is
\begin{align}
  \label{erw0}
\matrixx{\xi&1-\xi\\1-\xi&\xi} 
\end{align}
with $\xi\in\Be(p)$.
Hence results for the ERW follows from known results for irreducible
\Polya{} urns \cite{BaurBertoin2016}. (We have nothing to add here.)

In the \emph{elephant random walk with delays}
\cite{HKL2010,GutS2022,GutS2021}, the elephant has a third possibility:
with probability $r$ it makes a step 0 (i.e., stays put), regardless of the
remembered step. (Now $p+q+r=1$; we assume $p,q,r>0$.)
This can be modelled by a 3-colour urn, with colours representing steps $+1$,
$-1$, and $0$, and replacement matrix
\begin{align}
  \matrixx{\zeta_1&\zeta_2&\zeta_3\\
\zeta_2&\zeta_1&\zeta_3\\
0&0&1},
\end{align}
where $(\zeta_1,\zeta_2,\zeta_3)$ is a random vector with exactly one
component 1 and the others 0, and $(P[\zeta_i=1])_{i=1}^3=(p,q,r)$.
This \Polya{} urn is neither triangular nor irreducible, but it may be
regarded as a combination of two such urns (cf.\ \refR{Rnontri2})
as follows. Let as before $Y_n\in\set{\pm1,0}$ be the $n$th step, and let
$Z_n:=|Y_n|\in\set{0,1}$; $Z_n$ thus just records whether the elephant moves
or stays put. 
(The process $Z_n$  is called \emph{Bernoulli elephant random walk} in
\cite{GutS-BERW};  it has also been studied in  \cite{Bercu} and, for
somewhat different reasons, in \cite{HKL2014}.)
As noted by \cite[V.C]{BaurBertoin2016},
the process $(Z_n)$ can be modelled by a \Polya{} urn with
$q=2$ and one white ball for each step $\pm1$ and one black ball for each
step $0$ so far; the replacement matrix is 
\begin{align}
  \matrixx{1-\zeta_3&\zeta_3\\0&1}.
\end{align}
This is the urn in \refE{E2p} with $\xi=1-\zeta_3\in\Be(1-r)$.
The number of non-zero steps up to time $n$ is $W_n$, the number of white
balls in the urn. 
By conditioning  on  $Z_1:=|Y_1|$, we
may assume that $Z_1$ is deterministic. Moreover, the case $Z_1=0$ is trivial,
with $Y_n=Z_n=0$ for all $n\ge1$; hence we may assume $Z_1=1$,
and thus the urn starts with 1 white ball.
(\cite{GutS2021,GutS2023} use a different initial condition with
$Z_1\sim\Be(1-r)$.)
Then the urn is the same as in \refE{E2p1}, and \refT{T1} yields
\begin{align}\label{erw1}
  W_n/n^{1-r}\to\hcX_1,
\end{align}
as shown by other methods in \cite[Theorem 3.1]{GutS2021};
furthermore, 
\cite[Lemma 2.1]{Bercu} and
\cite[Theorem 5.1]{GutS2023}
show  that
\eqref{ML} holds (with $p$ replaced by $1-r$ and $r$ by $s$, say), 
and thus $\hcX_1$ has a Mittag-Leffler distribution
(as we saw in \refE{E2p1}).
Moreover, \cite{Bercu} and \cite{GutS-BERW} 
show that moment convergence holds in \eqref{erw1};
this also follows from \refT{TMD}.

Conditioned on the number $W_n$, the position of the
elephant is the same as for a standard ERW with $W_n$ steps,
and thus the limit results in \cite{Bercu} and \cite{GutS2023} for the ERW
with delays 
may easily be obtained by combining \eqref{erw1} and 
the results for the standard ERW obtained by \cite{BaurBertoin2016}
from the urn \eqref{erw0}; we leave the details to the reader.
\end{example}

\begin{example}\label{E3}
The triangular urn with $q=3$ and balanced deterministic replacements
(with all entries integers $\ge0$)
\begin{align}
  \matrixx{\ga&\gb&\gs-\ga-\gb\\
0&\gd&\gs-\gd\\
0&0&\gs}
\end{align}
was studied by \citet[Section 2.5]{Puy} and 
\citet[Section 10]{F:exact}; the results include convergence in distribution
(after normalization),
and, in some cases, 
convergence of all moments with explicit
formulas for moments of the limits.

In particular, they show 
(in our notation and correcting several typos) 
that
if $\ga>\gd>0$ and $\gb>0$, then
\begin{align}\label{e3a}
  X_{n2}/n^{\ga/\gs} \dto\hcX_2
\end{align}
where the limit has moments 
\begin{align}\label{um2}
  \E \hcX_2^\ellr
= \parfrac{\ga\gb}{\ga-\gd}^\ellr 
\frac{\gG(x_1/\ga+\ellr)\gG(|\bx|/\gs)}
{\gG(x_1/\ga)\gG((|\bx|+\ellr\ga)/\gs)}
.\end{align}
In this case, we have $\gs\ge \ga+\gb>\ga>\gd$ and thus 
$\glx_1=\glx_2=\gl_1=\ga$, $\gl_2=\gd$, $\hgl=\glx_3=\gl_3=\gs$, 
$\hkk=\kk_i=0$, $\gam_i=0$ ($i=1,2,3$). 
Hence, \refT{T1} yields \eqref{e3a} with the stronger convergence a.s. 
Moreover,
the leaders are 1 and 3,
and \refL{LhC} yields $\hcX_2=\hc_{21}\hcX_1$, where \refL{LhC} and \eqref{eq4}
show that $(1,\hc_{21})$ is a left eigenvector
of $\smatrixx{\ga&\gb\\0&\gd}$,
with eigenvalue $\ga$.
Consequently, $\hc_{21}=\gb/(\ga-\gd)$ and
\begin{align}\label{um3}
  \hcX_2 =\frac{\gb}{\ga-\gd}\hcX_1.
\end{align}
Furthermore, as noted in \cite{F:exact}, we may in this urn
be partially colour-blind and merge colours 2 and 3; then
$(X_{n1},X_{n2}+X_{n3})$ is a 2-colour urn with replacement matrix
$\smatrixx{\ga&\gs-\ga\\0&\gs}$; hence the moments $\E \hcX_1^r$ are given
by \eqref{e24a}, where now $\ga$ and $\gd$ are replaced by $\gs$ and $\ga$,
and $x_1+x_2$ is replaced by $|\bx|=x_1+x_2+x_3$, i.e.,
\begin{align}\label{e24a3}
    \E\hcX_1^r
=
\ga^r\frac{\Gamma\bigpar{|\bx|/\gs}\Gamma(x_1/\ga+r)}
 {\Gamma(x_1/\ga)\Gamma\bigpar{(|\bx|+r\ga)/\gs}},
\qquad r\ge0. 
\end{align}
Thus, \eqref{um2} follows by \eqref{um3}. 

Similarly, in the case $\ga=\gd>0$ and $\gb>0$, 
\cite{Puy} and \cite{F:exact} show
(again correcting several typos) 
\begin{align}\label{e3b}
  X_{n2}/\xpar{n^{\ga/\gs}\log n} \dto\hcX_2
\end{align}
where the limit has moments 
\begin{align}\label{um5}
  \E \hcX_2^\ellr
= \parfrac{\ga\gb}{\gs}^\ellr  
\frac{\gG(x_1/\ga+\ellr)\gG(|\bx|/\gs)}
{\gG(x_1/\ga)\gG((|\bx|+\ellr\ga)/\gs)}
.\end{align}
In this case 
$\glx_1=\glx_2=\gl_1=\gl_2=\ga$, $\kk_1=0$ and
$\kk_2=1$, and as above  $\hgl=\glx_3=\gl_3=\gs$
and $\hkk=\kk_3=0$.
Hence,
\refT{T1} yields \eqref{e3b} with convergence a.s.
Moreover, \refL{LhC} and \eqref{hsubl} yield
\begin{align}\label{um6}
  \hcX_2=\frac{\gb}{\gs}\hcX_1,
\end{align}
which yields \eqref{um5} by \eqref{e24a3}.

\cite{F:exact} and \cite{Puy} further prove moment convergence 
in \eqref{e3a} and \eqref{e3b}, which also follows from \refT{TMD}.

Note that  \cite{F:exact} and \cite{Puy}
assume the replacements to 
be integers, while we obtain the results above also for non-integer 
replacements.

A.s.\ convergence in \eqref{e3a} and \eqref{e3b} 
follows also from 
\cite{BoseDM}, see \refE{EBose}.  
\end{example}

\begin{example}\label{EBose}
  \citet{BoseDM} study a rather general class of balanced
triangular urns 
with  deterministic $\bxi_i$; thus $\xi_{ij}=r_{ij}$.
They show \as{} convergence of $X_{ni}$ suitably normalized, as in our
\refT{T1}. 

We introduce some of the notation from \cite{BoseDM}.
The colours are \setq{} (where they write $q=K+1$), 
and the replacement matrix is triangular,
so \eqref{q1}--\eqref{q3} hold with the natural order $<$.
The diagonal entries $r_{ii}=\gl_{ii}$ are denoted $r_i$, and the positions
of the
weak maxima in the sequence $r_1,\dots,r_q$ are denoted $i_1,\dots,i_{J+1}$.
Thus $r_{i_1}\le r_{i_2}\le \dots\le r_{i_{J+1}}$, and $r_k< r_{i_j}$ when
$i_j<k<i_{j+1}$.
Since the urn is balanced,
\refL{LBB} and \refR{RBB} show that 
$r_q=\gl_q=\hgl$ is
a maximum, and thus $i_{J+1}=q$. Clearly, $i_1=1$.
The \emph{$j$th block} of colours is $\set{i_j,\dots,i_{j+1}-1}$.
(In \cite{BoseDM}, the replacements are normalized by $\gl_q=\hgl=1$, which can
be assumed without loss of generality. 
It is also assumed that initially there is 1 ball in the urn, \ie, 
$\sum_ix_i=1$; 
this seems to be a mistake since one cannot in general
normalize both to 1 simultaneously.) 

\cite{BoseDM} says that the colours are arranged in \emph{increasing order} if
for every $k\in(i_j,i_{j+1})$ (with $1\le j\le J$), there exists $m\in[i_j,k)$
such that $r_{mk}>0$. Using our terminology, this is easily seen to be
equivalent to: If $k\in(i_j,i_{j+1})$, then $k$ is a descendant of $i_j$.
\cite[Proposition 2.1]{BoseDM} shows that in every balanced triangular urn,
the colours can be rearranged in increasing order. 
[Sketch of proof: 
Construct the blocks in backwards order. In each step find a colour $i$ (to be
labelled $i_j$) with $\gl_i$ maximal among the remaining colours; let the
next block consist of $i$ and all its remaining descendants.
Order this block in a suitable way, with $i$ first, and place it before the
previously constructed blocks. Repeat with the remaining colours.]

The main result of \cite{BoseDM} 
further assumes \cite[(2.2)]{BoseDM}, which says 
that for every $j=1,\dots,J$, there
exists $m\in[i_j,i_{j+1})$ such that $r_{m,i_{j+1}}>0$. In our terminology,
  and assuming (as in \cite{BoseDM}) that the colours are in natural order,
  this is equivalent to $i_{j+1}\succ i_j$.
It follows that the assumptions in \cite{BoseDM} 
(i.e.,  increasing order and their (2.2)) imply that 
$i_1,\dots,i_{J+1}$ are precisely 
the colours $i$ with $\glx_i=\gl_i$, \ie{}
our leaders and subleaders
(see \eqref{lead2} and \eqref{sublead}),
and that the leaders $\nu$ have distinct $\glx_\nu$.
Moreover,
for each leader $\nu=i_j$, the subleaders in $\sD_\nu$ 
(see \eqref{sD})
are $i_{j+1},\dots,i_{j+\ell}$ where 
$\ell=\ell_j\ge0$ is the largest integer with $\gl_{i_{j+\ell}}=\gl_{i_j}$;
these subleaders form a chain in the partial order $\prec$, and thus
$\kk_{i_{j+k}}=k$ for $k=0,\dots,\ell$.
The set $D_\nu^\kk$ is an interval $[i_j,i_{j+1})$ from a (sub)leader
to the next.

The \as{} convergence in the main result  \cite[Theorem 3.1]{BoseDM} now
follows from  \refT{T1}, \refL{LC}, and \refT{TEL}.
Moreover, \cite[Remark 3.3]{BoseDM} says that it is clear from their proof
that
\as{} convergence holds
also without their assumptions of increasing order and their (2.2);
\cite{BoseDM} 
thus essentially states our \refT{T1} for the case of a balanced urn with
deterministic replacements.

Furthermore,
\cite{BoseDM} also shows convergence in $L^2$, which is extended to $L^p$
for any $p$ by our \refT{TMD}.
\end{example}

\begin{example}\label{EcX=0}
  Consider an urn with $q=4$ and replacement matrix 
  \begin{align}
    \matrixx{\ga\eta_1-1&0&\eta_2&0\\
0& \gb\eta_1-1 & \eta_2&0\\
0 & 0 &\gam &0\\
0&0&0&\gd},
  \end{align}
where $\ga,\gb,\gam,\gd$ are positive integers and $\eta_1,\eta_2\in\Be(1/2)$
are independent.
We have $\gl_1=\ga/2-1$, $\gl_2=\gb/2-1$, $\gl_3=\gam$, and $\gl_4=\gd$.
Suppose that $\gl_4=\gl_1>\gl_2>\gl_3>0$, and start with $\bX_0=(1,1,0,1)$.

This urn satisfies \refAAZ, but not \ref{A-8}. 
We have $\hgl=\glx_1=\glx_3=\glx_4=\gl_1$,  $\glx_2=\gl_2$, and  $\kk_i=0$
$\forall i$.
\refT{TC-} shows that \eqref{tc1} holds for all colours $i$,
with 
\begin{align}
  e^{-\glx_i t}X_i(t)\asto \cX_i.
\end{align}
Furthermore, 
$\cX_4>0$ a.s., as is seen
by considering only balls of colour 4;
hence, by \eqref{el3}, 
$\cX_0=\hgl\qw(\cX_1+\cX_4)\ge\hgl\qw\cX_4>0$ a.s.

Furthermore, 3 is a follower of the leader 1, and 
\refL{LNDB} with \refR{Rdep-} yields
$\cX_3=c\cX_1$ for some $c>0$. (In fact, $c=1/2(\gl_1-\gl_3)$, by the same
argument as for \eqref{e2b}.)
Note also that $\cX_1$ and $\cX_2$ are independent.

It is obvious that colours 1 and 2 both may die out in a few draws, and that
if they do, they may or may not first generate a ball of colour 3.
If they do not die out, then \as{} $\cX_1>0$  and $\cX_2>0$ respectively,
see \refR{Rdie}.

Consequently, the following cases can appear, all with positive
probabilities:
\begin{romenumerate}
\item 
$\cX_1>0$ and then $\cX_3>0$ and thus
$X_3(t)$ grows at rate $e^{\glx_3t}=e^{\gl_1 t}$.
Similarly,
by \refT{T1--} and \eqref{hcx}, $X_{n3}/n\asto\hcX_3>0$.

\item 
$\cX_1=0$ but $\cX_2>0$; then $\cX_3=0$, but by considering the urn with
colour 2, 3, and 4 only (after 1 has died out), it follows that
$e^{-\gl_2}X_3(t)\asto \cX_3':=c'\cX_2>0$ for some $c'>0$, and thus
$X_3(t)$ grows at rate $e^{\gl_2 t}$; 
similarly
$X_{n3}/n^{\gl_2/\gl_1}\asto \hcX_3'>0$.

\item 
$\cX_1=\cX_2=0$ and both $X_1(t)$ and $X_2(t)$ die out,
but at least one of them first gets a ball of colour 3 as offspring.
Then, by considering only balls of colour 3 and 4 (when the others have died
out),
$e^{-\gl_3}X_3(t)\asto \cX_3''>0$, and thus
$X_3(t)$ grows at rate $e^{\gl_3 t}$;
similarly
$X_{n3}/n^{\gl_3/\gl_1}\asto \hcX_3''>0$.

\item 
$\cX_1=\cX_2=0$, and both  $X_1(t)$ and $X_2(t)$ die out
without producing an offspring of colour 3.
Then, $X_3(t)=0$ for all $t$ and thus $X_{3n}=0$ for all $n$.
\end{romenumerate}

This example shows that in a case when $\cX_i=0$ with positive probability,
it may still be possible to find precise limit results, but  different 
limit results may hold in different subcases.
It seems that this can be done very generally on a case by case basis,
but as said in \refR{RcX=0} we do not attempt any general statement.
\end{example}

\begin{example}\label{E+-}
An interesting counterexample is given by
the replacement matrix 
\begin{align}
\matrixx{0&\phantom+1\\0&\pm1},  
\end{align}
where $\xi_{22}=\pm1$  denotes a random variable with 
$\P(\xi_{22}=1)=\P(\xi_{22}=-1)=\frac12$.
In words: when we draw a white ball, it is replaced together with a black ball;
when we draw a black ball, we toss a coin and then (with probability
$\frac12$ each) either remove the ball or replace it together with another
black ball. 

We have $r_{22}=\E\xi_{22}=0$, and thus $\gl_1=\gl_2=0=\glx_1=\glx_2=\hgl$, 
$\kk_1=0$, $\kk_2=1$, $\hkk=1$, $\hkko=2$.
Note that this example is excluded from \refTs{T1-}--\ref{TC-} since 
\ref{A-7} does not hold.
(However,
\ref{A0}--\ref{A2}, \ref{A-5}, and \ref{A-6} hold, provided $x_1>0$ and
$x_2\in\bbZgeo$.)
We will see that, in fact, we do \emph{not} have \as{} convergence
of $B_n/n\qq=X_{n2}/n\qq$ and $B(t)/t=X_2(t)/t$ as \eqref{t1c} and
\eqref{tc1} would give; however, these converge in distribution. 

Consider first a continuous time urn with only black balls. 
This is a branching process where balls live an $\Exp(1)$ time and then
randomly either die or are split into two.
Let $Y(t)$ denote such an urn with black balls, starting with $Y(0)=1$,
and denote its \pgf{} by (for $|z|\le1$, say)
\begin{align}\label{e+-1}
  g_t(z):=\E z^{Y(t)}.
\end{align}
Then the Kolmogorov backward equation 
\cite[Theorem 12.22]{Kallenberg}
yields, for $t>0$,
\begin{align}\label{e+-2}
  \frac{\partial}{\partial t} g_t(z) = \frac12\bigpar{1+g_t(z)^2}-g_t(z)
=\frac12\bigpar{1-g_t(z)}^2
\end{align}
with $g_0(z)=z$, which has the solution
\begin{align}\label{e+-3}
  g_t(z) = 1-\Bigpar{\frac{1}{1-z}+\frac{t}2}\qw
=\frac{t(1-z)+2z}{t(1-z)+2}
=\frac{t}{t+2}+\frac{2}{t+2}\cdot\frac{\frac{2}{t+2}z}{1-\frac{t}{t+2}z}.
\end{align}
Consequently, $Y(t)$ has a modified geometric distribution:
$\P(Y(t)=0)=g_t(0)=\frac{t}{t+2}$ and the conditional distribution
$\bigpar{Y(t)\mid Y(t)>0}$ is $ \Gei\bigpar{\frac{2}{t+2}}$.
In particular, $\P(Y(t)=0)\to1$ as \ttoo; since 0 is an absorbing state, it
follows that \as{} $Y(t)=0$ for sufficiently large $t$; in other words,
$Y(t)$ dies out. (This follows 
also since $Y(t)$ is a time-changed simple random walk, absorbed at 0.)
A simple calculation yields $\E Y(t)=1$ and $\Var Y(t)=t$, in accordance
with \eqref{zlm00}--\eqref{zlm02} and \eqref{zlmgb}.
Note that $\bigpar{t\qw Y(t)\mid Y(t)>0}\dto \Exp\bigpar{\frac12}$ as \ttoo.
Roughly speaking, for large $t$, $Y(t)$ is non-zero with probability
$\approx 2/t$, and if it is, it is of order $t$.

Now consider the two-colour urn above, and assume that we start with
$\bX_0=(1,0)$, \ie, 1 white ball. 
The number of white balls is constant for this urn, and thus $W(t)=1$ for
all $t$ in the continuous-time urn.
This means that white balls are drawn according to a Poisson process $\Xi$
with
constant rate 1. If the times they are drawn are $(T_k)\xoo$, then we have
\begin{align}\label{e+-4}
  B(t)=\sum_{T_k\le t} Y_k(t-T_k),
\end{align}
where $Y_k$ are independent copies of the one-colour process $Y(t)$,
independent also of $(T_k)\xoo$.
Consequently, for $z\in\oi$ say,
 we have
 \begin{align}\label{e+-5}
   \E \bigpar{z^{B(t)}\mid \Xi}
&= \prod_{T_k\le t} g_{t-T_k}(z) = \exp\Bigpar{\sum_{T_k\le t} \log
   g_{t-T_k}(z)}
\notag\\&
= \exp\Bigpar{\intot \log g_{t-u}(z)\dd\Xi(u)}
 .\end{align}
Hence, by a standard formula for Poisson processes 
\cite[Lemma 12.2]{Kallenberg},
 \begin{align}\label{e+-6}
   \E {z^{B(t)}}&
= \E\exp\Bigpar{\intot \log g_{t-u}(z)\dd\Xi(u)}
= \exp\Bigpar{\intot \bigpar{g_{s}(z)-1}\dd s}
\notag\\&
= \exp\Bigpar{-\intot \Bigpar{\frac{1}{1-z}+\frac{s}2}\qw\dd s}
= \exp\Bigpar{-2\Bigpar{\log\Bigpar{\frac{1}{1-z}+\frac{t}2}
-\log\Bigpar{\frac{1}{1-z}}}}
\notag\\&
=\Bigpar{1+\frac{t}{2}(1-z)}\qww
=\Bigpar{\frac{\frac{2}{t+2}}{1-\frac{t}{t+2}z}}^2
 .\end{align}
Consequently, $B(t)$ has a negative binomial distribution
$\NBi\bigpar{2,\frac{2}{t+2}}$. 
In particular, 
\begin{align}
  \label{jbe}
\E B(t)=t, 
\end{align}
which also follows directly from \eqref{e+-4}.

It follows from \eqref{e+-6} that 
as \ttoo, 
for any $s>0$,
\begin{align} \label{e+-7}
\E e^{-sB(t)/t} 
=\Bigpar{1+\frac{t}{2}(1-e^{-s/t})}\qww
\to \Bigpar{1+\frac{s}{2}}\qww
\end{align}
and thus
$B(t)/t$ converges in distribution to a Gamma distribution:
\begin{align}\label{e+-8}
  t\qw B(t)\dto \gG(2,\tfrac12),
\qquad\text{as } \ttoo
.\end{align}
However, we will see that $B(t)/t$ does \emph{not} converge \as, which shows
that \refT{TC-} does not extend to this example.

To see this, we extend \eqref{e+-8} to
process convergence.
We claim that as $t\to\infty$, we have 
  \begin{align}\label{jb1}
    t\qw B(tx)\dto \cB(x):= \tfrac14 \BESQ^{4}(x)
\qquad \text{in $D\ooo$},
  \end{align}
where $\BESQ^4(x)$ denotes a squared 4-dimensional Bessel process 
\cite[Chapter XI]{RY}. Recall that 
\begin{align}\label{besq4}
\BESQ^4(x)  =|\cW(x)|^2 = \sum_{i=1}^4 \cW_i(x)^2,
\end{align}
where $\cW_1(x),\dots,\cW_4(x)$ are independent standard Brownian motions
(Wiener processes), and $\cW(x):=\bigpar{\cW_1(x),\dots,\cW_4(x)}$  thus is a
4-dimensional Brownian motion. Hence,
$\frac{1}4\BESQ^4(x)\sim\gG(2,x/2)$, in accordance with \eqref{e+-8}. 
Furthermore, \as{} $\BESQ^4(x)>0$ for every $x>0$.

The proof of \eqref{jb1} is somewhat technical and is given in \refApp{Ajb},
where we also extend the result to other initial values $(w_0,b_0)$,
see \refT{Tjb}.

Suppose now that $B(t)/t\asto \cZ$ for some random variable $\cZ$; then
$\cZ\sim\gG(2,\frac12)$ by \eqref{e+-8}.
Moreover, for every $r>0$, we would have
\begin{align}\label{jb3}
t\qw\bigpar{ B(rt)-rB(t)}=r\Bigpar{\frac{B(rt)}{rt}-\frac{B(t)}t}
\asto r(\cZ-\cZ)=0.
\end{align}
However, the process convergence in \eqref{jb1} implies finite dimension
convergence, and in particular
\begin{align}\label{jb4}
t\qw\bigpar{ B(rt)-rB(t)}
\dto \frac{1}{4}\bigpar{\BESQ^4(r)-r\BESQ^4(1)}.
\end{align}
Hence, \eqref{jb3} would imply
$\BESQ^4(r)=r\BESQ^4(1)$ \as, for every $r>0$, which obviously is false
(even for a single $r\neq1$).
This contradiction proves the claim that $B(t)/t$ does not converge a.s.
Hence, the convergence in \eqref{e+-8} holds
in distribution but not a.s.

We use \eqref{jb1} to derive corresponding results for the discrete-time urn.
Let, as above,
$\TT_n$ be the $n$th time that a ball is drawn, and 
let $N(t)$ be the total number of draws up to time $t$;
thus $N(\TT_n)=n$.
Since all balls are drawn with intensity 1, 
\begin{align}\label{jb5}
\tN(t):=N(t)-\intot\bigpar{W(s) +B(s)}\dd s 
=N(t)-t-\intot B(s)\dd s
\end{align}
is a local martingale with $\tN(0)=0$, and
it follows as in the proof of \refL{LtN}
(now using \eqref{jbe}) that $\tN(t)$ is a martingale.
In particular $\E\tN(t)=0$, and thus by \eqref{jb5} and \eqref{jbe}
  \begin{align}\label{jb5a}
\E N(t)
=t+ \E\intot B(s)\dd s
=t+ \intot\E B(s)\dd s
= t + t^2/2.
\end{align}

Furthermore, all jumps are $+1$ and thus the quadratic variation is by
\eqref{nn4} 
\begin{align}\label{jb6}
  [\tN,\tN]_t
=\sum_{0<s\le t}\gD N(s)
=N(t).
\end{align}
Consequently,
by Doob's inequality \eqref{nn6} and \eqref{jb5a},
\begin{align}\label{jb7}
  \E \tN^*(t)^2 
\le C \E   [\tN,\tN]_t
=C\E N(t)
= Ct + Ct^2.
\end{align}
In particular, $\tN^*(t)/t^2\pto0$ as \ttoo, 
which together with \eqref{jb5} implies that
\begin{align}\label{jb8}
  N(tx)/t^2 -\intox t\qw B(ty)\dd y&
= t^{-2} N(tx)-t^{-2}\int_0^{tx} B(s)\dd s
= t^{-2} \tN(tx)+t\qw x
\pto 0
\end{align}
in $D\ooo$, and consequently \eqref{jb1} implies
\begin{align}\label{jb9}
  N(tx)/t^2 \dto \cV(x):=\intox \cB(y)\dd y,
\end{align}
in $D\ooo$, 
jointly with \eqref{jb1}.
Note that $\cV(x)$ is a continuous stochastic process which strictly
increases from $\cV(0)=0$ to $\infty$.
Define $\tau$ by 
\begin{align}\label{jb0}
\cV(\tau)=1; 
\end{align}
thus $\tau$ is random with $0<\tau<\infty$ a.s.
It follows easily from \eqref{jb9}  (we omit the details) that, jointly with
\eqref{jb1},
\begin{align}\label{jc1}
  \TT_n/\sqrt n \pto \tau
\end{align}
and as a consequence
\begin{align}\label{jc2}
  B_n/\sqrt n = B(\TT_n)/\sqrt n 
\dto \cB(\tau).
\end{align}
This proves convergence in distribution of $B_n/\sqrt n$.
The limit $\cB(\tau)$ is determined by \eqref{jb1}, \eqref{jb9}, and
\eqref{jb0}; unfortunately we do not know any simpler description of the
limit distribution, and we leave
it as an open problem to find one.

For the discrete-time urn,
we thus have convergence in distribution of 
$B_n/n^{1/2}$.
(The exponent $\xfrac12$ equals $\kk_2/\hkk_0$, just as in
\refT{T1}\ref{T10} although we cannot apply that theorem.) 
However, we do not have convergence \as{} in \eqref{jc2}.
In fact, if $B_n/n\qq\asto \hcZ$ for some $\hcZ$, then
for every $x>0$, as \ttoo,
\begin{align}\label{jc3}
  \frac{B(xt)^2}{N(xt)\phantom{{}^2}}
= \frac{B_{N(xt)}^2}{N(xt)}
\asto \hcZ^2,
\end{align}
while the joint convergence of \eqref{jb1} and \eqref{jb9} implies
\begin{align}\label{jc4}
  \frac{B(xt)^2}{N(xt)\phantom{{}^2}}
=  \parfrac{B(xt)}{t}^2\cdot\frac{t^2}{N(xt)}
\dto \frac{\cB(x)^2}{\cV(x)}
\end{align}
in $D(0,\infty)$, \ie, in $D[a,b]$ for every $0<a<b<\infty$.
Consequently, \eqref{jc3} would imply that \as{}
$\cB(x)^2/\cV(x)=\hcZ^2$ for every $x>0$, and thus \as
\begin{align}\label{jc5}
  \intox\cB(y)\dd y = \cV(x) = \hcZ\qww \cB(x)^2,
\qquad x>0.
\end{align}
But this impossible, for example because \eqref{jc5} would imply that
$\cB(x)$,
and thus $\BESQ^4(x)$,
is differentiable (and, moreover, linear). This contradiction shows
that $B_n/\sqrt n$ does not converge a.s.
\end{example}

\begin{example}\label{E--}
A counterexample somewhat similar to \refE{E+-} is given by
the replacement matrix 
\begin{align}
  \matrixx{0&\phantom+1\\0&-1} .
\end{align}
In words: when we draw a white ball, it is replaced together with a black ball;
when we draw a black ball, we discard it.

We have $\gl_1=0$ and $\gl_2=-1$, and thus $\glx_1=\glx_2=0$.
Note that this example too  is excluded from \refT{TC-} since 
\ref{A-7} does not hold.
(However, again \ref{A0}--\ref{A2}, \ref{A-5}, and \ref{A-6} hold, provided
$x_1>0$ and $x_2\in\bbZgeo$.)
We will see that, as in \refE{E+-},  
$B(t)$ converges in distribution but \emph{not} a.s.

Suppose that the urn starts with a single white ball, \ie, $\bX_0=(1,0)$. 
Then, for the
continuous-time urn, $W(t)=1$ for all $t$, and thus white balls are drawn
according to a Poisson process with constant rate 1.
At each draw in this Poisson process, we add a black ball. Black balls live
an exponential time with mean 1, and then disappear.
Consequently, $B(t)$, the number of black balls, is a birth-death
process where the birth rate is constant 1 and the death rate equals the
number of particles. (Thus $\mu_k=1$ and $\gl_k=k$ in the standard notation.)

Let $T_k$ be the time of the $k$th white draw, and $L_k$ the life-length of the
black ball that then is added. Then the pairs $(T_k,L_k)$ form a Poisson process
in $\bbR_+^2$ with rate $e^{-y}\dd x\dd y$. The number of black balls at
time $t$ is
\begin{align}
  B(t)=\sumk\indic{T_k\le t} \indic{L_k> t-T_k}
=\sumk\indic{(T_k,L_k)\in D_t}
\end{align}
where $D_t:=\set{(x,y)\in\bbR^2:0<x\le t,\,y>t-x}$.
Consequently, $B(t)$ has a Poisson distribution $\Po(\mu(t))$,
where
\begin{align}
\mu(t)=\int_{D_t}e^{-y}\dd x \dd y  
=\intot\int_{t-x}^\infty e^{-y}\dd y \dd x
=\intot e^{x-t} \dd x
= 1-e^{-t}.
\end{align}

As \ttoo, we thus have convergence in distribution $B(t)\dto \Po(1)$.
Obviously, we cannot have convergence a.s., since $B(t)$ jumps $+1$ at every
$T_k$, and $T_k\to\infty$ as $k\to\infty$.

$B_n$ does not converge in distribution as \ntoo{}
for a simple parity reason: since $B_n$
changes by $\pm1$ at each draw, we have $B_n\equiv n \pmod2$.
However, $B_n$ is a 
Markov chain (since $W_n$ is constant), 
it is irreducible,
and it is easily
seen that the expected time to return to 0 is finite;
hence the Markov chain $B_n$ is  positive recurrent.
The period is 2, and 
and it follows  that  the two subsequences $B_{2n}$ and $B_{2n+1}$ 
converge in distribution to some limits, say
$\hB\even$ and $\hB\odd$.
The mixture (with equal weights) of $\hB\even$ and $\hB\odd$ has a stationary
distribution for the Markov chain, and it is easily found that the limit
distributions are given by
\begin{align}
  \P(\hB\even=k)
&=\frac{k+1}{k!}e^{-1}\indic{k \text{ is even}},
\\
  \P(\hB\odd=k)
&=\frac{k+1}{k!}e^{-1}\indic{k \text{ is odd}}.
\end{align}
These can be described as
$\Po(1)$ rounded up to nearest even or odd integer, respectively.

Obviously, we do not have convergence a.s., even for these subsequences.
\end{example}

\begin{example}\label{ELlogL}
  Consider a diagonal urn with $q=2$ and replacement matrix
  $\smatrixx{\xi_{11}&0\\0&\xi_{22}}$ as in \cite{Athreya1969}; cf.\ the
  deterministic case in \refE{ED}.
Assume  that $\xi_{11},\xi_{22}\in\bbZgeo$ a.s., 
that $\E\xi_{11}=\E\xi_{22}=1$, and that
$\E\xi_{22}^2<\infty$
but $\E\xi_{11}\log\xi_{11}=\infty$, so that \ref{A2} does \emph{not} hold.
(We may simply take $\xi_{22}=1$ a.s.)
Then the stochastic processes 
$X_1(t)$ and $X_2(t)$ are independent Markov branching processes, and
\cite[Theorem III.7.2]{AN} shows that, as \ttoo,
\begin{align}\label{ellogl1}
  e^{-t}X_1(t)\asto 0.
\end{align}
Hence, \eqref{tc1} holds with $\cX_1=0$ a.s., while \refT{TC} (applied to
colour 2 only), or  \cite[Theorem III.7.2]{AN} again, shows that
\eqref{tc1} holds also for $i=2$ with $\cX_2>0$ a.s.
It follows that $X_1(t)/X_2(t)\asto0$ as \ttoo, and thus
$X_{n1}/X_{n2}\asto0$ as \ntoo.
It follows easily that \eqref{t1b} holds, which in this case is
$X_{ni}/n\asto\hcX_i$, with $\hcX_1=0$ and $\hcX_2=1$;
we thus have 
 $\cX_1=\hcX_1=0$, in contrast to \refTs{T1} and \ref{TC}.

This shows that the main results in the present paper  do not hold without
assuming at least $\E\xi_{ij}\log\xi_{ij}<\infty$. 
We have for convenience assumed the stronger second moment condition \ref{A2}, 
but as said in \refR{Rmom}, we conjecture that it can be weakened.
\end{example}

\begin{ack}
  I thank Allan Gut for help with references.
\end{ack}

\appendix
\section{Absolute continuity and conditioning}\label{AA}
In this appendix we state three general lemmas on absolute continuity of
distributions and conditioning.
We find them intuitively almost obvious, but only almost,
and since we do not know any references,
we provide complete proofs. 
For two measures $\mu$ and $\gl$ on the same space, we let $\mu\ll\gl$
denote that $\mu$ is absolutely continuous with respect to $\gl$, \ie, that
$\gl(B)=0\implies \mu(B)=0$.

We recall some further standard definitions:

A measure space $(\fX,\cX)$ is a \emph{Borel space} if it is (or is
isomorphic to) a Borel set in a complete separable metric space with its
Borel \gsf,
see \cite[Appendix A1]{Kallenberg}.
This includes, for example, $\bbR$, $\bbR^n$, and the function
space $D\ooo$; moreover, any finite or countable product of Borel spaces is a
Borel space.

If 
$(\fX,\cX)$ and $(\fY,\cY)$ are two measurable spaces, then a
\emph{probability kernel} from $\fX$ to $\fY$ is a mapping
$\mu:\fX\times\cY\to\oi$ such that 
$B\mapsto \mu(x,B)$ is a probability measure on $(\fY,\cY)$ for every fixed
$x\in \fX$, and furthermore
$x\mapsto \mu(x,B)$ is measurable on $(\fX,\cX)$ for every fixed $B\in \cY$, 
see \cite[p.~20]{Kallenberg}.

If $X$ and $Y$ are random variables with values in measurable spaces
$(\fX,\cX)$ and $(\fY,\cY)$, respectively,  then a \emph{regular conditional
  distribution} of $X$, given $Y$, is a probability kernel
$\mu$ from $\fY$ to $\fX$ such that for any fixed $B\in\cX$,
\begin{align}\label{aa1}
  \mu(Y,B)=\P\sqpar{X\in B\mid Y}
\qquad \text{a.s.},
\end{align}
see \cite[p.~106--107]{Kallenberg}.
It follows that for any measurable $f:\cX\to\ooox$,
\begin{align}\label{aa11}
\E\sqpar{f(X)\mid Y}=  \int_{\fX}f(x)\mu(Y,\ddx x)
\qquad \text{a.s.}
\end{align}
Similarly \cite[(7) on p.~108]{Kallenberg}, for any measurable
$f:\cX\times\cY\to\ooox$, 
\begin{align}\label{aa12}
\E f(X,Y)= \E \int_{\fX}f(x,Y)\mu(Y,\ddx x).
\end{align}

If $(\fX,\cX)$ is a Borel space, then a such a
regular conditional distribution $\mu$ exists, and the probability
measure $\mu(y,\cdot)$ is $\cL(Y)$-\aex{} unique 
(in the standard sense that two different such kernels are equal for
$\cL(Y)$-\aex{} $y\in \cY$)
\cite[Theorem 6.3]{Kallenberg}. 

\begin{lemma}\label{LX}
Let $X,Y,Z$ be random variables taking values in Borel spaces $(\fX,\cX)$,
$(\fY,\cY)$, $(\fZ,\cZ)$, respectively, and suppose that $Z=\gf(Y)$ for some
measurable function $\gf:\fY\to \fZ$.  
Let $\gl$ be a measure on $\fX$, and 
suppose that 
the regular conditional distribution $\mu(y,\cdot)$ of $X$ given $Y$ 
is absolutely continuous with respect to $\gl$ for $\cL(Y)$-\aex{} $y\in\fY$.
Then
the regular conditional distribution $\mu'(z,\cdot)$ of $X$ given $Z$ 
is absolutely continuous with respect to $\gl$ for $\cL(Z)$-\aex{} $z\in\fZ$.
\end{lemma}

\begin{remark}\label{RLX}
  Although \refL{LX} is stated for conditionings on single random variables
  $Y$ and 
  $Z$, it holds also for conditionings on finite or countably infinite sequences
of random variables (taking values in possibly different Borel spaces), 
since such sequences can be regarded as a single
variable in a suitable product space.
\end{remark}

\begin{proof}
  If $B\subset \fX$ is any set with $\gl(B)=0$, then $\mu(Y,B)=0$ a.s., and
  thus, by \eqref{aa1},
  \begin{align}\label{aa2}
    \mu'(Z,B)=\E\sqpar{\etta_B(X)\mid Z}
=\E\bigsqpar{\E\sqpar{\etta_B(X)\mid Y}\mid Z}
=\E\sqpar{\mu(Y,B)\mid Z}
=0
\quad \text{a.s.}
  \end{align}
However, this is for a fixed $B$, while the conclusion of the lemma is that 
\as{} \eqref{aa2} holds simultaneously for every $\gl$-null set $B\subset\fX$.
There is in general an uncountable number of $\gl$-null sets $B\subset\fX$,
and we do not see how to use the argument in \eqref{aa2} to prove the result. 

Instead, we argue as follows.
First, by if necessary changing $\mu(y,\cdot)$ on a $\cL(Y)$-null set
of $y$, we may assume that 
\begin{align}\label{aa21}
  \mu(y,\cdot)\ll\gl \qquad \text{for every $y\in\fY$}.
\end{align}

Let $\nu(z,\cdot)$ be the regular conditional distribution of $Y$ given $Z$.
Then, for any $B\in\cX$, a.s., using \eqref{aa1} and \eqref{aa11},
\begin{align}\label{aa3}
\P[X\in B\mid Z]&=
\E\sqpar{\etta_B(X)\mid Z}
=\E\bigsqpar{\E\sqpar{\etta_B(X)\mid Y}\mid Z}
=\E\sqpar{\mu(Y,B)\mid Z}
\notag\\&
=\int_{\fY}\mu(y,B)\nu(Z,\ddx y)
.\end{align}
This shows that (a version of) the regular conditional distribution $\mu'$
is given by the composition of the kernels $\nu$ and $\mu$ defined by
\begin{align}\label{aa4}
  \mu'(z,B):=\int_{\fY}\nu(z,\ddx y)\mu(y,B),
\qquad B\in\cX;
\end{align}
note that this composition is a probability kernel, see \eg{} the more general
\cite[Lemma 1.41(iii)]{Kallenberg}.

Now, if $\gl(B)=0$, then \eqref{aa21} shows that
$\mu(y,B)=0$ for every $y$, and hence \eqref{aa4} yields
$\mu'(z,B)=0$ for every $z\in\fZ$.
Consequently, $\mu'(z,\cdot)\ll\gl$ for every $z\in\fZ$.
\end{proof}

\begin{lemma}\label{LY}
Let $X$ and $Y$ be random variables taking values in Borel spaces
$(\fX,\cX)$ and $(\fY,\cY)$, respectively.
Let $\gl$ and $\gl'$ be $\gs$-finite measures on $\fX$ and $\fY$,
respectively.
Suppose that $\cL(Y)\ll\gl'$ and 
that the regular conditional distribution $\mu(y,\cdot)$ of $X$ given $Y$ 
satisfies $\mu(y,\cdot)\ll \gl$ for $\cL(Y)$-\aex{} $y\in \fY$.
Then the distribution of $(X,Y)$ in $\fX\times \fY$ is absolutely continuous
with respect to $\gl\times\gl'$.
\end{lemma}

\begin{proof}
By if necessary changing $\mu(y,\cdot)$ on a $\cL(Y)$-null set of $y\in\fY$,
we may assume that \eqref{aa21} holds. 

  Let $B\subset \fX\times \fY$ with $\gl\times\gl'(B)=0$.
For $y\in\fY$, let $B_y:=\set{x\in\fX:(x,y)\in B}$.
Let $A:=\set{y\in\fY:\gl(B_y)>0}$.
By Fubini's theorem,
\begin{align}
  0
=\gl\times\gl'(B)
=\int_{\fX\times\fY}\etta_B(x,y)\dd\gl(x)\dd\gl'(y) 
=\int_{\fY}\gl(B_y)\dd\gl'(y)
\end{align}
and thus 
$\gl(B_y)=0$ for $\gl'$-a.e.\ $y$, i.e.,
$\gl'(A)=0$. Since $\cL(Y)\ll\gl'$, this implies 
\begin{align}\label{aa6}
\P(Y\in A)=0.  
\end{align}

Furthermore, by \eqref{aa12},
\begin{align}\label{aa7}
\P\sqpar{(X,Y)\in B} &=
\E\etta_B(X,Y) 
=\E\int_{\fX} \etta_B(x,Y)\mu(Y,\ddx x)
=\E\int_{\fX} \etta_{B_Y}(x)\mu(Y,\ddx x)
\notag\\&
=\E \mu(Y,B_Y).
\end{align}
If $Y\notin A$, then $\gl(B_Y)=0$, and thus $\mu(y,B_Y)=0$ for every $y$
by \eqref{aa21};
in particular $\mu(Y,B_Y)=0$. By \eqref{aa6}, this shows that $\mu(Y,B_Y)=0$
a.s., and thus \eqref{aa7} yields
$\P\sqpar{(X,Y)\in B}=0$.
\end{proof}

For easy reference, we state also an elementary result on absolute
continuity in $\bbR^d$, as in the main part of the paper this tacitly means
with respect to Lebesgue measure.

\begin{lemma}
  \label{LZ}
Let $T:\bbR^n\to\bbR^m$ be a linear operator,
where $1\le m\le n$.
If $\bX$ is a random vector in $\bbR^n$ with an \abscont{}
distribution, and $T$ is onto, then the distribution of $T(\bX)$ in $\bbR^m$
is \abscont.
\end{lemma}
\begin{proof}
  By changes of bases, we may assume that $T$ is the projection to the first
  $m$ coordinates. 
Let $\gl_d$ denote the Lebesgue measure in $\bbR^d$. 
If $A\subset\bbR^m$ with $\gl_m(A)=0$, then
\begin{align}
  \P\bigpar{T(\bX)\in A} = \P \bigpar{\bX\in A\times \bbR^{n-m}}=0,
\end{align}
since $\gl_n(A\times\bbR^{n-m})=0$.
\end{proof}

\section{$L^p$ theory}\label{ALp}

The proofs in this paper frequently use martingales and $L^2$ theory,
in particular the identity \eqref{nn5}.
In this appendix, we extend the results to $L^p$ estimates for any $p>1$
by combining the arguments in  \refS{S1} with
the Burkholder--Davis--Gundy inequalities
(see \eg{} \cite[Theorem 26.12]{Kallenberg}), 
which say that
if $p\ge1$,
then there exist constants $c=c(p)$ and $C=C(p)$ 
such that for every (\ctime)
local martingale $M(t)$
\begin{align}  \label{BDG}
c \E[M,M]_t^{p/2} \le
\E M^*(t)^p
\le C \E[M,M]_t^{p/2},
\qquad 0\le t\le \infty. 
\end{align}
(All constants in this appendix may depend on the exponent $p$.)
We will mainly use the second inequality.

This extension to $L^p$ leads to two major results.
Using the case $p>2$, we will obtain a proof of \refT{TMCp} (and therefore
\refT{TMD}) showing convergence in $L^p$ and thus
moment convergence  under natural conditions.
Moreover, using the case $1<p<2$, we show that, as said in
\refR{Rmom}, our main results hold also if we weaken the $L^2$ condition
\ref{A2} to $L^p$ for some $p>1$.
More precisely, we will show the following.

\begin{theorem}
  \label{Tp}
\refTs{T1} and \ref{TC}, and their extensions \refTs{T1-}--\ref{TC-},
all hold also if \ref{A2} is replaced by the weaker
\begin{PQenumerate}{\ALnude{p}}
\item \label{Ap}
$\E|\xi_{ij}|^p<\infty$ for all $i,j\in\sQ$ and some $p>1$.
\end{PQenumerate}
\end{theorem}

Also other  results in this paper,
for example the results on the drawn colours in \refS{Sdraw},  hold if
\ref{A2} is replaced by 
\ref{Ap}, provided we replace any $L^2$-norms by  $L^p$ norms $\norm{\;}_p$;
see also \refR{Rmomp} for
the results on moments in \refS{Smoments}.
We leave the details to the reader.

We will basically follow the arguments in the main part of the paper,
replacing $L^2$ estimates by $L^p$ estimates, but sometimes the details of
the arguments will differ. Moreover, we have chosen to first focus on obtaining
the $L^p$ estimates, leading to the proof of \refT{TMCp} (partly because
this seems to be of greater interest for applications); 
we then return to the arguments yielding \as{} convergence and \refT{Tp}.
As before, we argue in several steps.

\subsection{A single colour not influenced by others}\label{SS11p}

We begin with a colour $i$ that is not influenced by any other
(i.e., $i\in\sQmin$), and prove
an $L^p$-version of \refLs{LM} and \ref{ZLM}.

\begin{lemma}\label{LMp}
Assume \ref{A0}--\ref{A3}, \ref{A-5} (or \ref{A+}),
and \ref{Ap} for some $p>1$.
  Let $i\in\sQmin$, and assume 
  \begin{align}\label{ALMp}
 \text{either}\quad   
i\notin\sQm \text{ (i.e., $\xi\iii\ge0$ a.s.)}
\quad\text{or}\quad
\gl_i>0.
  \end{align}
Then
\begin{romenumerate}
  
\item \label{LMp1}
The martingale $e^{-\gl_i t}X_i(t)$ is $L^p$-bounded, and 
thus the \as{} limit \eqref{lm1} holds for some limit $\cX_i$,
and
\begin{align}\label{lmp1}
\tXX_i:=  \sup_{t\ge 0} \bigset{e^{-\gl_i t} X_i(t)}\in L^p.
\end{align}

\item \label{LMp2}
Let $(T_k)\xoo$ be the times a ball of colour $i$ is drawn, and
let $\eta_k$ be the number of balls of colour $i$ that are added at
time $T_k$.
Let $0<q\le p$, and
let $f:\bbR\to\bbR$ be a function such that $\E |f(\xi_{ii})|^q<\infty$.
Finally, let $\mu>0$ be such that 
$(1\bmin q)\mu>\gl_i$.
Then (with \as{} convergent sum)
\begin{align}\label{lmp2}
  \sumk e^{-\mu T_k}f(\eta_k) \in L^q.
\end{align}
\end{romenumerate}
\end{lemma}
 
The statement in \ref{LMp1} that the martingale is  $L^p$-bounded, is 
(since we have $p>1$)
by Doob's inequality equivalent to \eqref{lmp1}, but we state both for
emphasis. Moreover, the statements are equivalent to 
$\cX_i:=\lim_\ttoo e^{-\gl_i t}X_i(t)\in L^p$.
Note also that the definition of $\tXX_i$ in \eqref{lmp1} agrees with
\eqref{j4*} since $\kk_i=0$ and $\glx_i=\gl_i$ when $i\in\sQmin$.
Similarly, \eqref{ALMp} is equivalent to \ref{A-7} for $i$,
but for later use we prefer the form \eqref{ALMp}.

\begin{proof}
\resetstepx
If $\gl_i\le0$, then by \eqref{ALMp} we have $i\notin\sQm$ and thus
$\xi\iii\ge0$ \as; consequently $\xi\iii=0$ \as{} and $X_i(t)$ is constant
so the results are trivial (with an empty sum in \eqref{lmp2}).
We thus assume $\gl_i>0$.

\stepx{\ref{LMp2} for $q=1$}\label{LMp-I}
Since $\eta_k\eqd\xi\iii$ and is independent of $T_k$, we have
by \refL{L9+}\ref{L9+b},
letting \ttoo,
\begin{align}
  \E \sumk e^{-\mu T_k}\bigabs{f(\eta_k)}
= a_i\E|f(\xi_{ii})|\intoo e^{-\mu s}\E X_i(s)\dd s,
\end{align}
which is finite by \eqref{cb4} and the assumption $\mu>\gl_i$.

\stepx{\ref{LMp2} for $q<1$}\label{LMp-II}
We have
\begin{align}
  \Bigpar{\sumk e^{-\mu T_k}\bigabs{f(\eta_k)}}^q
\le \sumk e^{-q\mu T_k}\bigabs{f(\eta_k)}^q \in L^1,
\end{align}
by \refStep{LMp-I} applied to $|f|^q$ and $q\mu$.

\stepx{If \ref{LMp1} holds for some $p>1$, then \ref{LMp2} holds for all
  $q\le p$}\label{LMp-III}
We have already proved the case $q\le1$, so we may assume $q>1$.
In particular, $\E|f(\xi_{ii})|<\infty$.
Furthermore, by induction (on $\ceil{\log_2 q}$), we may assume that
\ref{LMp2} holds if $q$ is replaced by $q/2$.

We consider first two special cases, and then the general one.

\pfcase{$\E f(\xi\iii)=0$}
In this case,
\begin{align}\label{C7}
M(t):=  \sumk \indic{T_k\le t}e^{-\mu T_k}f(\eta_k)
\end{align}
is a local martingale with $M(0)=0$, since each $T_k$ is a stopping time and
$f(\eta_k)$ has mean 0 and is independent of  $\cF_{T_k}$.
(Cf.\ $Z_2$ in the proof of \refL{L12}.) 
The quadratic variation is by \eqref{nn4} (\cf{} \eqref{j21})
\begin{align}\label{feb12}
[M,M]_t =  \sumk \indic{T_k\le t}e^{-2\mu T_k}f(\eta_k)^2.
\end{align}
We have $f(\xi\iii)^2\in L^{q/2}$, so by the induction hypothesis, we have
$[M,M]_\infty\in L^{q/2}$. (If $q/2<1$, note that $(q/2)2\mu=q\mu>\mu>\gl_i$.)
Consequently, \eqref{BDG} shows that $M$ is an $L^q$-bounded martingale, which
yields \eqref{lmp2}.

\pfcase{$f=c$ is a constant}
It suffices to consider the case $c=1$.
We now define
\begin{align}\label{feb10}
M(t):= \sumk \indic{T_k\le t}e^{-\mu T_k}-\intot e^{-\gl_is}a_iX_i(s)\dd s
\end{align}
and note that $M(t)$ is a local martingale
(Cf.\ $Z_3$ in the proof of \refL{L12} and \eqref{j31}.) 
The quadratic variation is  (\cf{} \eqref{j33})
\begin{align}
[M,M]_t =  \sumk \indic{T_k\le t}e^{-2\mu T_k}.
\end{align}
Hence, in this case too, the induction hypothesis yields 
$[M,M]_\infty\in L^{q/2}$, and thus, \eqref{BDG} shows that $M(t)$ is an
$L^q$-bounded martingale.
Furthermore,
\begin{align}\label{pj1}
\sumk e^{-\mu T_k}= M(\infty) + \intoo e^{-\mu s}a_iX_i(s)\dd s.
\end{align}
By assumption, \ref{LMp1} holds, and thus $\tXX_i\in L^p\subseteq L^q$ by
\eqref{lmp1}.
Furthermore, since $\mu>\gl_i$,
\begin{align}\label{pj2}
  \intoo e^{-\mu s}X_i(s)\dd s \le 
\intoo e^{-\mu s+\gl_i s}\tXX_i\dd s = (\mu-\gl_i)\qw \tXX_i
\in L^q.
\end{align}
This and \eqref{pj1} show that $\sum_k e^{-\mu T_k}\in L^q$, which is
\eqref{lmp2} in this case.

\pfcase{General $f$}
We use the decomposition $f(x)=\bigpar{f(x)-\E f(\xi\iii)} + \E f(\xi\iii)$
and the two preceding cases.

\stepx{\ref{LMp1} holds for all $p>1$}
By induction (on $\ceil{\log_2p}$), we may for $p>2$ assume that
\ref{LMp1} holds if $p$ is replaced by $p/2$.

As in the proof of \refL{LM}, we let $M(t):=e^{-\gl_it}X_i(t)$, so that
$M(t)$ is a martingale with quadratic variation \eqref{lm4}:
\begin{align}\label{pj4}
  [M,M]_t
= X_i(0)^2+\sumk \indic{T_k\le t} e^{-2\gl_iT_k}\eta_k^2.
\end{align}
If $1<p\le2$, 
we apply \ref{LMp2} with  $q=p/2\le1$;
this case holds by \refStep{LMp-I} or \refStep{LMp-II}. 
If $p>2$,
we apply \ref{LMp2} with  $q=p/2$ and $p$ replaced by $p/2>1$;
this case holds by the induction hypothesis and \refStep{LMp-III}.
In both cases,
we take  $f(x)=x^2$ and note that $\E |f(\xi\iii)|^{q}=\E|\xi\iii|^p<\infty$.
Hence, taking $\mu:=2\gl_i$, \eqref{lmp2} shows that $[M,M]_\infty\in L^{p/2}$.
Consequently, \eqref{BDG} shows that $M$ is an $L^p$-bounded martingale 
and that
\eqref{lmp1} holds. 
This completes the proof.
\end{proof}

We will also need a quantitative version of \refL{LMp}\ref{LMp1}.
\begin{lemma}\label{LMp99}
  Under the assumptions of \refL{LMp}, 
we have
\begin{align}\label{lmp99}
\Bignorm{ \sup_{t\ge 0} \bigabs{e^{-\gl_i t} X_i(t)}}_p\le C \ceil{X_i(0)},
\end{align}
where $C$ does not depend on $X_i(0)$.
\end{lemma}
\begin{proof}
  It ought to be straightforward to keep track of the norms of all quantities
  in the proof of \refL{LMp}, but it seems simpler to argue as follows.
First, if $X_i(0)=m$ is an integer, then the process $X_i(t)$ can be seen as
the sum of $m$ independent and identially distributed processes
$X_i\kkk(t)$,
$k=1,\dots,m$, each started with $X_i\kkk(0)=1$ (but otherwise the same as
$X_i(t)$).
Hence, 
\begin{align}
 \sup_{t\ge 0} \bigabs{e^{-\gl_i t} X_i(t)}
\le 
\sum_{k=1}^m  \sup_{t\ge 0} \bigabs{e^{-\gl_i t} X_i\kkk(t)},
\end{align}
and \eqref{lmp99} follows by \eqref{lmp1} and Minkowski's inequality.

If $X_i(0)$ is not an integer, 
let $X_i'(t)$ be an independent copy of $X_i(t)$ started with
$X_i'(0)=\ceil{X_i(0)}-X_i(0)$. Then $X_i(t)+X'_i(t)$ is a copy of the same
process started with $\ceil{X_i(0)}$.
Since $0\le X_i(t)\le X_i(t)+X'_i(t)$, the result follows from the special
case just treated.
\end{proof}

\subsection{A colour only produced by one other colour}\label{SS12p}

Consider now the situation in \refLs{L12} and \ref{ZL12}, with two colours
$i$ and $j$ such that $\sP_i=\set{j}$, and also $X_i(0)=0$.
We fix such $i$ and $j$ throughout this subsection. (We may repeat the
assumptions for emphasis.)

As in \refSS{SS12}, let $0<T_1<T_2<\dots$ be the times that a ball of
colour $j$ is drawn,
and let 
$\cF_{T_k}$ be the corresponding \gsf{s}.

Recall the notation \eqref{j4*}.
Define also, for $\mu\in\bbR$,
\begin{align}\label{kkx}
  \kkx(\mu):=
  \begin{cases}
    \kk_j,&\mu<\glx_j,\\
    \kk_j+1,&\mu=\glx_j,\\
    0,&\mu>\glx_j.
  \end{cases}
\end{align}
Note that, by \eqref{kk}, $\kk_i=\kkx(\gl_i)$. 
(We are mainly interested in the case $\mu=\gl_i$, but we use induction
to prove \refL{L12p0} below, and we will then need more general $\mu$.)

\begin{lemma}\label{LZ4}
Let $\mu\in\bbR$ and define
\begin{align}\label{pq41}
V(t)&:= \intot  e^{-\mu s}X_j(s)\dd s
.\end{align}
Then
\begin{align}
\label{pq42}
\tVV&:=
\sup_{t\ge0}\bigcpar{ (t+1)^{-\kkx(\mu)}e^{-(\glx_j-\mu)_+t}|V(t)|}
\le C\tXX_j
.\end{align}
\end{lemma}
\begin{proof}
 By \eqref{j4*} we have,
considering the three cases in \eqref{kkx} separately,
\begin{align}\label{pq94}
V(t) 
\le 
 \intot  (s+1)^{\kk_j}e^{(\glx_j-\mu) s} \tXX_j\dd s 
\le C\tXX_j (t+1)^{\kkx(\mu)}e^{(\glx_j-\mu)_+t},
\end{align}
and \eqref{pq42} follows.
\end{proof}

\begin{lemma}\label{L12p0}
Assume \ref{A0}--\ref{A3}, \ref{A-5} (or \ref{A+}),
and \ref{Ap} for some $p>1$. 
Suppose that
$i,j\in\sQ$ are such that $\sP_i=\set{j}$ and $X_i(0)=0$,
and suppose also that
\begin{align}\label{j2p}
\tXX_j\in L^p.
\end{align}
Let $(\zeta_k)\xoo$ be a sequence of random variables with the same
distribution such that $\zeta_k$ is independent of $\cF_{T_k}$.
Let $\mu\in\bbR$ be such that
\begin{align}\label{xp1}
  \begin{cases}
    \mu\ge0,& \text{if }i\notin\sQm,\\
  \mu\lor\glx_j>0,& \text{if }i\in\sQm,\\
  \end{cases}
\end{align}
and let
\begin{align}\label{pq1}
  Z(t):=
\sumk \indic{T_k\le t} e^{-\mu T_k}\zeta_k.
\end{align}
If $1< q\le p$ and
$\E|\zeta_k|^q<\infty$,
then
\begin{align}\label{pq2}
\tZZ:=
\sup_{t\ge0}\bigcpar{ (t+1)^{-\kkx(\mu)}e^{-(\glx_j-\mu)_+t}|Z(t)|}
\in L^q.  
\end{align}
\end{lemma}

\resetCase
\resetcase
\resetcasex

\begin{proof}
The proof is similar to \refStep{stepZ123b} in the proof of \refL{L12}
(which essentially is the case $q=2$
of the present lemma), now using \eqref{BDG}.

By induction (on $\ceil{\log_2q}$), we may for $q>2$ assume that the lemma
holds for $q/2$. 

\pfCase{$\E\zeta_k=0$}\label{CQ1}
In this case, $Z(t)$ is a local martingale, for the same reason as $M(t)$ in
\eqref{C7}. Its quadratic variation is, by \eqref{nn4} again,
\begin{align}\label{pq3}
  [Z,Z]_t = \sumk \indic{T_k\le t} e^{-2\mu T_k}|\zeta_k|^2.
\end{align}
We consider two subcases.

\pfcase{$1< q\le2$}
By \eqref{BDG}, together with
\eqref{pq3} and the independence of $\zeta_k$ and $T_k$, we have, since
$q/2\le1$,
\begin{align}\label{pq11}
  \E Z^*(t)^q&\le 
C \E [Z,Z]_t^{q/2}
= C \E\Bigpar{\sumk \indic{T_k\le t} e^{-2\mu T_k}|\zeta_k|^2}^{q/2}
\notag\\&
\le C\E\sumk \indic{T_k\le t} e^{-q\mu T_k}|\zeta_k|^q
.\end{align}
Hence, \refL{L9+}\ref{L9+b} and \eqref{j4*} yield
\begin{align}\label{pq12}
  \E Z^*(t)^q& 
\le C a_j \E \intot e^{-q\mu s}X_j(s)\dd s
\le C \E\tXX_j \intot (s+1)^{\kk_j}e^{(\glx_j-q\mu)s}\dd s.
\end{align}
In the sequel, we allow constants $C$ to depend on $\norm{\tXX_j}_p$
(which is finite by assumption); hence we may absorb $\E\tXX_j$ into $C$ in
\eqref{pq12}. 

We consider three subsubcases:

\pfcasexx{$\glx_j< q\mu$}
In this case, we may take $t=\infty$ in \eqref{pq12}
and obtain $Z^*(\infty)\in L^q$. The result \eqref{pq2} follows since
$\tZZ\le Z^*(\infty)$.

\pfcasexx{$\glx_j\ge q\mu$ and $\glx_j>0$}
Define, similarly to \eqref{j53},
\begin{align}\label{pq13}
  \tZc(n)
&:=\sup_{n-1\le t\le n} {(t+1)^{-\kkx(\mu)}e^{-(\glx_j-\mu)t}|Z(t)|}
\le 
C n^{-\kkx(\mu)}e^{-(\glx_j-\mu) n} Z^*(n)
.\end{align}
By \eqref{pq12}, we thus have
\begin{align}\label{pq14}
\E  \tZc(n)^q&
\le 
C n^{-q\kkx(\mu)}e^{-q(\glx_j-\mu) n} 
  \int_0^n  (s+1)^{\kk_j}e^{(\glx_j-q\mu)s} \dd s
\notag\\&
\le 
C n^{1+\kk_j-q\kkx(\mu)}e^{(-q(\glx_j-\mu)+\glx_j-q\mu) n}
=C n^{1+\kk_j-q\kkx(\mu)}e^{-(q-1)\glx_j n}
.\end{align}
We have $(q-1)\glx_j>0$,
and thus \eqref{pq14} implies
\begin{align}\label{pq16}
  \E (\tZZ)^q \le \E \sumn  \tZc(n)^q <\infty,
\end{align}
which shows \eqref{pq2}.

\pfcasexx{$\glx_j\ge q\mu$ and $\glx_j\le0$}
Then also $\mu\le0$, and thus $\mu\lor\glx_j\le0$.
By \eqref{xp1}, we must have $i\notin\sQm$ and $\mu=0$, and then also
$\glx_j=0$. 
Hence, $\kk(\mu)=\kk_j+1$ by \eqref{kkx}. 
We define, similarly to \eqref{j3c1},
\begin{align}\label{pq17}
  \tZcc(n)
&:=\sup_{2^{n-1}\le t\le 2^n}t^{-\kkx(\mu)}e^{-(\glx_j-\mu) t} |Z(t)|
=\sup_{2^{n-1}\le t\le 2^n}t^{-\kk_j-1} |Z(t)|
.\end{align}
Similarly to \eqref{pq14}, it follows from \eqref{pq12} that
\begin{align}\label{pq18}
\E \tZcc(n)^q&
\le C 2^{-q(\kk_j+1)n}\int_0^{2^n} (s+1)^{\kk_j}\dd s
\le C 2^{(1-q)(\kk_j+1)n}
\le C 2^{-(q-1)n}
.\end{align}
Hence, \eqref{pq2} follows by
\begin{align}\label{pq19}
  \E (\tZZ)^q \le \E Z^*(1)^q+\E \sumn  \tZcc(n)^q <\infty,
\end{align}
since $\E Z^*(1)^q<\infty$ by \eqref{pq12}.

\pfcase{$q>2$}
By \eqref{BDG}, together with \eqref{pq3} 
and the induction hypothesis 
(for $q/2$ and $2\mu$, using $\E|\zeta_k^2|^{q/2}<\infty$)
\begin{align}\label{pq21}
  \E Z^*(t)^q \le C\E [Z,Z]_t^{q/2}
\le C\bigpar{ (t+1)^{\kkx(2\mu)}e^{(\glx_j-2\mu)_+t}}^{q/2}
.\end{align}

Again, we consider three subsubcases:
\resetcasex

\pfcasexx{$\glx_j<2\mu$}
Then $\kkx(2\mu)=0$ by \eqref{kkx}, and thus we may let \ttoo{} in \eqref{pq21}
and obtain $\E Z^*(\infty)^q<\infty$, which yields \eqref{pq2} since
$\tZZ\le Z^*(\infty)$.

\pfcasexx{$\glx_j\ge2\mu$ and $\glx_j>0$}
Define again $\tZc(n)$ by \eqref{pq13}. Then, by \eqref{pq21},
\begin{align}\label{pq22}
  \E  \tZc(n)^q&
\le 
C n^{-q\kkx(\mu)+\frac{q}{2}\kkx(2\mu)}e^{(-q(\glx_j-\mu)+\frac{q}{2}(\glx_j-2\mu)) n}
\notag\\&
\le
C n^{-q\kkx(\mu)+\frac{q}{2}\kkx(2\mu)}e^{-\frac{q}{2}\glx_j n}
.\end{align}
Hence, we have again \eqref{pq16}, and thus \eqref{pq2}.

\pfcasexx{$\glx_j\ge2\mu$ and $\glx_j\le0$}
Again, by \eqref{xp1}, 
we must have $i\notin\sQm$ and $\mu=0$, and then also
$\glx_j=0$. Hence, $\kkx(\mu)=\kkx(2\mu)=\kk_j+1$.
Define again $\tZcc(n)$ by \eqref{pq17}. Then, by \eqref{pq21},
\begin{align}
  \E  \tZcc(n)^q&
\le 
C 2^{(-q(\kk_j+1)+\frac{q}{2}\kkx(2\mu))n}
=
C 2^{-\frac{q}2(\kk_j+1)n}
.\end{align}
Consequently, \eqref{pq19} holds, and thus \eqref{pq2}.

\pfCase{$\zeta_k=c$ is a constant}\label{CQ2}
We may assume $c=1$.

Let $V(t)$ be as in \eqref{pq41} and define
\begin{align}
\label{pq91}
  M(t)&:=
Z(t)-a_jV(t)
=
\sumk \indic{T_k\le t} e^{-\mu T_k}
-a_j\intot  e^{-\mu s}X_j(s)\dd s.
\end{align}
Then $M(t)$ is a local martingale by the same argument as for 
\eqref{feb10} (i.e., as for $Z_3$ in the
proof of \refL{L12}), and its quadratic variation is
(\cf{} \eqref{j33})
\begin{align}\label{pq92}
 [M, M]_t=
\sumk \indic{T_k\le t} e^{-2\mu T_k}.
\end{align}
This is the same as in \eqref{pq3}, except for the factor $|\zeta_k|^2$ there
(or the same if we choose $\zeta_k=\pm1$).
Consequently, the argument in \refCase{CQ1} yields
\begin{align}\label{pq93}
\tMM:=
\sup_{t\ge0}\bigcpar{ (t+1)^{-\kkx(\mu)}e^{-(\glx_j-\mu)_+t}|M(t)|}
\in L^q.  
\end{align}
Furthermore, \eqref{pq42} and the assumption \eqref{j2p} yield
$\tVV\in L^p\subseteq L^q$.
Consequently, \eqref{pq2} follows by \eqref{pq91}.

\pfCase{General $\zeta_k$}\label{CQ3}
The result follows from Cases \ref{CQ1} and \ref{CQ2} by the decomposition
$\zeta_k=(\zeta_k-\E\zeta_k)+\E\zeta_k$.
\end{proof}

\begin{lemma}\label{L12p1}
Assume \ref{A0}--\ref{A3}, \ref{A-5} (or \ref{A+}),
and \ref{Ap} for some $p>1$.
Suppose that
$i,j\in\sQ$ are such that $\sP_i=\set{j}$ and $X_i(0)=0$, and 
suppose also that \eqref{ALMp} holds.
If\/
$\tXX_j\in L^p$, 
then
$\tXX_i\in L^p$.
\end{lemma}
 
\begin{proof}
 We use again the decomposition \eqref{j3}, where we recall
that $Y_k(t)$ denote copies of the  one-colour
process in \refSS{SS11}, and that the process $Y_k(t)$ is independent of 
$\cF_{T_k}$.
Let 
\begin{align}\label{pp1}
\zeta_k:=
\sup_{t\ge0}\bigset{e^{-\gl_i t}Y_k(t)}
.\end{align}
Then \eqref{j3} implies
\begin{align}\label{pp2}
e^{-\gl_i t} X_i(t)
&=\sumk  \indic{T_k\le t}e^{-\gl_i  T_k}\cdot{e^{-\gl_i(t-T_k)}Y_k(t-T_k)}
\notag\\&
\le
\sumk \indic{T_k\le t} e^{-\gl_i T_k} \zeta_k
=:Z(t)
.\end{align}

Let $\eta_k=Y_k(0)$, and recall that $\eta_k\eqd\xi_{ji}\in L^p$.
Then, conditioning on $\eta_k$,   
\refL{LMp99} applies to $Y_k(t)$ 
and shows that
\begin{align}
  \E \bigpar{\zeta_k^p\mid \eta_k} \le C \ceil{\eta_k}^p \le C(\eta_k+1)^p.
\end{align} 
Consequently, using \ref{Ap},
\begin{align}
  \E \zeta_k^p \le C \E(\eta_k+1)^p =\E(\xi_{ji}+1)^p<\infty.
\end{align}

We apply \refL{L12p0} to the sum $Z(t)$
in \eqref{pp2}, taking $q=p$ and $\mu=\gl_i$.
Note that then
\eqref{xp1} holds: if $i\notin\sQm$ then $\xi_{ii}\ge 0$ a.s.,
and thus $\mu=\gl_i\ge0$; if $i\in\sQm$ then  $\mu=\gl_i>0$ by our assumption
\eqref{ALMp}.
Hence, \eqref{pq2} holds.
Furthermore, $\kkx(\mu)=\kk_i$ and $(\glx_j-\mu)_+=(\glx_i-\gl_i)$,
as is easily verified by considering the three cases in \eqref{kkx} (and
after \eqref{jzz}) separately.
Hence,  \eqref{j4*}, \eqref{pp2}, and  \eqref{pq2} yield
\begin{align}
\tXX_i   
=\sup_{t\ge0}\bigset{(t+1)^{-\kk_i}e^{-(\glx_i-\gl_i)t}\cdot e^{-\gl_i t}X_i(t)}
\le  \tZZ
\in L^p
.\end{align}
\end{proof}

\subsection{The general case for a single colour}\label{SSLGp}
We now consider any colour $i\in \sQ$.

\begin{lemma}\label{LGp}
Assume \ref{A0}--\ref{A3}, either 
$\bigl($\ref{A-5} and \ref{A-7}$\bigr)$  
or \ref{A+},
and \ref{Ap} for some $p>1$.
Let $i\in\sQ$, and assume that for every $j\in\sP_i$,
we have
$\tXX_j\in L^p$.
Then
$\tXX_i\in L^p$.
\end{lemma}
\begin{proof}
Since \ref{A+} implies
\ref{A-5} and \ref{A-7}, these hold in any case.
We consider two cases. \resetCase

\pfCase{\eqref{ALMp} holds, \ie{} $i\notin\sQm$ or $\gl_i>0$}\label{LGpC1}
  As in the proof of \refLs{LG} and \ref{ZLG}, we split the colour $i$ into
  subcolours $i_0$ and $i_j$, $j\in\sP_i$.
We then use \refL{LMp}\ref{LMp1} for $i_0$ and \refL{L12p1} for every $i_j$,
and the result  follows by \eqref{g66}.

\pfCase{$i\in\sQm$ and $\gl_i\le0$}\label{LGpC2}
Then \ref{A-7} yields  $\glx_i>0$. 
Hence $\glx_i>\gl_i$, and 
since $\glx_i=\gl_i\lor \max_{j\in\sP_i}\glx_j$,
we have $\sP_i\neq\emptyset$ and
\begin{align}
  \label{feb3}
\glx_i= \max_{j\in\sP_i}\glx_j.
\end{align}

Consider a modification $\xxi_{ii}$ of $\xi_{ii}$
such that $\xxi\iii\ge\xi\iii$ \as, and
$\bar\gl_i:=a_i\E\xxi_{ii}\in (0,\glx_i)$.
(For example, let
$\xxi_{ii}:=\xi_{ii}\lor \txi$ where $\txi\in\set{\pm1}$ 
is independent of $\xi\iii$ and $\P(\txi=1)\in(0,1]$ is chosen suitably.)
Modify the urn by replacing $\xi\iii$ by $\xxi\iii$; this does not affect
any colour $j\prec i$; in particular $X_j(t)$ and all draws of colour $j$
remain the same for every $j\in\sP_i$, 
but at each draw of colour $i$ we may add more balls of
colour $i$; 
hence, letting $\xX_i(t)$ denote the number of balls of colour $i$ in the
modified urn, we have, using an  obvious coupling of the two urns, 
\begin{align}\label{feb1}
\xX_i(t)\ge X_i(t), \qquad t\ge0. 
\end{align}
The modified urn satisfies all our
conditions in the present lemma, and since $\bar\gl_i>0$, the already
proven \refCase{LGpC1} shows that (with obvious notation)
 $\xtXX_i\in L^p$.
Furthermore, $\bar\gl_i<\glx_i$, and thus \eqref{feb3} implies
\begin{align}
  \label{feb4}
\xglx_i:=\bar\gl_i\lor \max_{j\in\sP_i}\glx_j
=\bar\gl_i\lor \glx_i
=\glx_i. 
\end{align}
Similarly, using \eqref{kk}, we have $\bar\kk_i=\kk_i$.
Consequently, the exponents in \eqref{j4*} are the same for $X_i(t)$ and
$\xX_i(t)$, and thus \eqref{feb1} implies
\begin{align}
\tXX_i\le  \xtXX_i.
\end{align}
Since, as just shown, $\xtXX_i\in L^p$, this completes the proof.
\end{proof}

\subsection{$L^p$ bounds and convergence}

\begin{lemma}\label{Lp}
  Assume \ref{A0}--\ref{A3}, either 
$\bigl($\ref{A-5} and \ref{A-7}$\bigr)$  
or \ref{A+},
and \ref{Ap} for some $p>1$.
Then $\tXX_i\in L^p$ for every $i\in\sQ$.
\end{lemma}
\begin{proof}
 By \refL{LGp} and induction  on the colour $i$.
\end{proof}

\begin{proof}[Proof of \refT{TMCp}]
By \refL{Lp}, $\tXX_j\in L^p$; hence
it follows, exactly as for the case $p=2$ in \refT{TMC2}, 
that the collection $\bigset{|\tX_i(t)|^p:t\ge1}$ is uniformly integrable, and
thus the convergence \eqref{tc1} holds also in $L^p$.
\end{proof}

Once we have proved \refT{Tp}, the proof of \refT{TMCp} applies to any
$p>1$, as claimed in \refR{Rmomp}.

\subsection{A.s.\ convergence}

We now turn  \refT{Tp}, \ie, that our \as{} convergence results hold
also if \ref{A2} is replaced by \ref{Ap}.
We may assume $1<p<2$, since for $p\ge2$ the assumption
\ref{Ap} implies \ref{A2}, and the
results are already proven.

We begin by 
extending 
\refL{ZLM} to $1<p<2$, 
complementing \refLs{LMp} and \ref{LMp99} for the case 
excluded there when \eqref{ALMp} does not hold.

\begin{lemma}\label{ZLMp}
Assume \ref{A0}--\ref{A3}, \ref{A-5}, and \ref{Ap} for some $p\in(1,2]$.
Let $i\in\sQmin$ and assume $\gl_i\le0$.
Let $x_0:=X_i(0)$.
  \begin{romenumerate}[-10pt]
  \item\label{ZLM0p} 
If\/ $\gl_i=0$, then $X_i(t)$ is a martingale with
    \begin{align}
\label{zlm01p}
      X_i(t)&\asto \cX_i=0,
\qquad\text{as \ttoo},
\\\label{zlm0p}
\E X_i^*(t)^p&\le Cx_0^p+Cx_0 t,
\qquad\text{ for every $t<\infty$}
.    \end{align}
Furthermore, for every $\gd>0$,
\begin{align}\label{zlm04p}
\E \bigpar{\sup_{ t\ge0}\cpar{e^{-\gd t}X_i(t)}}^p&<\infty.
\end{align}

  \item\label{ZLMp-} 
If\/ $\gl_i<0$, then,
with $M(t):=e^{-\gl_it}X_i(t)$, 
    \begin{align}\label{zlm-1p}
      X_i(t)&\asto \cX_i=0,
\qquad\text{as \ttoo},
\\\label{zlm-2p}
\E{X_i(t)^p}&\le C x_0^p e^{\gl_i t} 
,
\\\label{zlm-3p}
\E M^*(t)^p&\le Cx_0^pe^{-(p-1)\gl_it}.
    \end{align}
  \end{romenumerate}
\end{lemma}

\begin{proof}
In both cases, $X_i(t)$ is a (sub)critical \ctime{} branching process and
therefore \as{} dies out, see \refR{RBP}, which gives \eqref{zlm01p} and
\eqref{zlm-1p}. 

Recall from \eqref{cb3} that  $M(t):=e^{-\gl_it}X_i(t)$ is a martingale
(this does not require \ref{A2}, only $\E\xi_{ii}<\infty$), and 
hence, or by \eqref{cb4},
\begin{align}\label{mb1}
  \E X_i(t)=e^{\gl_i t}\E M(t)
= e^{\gl_it}M(0)
= x_0e^{\gl_it}.
\end{align}
It follows from \eqref{lm4} that
\begin{align}\label{lm4p}
  [M,M]_t^{p/2}
\le x_0^p+\sumk\indic{T_k\le t}e^{-p\gl_iT_k}\eta_k^p,
\end{align}
and hence, using \refL{L9+}\ref{L9+b} and \eqref{mb1},
\begin{align}\label{lm5p}
\E  [M,M]_t^{p/2}
\le x_0^p+ C\intot e^{-p\gl_is}\E X_i(s) \dd s
=x_0^p+ C x_0\intot e^{-(p-1)\gl_is}\dd s.
\end{align}
If $\gl_i=0$, this yields \eqref{zlm0p} by \eqref{BDG};
then \eqref{zlm04p} follows as in \eqref{zlm03-4}.

If $\gl_i<0$, then \eqref{lm5p} and \eqref{BDG} yield,
recalling that $x_0$ is an integer,
\begin{align}\label{lm6p}
\E M^*(t)^p
\le Cx_0^p+ C x_0 e^{-(p-1)\gl_it}
\le Cx_0^p e^{-(p-1)\gl_it}
\end{align}
and thus
\begin{align}\label{lm7p}
\E X_i(t)^p&
=e^{p\gl_it}\E M^*(t)^p
\le Cx_0^p e^{\gl_it}
,\end{align}
showing \eqref{zlm-3p} and \eqref{zlm-2p}.
\end{proof}

\begin{proof}[Proof of \refT{Tp}]
As said above, we may assume $1<p<2$.
It suffices to prove the \ctime{} versions \refT{TC} and \ref{TC-};
then the proofs of \refTs{T1}, \ref{T1-}, and \ref{T1--} are as before.

By \refL{Lp}, we have $\tXX_i\in L^p$ for every $i\in\sQ$,
but it remains to show that $\tX_i(t)$ converges.
We follow the proof of \refT{TC} (and \refT{TC-}) step by step
in the claims below which extend the limit statements in
the lemmas in \refS{S1} and \refS{S-}
(recall that $L^p$ versions of the $L^2$ estimates there already are given); 
we omit some details.
Assumptions on $\norm{\tXX_j}_2$ are replaced by $\norm{\tXX_j}_p$
(and hold in our case by \refL{Lp}).
We assume in the sequel 
\ref{A0}--\ref{A3}, \ref{A-5} (or \ref{A+}), and \ref{Ap} for some $p\in(1,2]$.
(But not \ref{A-7} unless said so.) 

Note first that \eqref{cb4}--\eqref{cb3} still hold. (In fact, they require
only the first moment $\E\xi_{ii}<\infty$.)

\claim{In \refLs{LM} and \ref{ZLM},
the convergence $\tX_i:=e^{-\gl_it}X_i(t)\asto\cX_i$ still holds.
 We still have $\cX_i>0$ \as{} when $i\notin\sQm$,
and $\P(\cX_i>0)>0$ when
if $i\in\sQm$ and $\gl_i>0$.
}
\begin{proof}
The convergence $\tX_i(t)\asto\cX_i$ for some limit $\cX_i\in L^p$
is shown in \refL{LMp} or \ref{ZLMp}.
Furthermore, if $i\notin\sQm$, then the argument in the proof of
\refL{LM} shows that $\cX_i>0$ a.s.; otherwise, 
if $\gl_i>0$, then $\E\cX_i=X_i(0)>0$ and thus at least $\P(\cX_i>0)>0$.  
\end{proof}

\claim{In \refLs{L12} and \ref{ZL12},
the convergence
$\tX_i(t)\asto\cX_i$  still holds;
furthermore, $\cX_j>0\implies\cX_i>0$ a.s.
}
\begin{proof}
We argue as in the proof of \refLs{L12} and \ref{ZL12}, using again the
decompositions \eqref{jzz} and \eqref{jxz}. 
For $Z_4(t)$, the argument in \eqref{j43}--\eqref{j46=} holds without changes.
Also for $Z_3(t)$ the argument in the proof of \refL{L12} still holds,
since the $L^2$ estimate \eqref{j35} remains valid.

The remaining terms $Z_1(t)$ and $Z_2(t)$ are still local martingales.
For $Z_2(t)$ we have by the definition \eqref{j21}, \eqref{BDG}, 
\refL{L9+}\ref{L9+b},
and \eqref{j4*}
\begin{align}\label{maj21}
\E Z_2^*(t)^p&
\le C \E  [Z_2,Z_2]_t^{p/2}
\le C
\E \sumk \indic{T_k\le t}e^{-p\gl_i T_k}|\eta_k-r_{ji}|^p
\notag\\&
= C \E \intot e^{-p\gl_i s} X_j(s)\dd s
\notag\\&
\le C \E \tXX_j\intot (s+1)^{\kk_j}e^{(\glx_j-p\gl_i) s}\dd s
.\end{align}
We now argue as in \refStep{stepZ123b} of the proof of \refL{L12}, using
\eqref{maj21} instead of \eqref{j51} and $L^p$ instead of $L^2$;
we replace 2 by $p$ in all exponents, and
the cases \ref{step5-i'} and \ref{step5-ii'} are replaced by $\glx_j<p\gl_i$
and ($\glx_j\ge p\gl_i$ and $\glx_j>0$).  
We obtain as before that $\tZ_2(t)\asto\cZ_2$
as \ttoo, where the limit $\cZ_2=0$ except in case \ref{gli>}.

For $Z_1(t)$ we use \eqref{j5a} and both directions of \eqref{BDG} to obtain
\begin{align}\label{ma1}
\E |Z^*_1(t)|^p \le C \E [Z_1,Z_1]_t^{p/2} 
\le C\E\sumk [Z_1\kkk,Z_1\kkk]_t^{p/2}
\le C\E\sumk |Z_1\kkkx(t)|^{p}
.\end{align}
The definition \eqref{j11} yields
\begin{align}\label{ma2}
  \bigabs{Z_1\kkkx(t)}^p
= 
\indic{t\ge T_k}e^{-p\gl_i T_k} |\tY_k^*(t-T_k)|^p,
\end{align}
where
the martingale $\tY_k(t):=e^{-\gl_i t}Y_k(t)-Y_k(0)$ is 
independent of $T_k$,
and $Y_k(t)$ is a copy of the  one-colour
process in \refSS{SS11} and \refL{ZLMp}.
Thus \eqref{ma1} and \refL{L9+}\ref{L9+b}
yield
\begin{align}\label{ma3}
\E |Z^*_1(t)|^p &
\le C \E \sumk \indic{t\ge T_k}e^{-p\gl_i T_k} \tY_k^*(t)^p
\notag\\&
= C\E\bigsqpar{\tY_1^*(t)^p} \E \intot e^{-p\gl_i s} X_j(s)\dd s.
\end{align}
We separate three cases, as in the proof of \refL{ZL12}:
\begin{PXenumerate}
\item $i\notin\sQm$ or $\gl_i>0$:
This is condition \eqref{ALMp}; thus \refL{LMp99} applies 
to $Y_k(t)$
conditionally
on $\eta_k$, which gives
\begin{equation}\label{mc1}
  \E \bigsqpar{\tY_k^*(t)^p\mid\eta_k}\le C \ceil{\eta_k}^p
,\end{equation}
and consequently
\begin{equation}\label{mc2}
  \E \bigsqpar{\tY_1^*(t)^p}\le C
.\end{equation}
Thus \eqref{ma3} yields the same estimate \eqref{maj21} as for $Z_2^*(t)$,
and the argument after \eqref{maj21} applies to $Z_1$ too and yields
$\tZ_1(t)\asto\cZ_1$ for some $\cZ_1$, 
with $\cZ_1=0$ except in case \ref{gli>}.

\item $i\in\sQm$ and $\gl_i=0$:
(As in the proof of \refL{ZL12}, the assumptions then yield
$\glx_j=\glx_i>0$.)
In this case,
$\tY_k(t)=Y_k(t)-Y_k(0)$ where $Y_k(t)$ is a copy of $X_i(t)$ in
\refL{ZLMp}. Thus, \eqref{zlm0p} yields, 
recalling that now $\eta_k$ is an integer,
for $t\ge1$,
\begin{equation}\label{mc3}
  \E \bigsqpar{\tY_k^*(t)^p\mid\eta_k}\le C \eta_k^pt
,\end{equation}
and consequently
\begin{equation}
  \E \bigsqpar{\tY_1^*(t)^p}\le Ct
.\end{equation}
This and \eqref{ma3} yield, still for $t\ge1$,
\begin{align}\label{ma3b}
\E |Z^*_1(t)|^p &
\le C t \E \intot e^{-p\gl_i s} X_j(s)\dd s
\le C t \E\tXX_j\intot (s+1)^{\kk_j}e^{(\glx_j-p\gl_i)s}\dd s
\notag\\&
\le C t^{\kk_j+1}e^{(\glx_j-p\gl_i)t}.
\end{align}
This differs from \eqref{maj21} by an extra factor $t$, but
the argument after \eqref{maj21} 
(in this case modifications of \eqref{j53}--\eqref{j56}) 
still works and yields
$\tZ_1(t)\asto0$.

\item $\gl_i<0$:
(The assumptions yield $\glx_j=\glx_i>0$ in this case too.)
In this case,
we condition on $T_k$ and $\eta_k$ and then use \eqref{ma2}
and \eqref{zlm-3p}, yielding
\begin{align}\label{ma2a}
\E\bigsqpar{ \abs{Z_1\kkkx(t)}^p\mid T_k,\eta_k}&
= 
\indic{t\ge T_k}e^{-p\gl_i T_k} \E\bigsqpar{|\tY_k^*(t-T_k)|^p\mid T_k,\eta_k}
\notag\\&
\le 
C\indic{t\ge T_k}e^{-p\gl_i T_k} \eta_k^p e^{-(p-1)\gl_i(t-T_k)}
\notag\\&
=
C\indic{t\ge T_k}e^{-\gl_i T_k}  e^{-(p-1)\gl_it}\eta_k^p
.\end{align}
Hence, \eqref{ma1}, \refL{L9+}\ref{L9+b}, and \eqref{j4*} yield
\begin{align}\label{mb3}
\E |Z^*_1(t)|^p &
\le C \E \sumk \E\bigsqpar{ \abs{Z_1\kkkx(t)}^p\mid T_k,\eta_k}
\notag\\&
\le C e^{-(p-1)\gl_it}\E\sumk \indic{t\ge T_k}e^{-\gl_i T_k}\eta_k^p
\notag\\&
=
C e^{-(p-1)\gl_it}\E\intot e^{-\gl_i s}X_j(s)\dd s
\notag\\&
\le
C e^{-(p-1)\gl_it}\E\tXX_j\intot (s+1)^{\kk_j}e^{(\glx_j-\gl_i) s}\dd s
\notag\\&
\le C (t+1)^{\kk_j}e^{(\glx_j-p\gl_i)t}
.\end{align}
With $\tZcl(n)$ defined in \eqref{j53},
we again modify \eqref{j54} using $p$th powers, now using \eqref{mb3} instead
of \eqref{j51}, and it follows similarly to \eqref{j55} that 
\begin{align}
\E \sumn \tZcl(n)^p<\infty.  
\end{align}
Hence, $\tZ_1\asto0$. 
\end{PXenumerate}

Finally, in all cases,
it follows from \eqref{jxz} that $\tX_i(t)\asto\cX_i$.
Furthermore, we have
$\cX_j>0\implies \cX_i>0$ \as{} by
the argument in \refStep{stepVI} of \refL{L12}, if necessary
modified as in the proof of \refL{ZL12}.
\end{proof}

\claim{In  \refL{ZL1200},
$e^{-\gd t}X_i(t)\asto0$  still holds.
}
\begin{proof}
  The same as for \refL{ZL1200}.
\end{proof}

\claim{In \refLs{LG} and \ref{ZLG},
$\tX_i\asto\cX_i$  still holds, and so do the claims on $\cX_i>0$.
}\label{cl4}
\begin{proof}
  The same as for \refLs{LG} and \ref{ZLG}, using the claims above.
Note that the assumptions now include \ref{A-7}.
\end{proof}

To complete the proof of \refT{Tp},
we now obtain by induction from \ref{cl4} above that
for every $i\in\sQ$ we have
$\tX_i\to\cX_i$ as \ttoo, \ie, \eqref{tc1}.
This proves that \refTs{TC} and \ref{TC-} hold with \ref{A2} replaced by
\ref{Ap}, which as said above completes the proof.
\end{proof}

\section{Proof of \eqref{jb1}}\label{Ajb}
In this appendix we prove \eqref{jb1} in \refE{E+-}, using results from
Markov process theory.

Suppose, more generally, that the \Polya{} urn in \refE{E+-}
starts with $w_0=\ga>0$ white balls
and $b_0\ge0$ black balls.
Thus $W(t)=\ga$ for all $t\ge0$. 
(We do not have to assume that $\ga$ is an integer, although we must have
$b_0\in\bbZgeo$, since we allow subtractions.)
The stochastic process $B(t)$ is a time-homogeneous pure-jump Markov process
on $\bbZgeo$ with jumps
\begin{align}
  \label{ja1}
  \begin{cases}
    +1& \text{with intensity}\quad \ga+\frac12B(t),
\\
    -1& \text{with intensity}\quad \frac12B(t).
  \end{cases}
\end{align}
We define for any real $\ell>0$ the scaled process
\begin{align}\label{ja2}
  \tB_\ell(t):=\ell\qw B(\ell t).
\end{align}
 It follows from \eqref{ja1} that $\tB_\ell$ 
is a pure-jump Markov process with jumps
\begin{align}
  \label{ja3}
  \begin{cases}
    +1/\ell& \text{with intensity}\quad 
\ell\bigpar{\ga+\frac12B(\ell t)}
= \ga\ell+\frac{\ell^2}2\tB_\ell(t),
\\
    -1/\ell& \text{with intensity}\quad 
\frac{\ell^2}2\tB_\ell(t).
  \end{cases}
\end{align}
In other words, the generator $\cA_\ell$ of 
the Markov process
$\tB_\ell$ is given by,
see \eg{} \cite[Theorem 19.23]{Kallenberg},
\begin{align}
  \label{ja4}
\cA_\ell f(x) =
\ga\ell\bigpar{f(x+\ell\qw)-f(x)}
+ x\frac{\ell^2}{2}
\bigpar{f(x+\ell\qw)+f(x-\ell\qw)-2f(x)}.
\end{align}
Since $B(t)$ takes its values in $\bbZgeo$, 
$\tB_\ell(t)$ is a Markov process on the state space
$\ell\qw\bbZgeo=\set{0,\ell\qw,2\ell\qw,\dots}$.
For technical reasons, we extend it to a pure-jump Markov process on $\ooo$
by defining the 
intensities to be as in \eqref{ja3} whenever $\tB_\ell(t)\ge\ell\qw$;
otherwise, if $\tB_\ell(t)=b<1/\ell$,
we jump $+1/\ell$ with intensity $\ga\ell$ and
$\pm b$ with intensities $\frac{\ell^2}2 b$ each.
(This makes no difference if we start with an integer number of black balls.)
The generator of the extended process is 
\begin{align}
  \label{ja5}
\cA_\ell f(x) =
\ga\ell\bigpar{f(x+\ell\qw)-f(x)}
+ x\frac{\ell^2}{2}
\bigpar{f(x+h_{\ell,x})+f(x-h_{\ell,x})-2f(x)},
\end{align}
where $h_{\ell,x}:=\ell\qw\bmin x$.

Let $C^2\ooo$ be the space of all
continuous functions $f$  on $\ooo$ that have two
continuous derivatives in $(0,\infty)$ with $f'$ and
$f''$ extending continuously to $\ooo$.
Let further
\begin{align}\label{ja6}
\cC:=\bigset{f\in C^2\ooo: f(x),f'(x),f''(x)=O(e^{-\eps x})
\text{ for some $\eps>0$}}.  
\end{align}
It follows from \eqref{ja5} and Taylor's formula that if $f\in\cC$, then
as $\ell\to\infty$,
\begin{align}\label{ja7}
  \cA_\ell f(x) \to \cA f(x):= \ga f'(x) + \frac{x}2 f''(x)
\end{align}
uniformly in $x\in\ooo$, and thus in the space $C_0\ooo$.
Note that $4\cA= 2x\frac{\ddx^2}{\ddx x^2} + 4\ga \frac{\ddx}{\ddx x}$
is the generator of the squared Bessel process $\BESQ^\gd$ with dimension
$\gd:=4\ga$, see \cite[p.~443 and Proposition VII.(1.7)]{RY}; 
hence $\cA$ is the generator of 
$\BESQ^{4\ga}(t/4)\eqd \frac{1}{4}\BESQ^{4\ga}(t)$
\cite[Proposition XI.(1.6)]{RY}; 
It is well known that $\BESQ^{4\ga}(t)$ is a Feller process on $\ooo$
\cite[p.~442]{RY}, and 
it is easily verified directly
that each $\cA_\ell$ also is a Feller process on $\ooo$.
The transition probabilities $q_t^\gd(x,y)$ of $\BESQ^\gd$
are given explicitly in \cite[XI.(1.4)]{RY}, and a simple calculation
using \cite[(10.29.4)]{NIST} shows that
\begin{align}\label{ja8'}
  \frac{\partial}{\partial x} q_t^\gd(x,y)&
= \frac{1}{2t}\bigpar{q_t^{\gd+2}(x,y)-q_t^\gd(x,y)},
\intertext{and hence}
\label{ja8''}
  \frac{\partial^2}{\partial x^2} q_t^\gd(x,y)&
= \frac{1}{4t^2}\bigpar{q_t^{\gd+4}(x,y)-2q_t^{\gd+2}(x,y)+q_t^\gd(x,y)}.
\end{align}
Since $\BESQ^\gd$ is a Feller process for every $\gd>0$, 
the transition operator $T^\gd_tf(x):=\int q_t^\gd(x,y)f(y)\dd y$ maps
$C_0\ooo$ into itself for every $t>0$, 
and it follows from \eqref{ja8'}--\eqref{ja8''}
that $T^\gd_t$ maps $C_0\ooo$ into $C^2\ooo$; moreover, using also again
the explicit form of $q^\gd_t$ in \cite[XI.(1.4)]{RY}, 
it is easy to see that $T_t$ maps $\cC$ into itself.
Hence, it follows from \cite[Proposition 19.9]{Kallenberg} that $\cC$ is a
core for the generator $4\cA$, and thus also for $\cA$.
Consequently, \eqref{ja7} and \cite[Theorem 19.25]{Kallenberg} show the
following.
\begin{theorem}\label{Tjb}
Let the \Polya{} urn in \refE{E+-}
start with $w_0=\ga>0$ white balls
and $b_0=0$ black balls.
Then, as $\ell\to\infty$, we have 
  \begin{align}\label{tjb}
\ell\qw B(\ell t)
=   \tB_\ell(t)
\dto  \frac{1}4\BESQ^{4\ga}(t)
\qquad \text{in $D\ooo$}.
  \end{align}
\end{theorem}
The squared Bessel process $\BESQ^{4\ga}$
in \eqref{tjb} is the standard one starting at 0.
More generally, if we start the urn with $\ga$ white balls and 
$\gb \ell +o(\ell)$ black balls, then 
the same proof shows that
\eqref{tjb} holds with the initial value
$\frac{1}4\BESQ^{4\ga}(0)=\gb$.

\newcommand\AAP{\emph{Adv. Appl. Probab.} }
\newcommand\JAP{\emph{J. Appl. Probab.} }
\newcommand\JAMS{\emph{J. \AMS} }
\newcommand\MAMS{\emph{Memoirs \AMS} }
\newcommand\PAMS{\emph{Proc. \AMS} }
\newcommand\TAMS{\emph{Trans. \AMS} }
\newcommand\AnnMS{\emph{Ann. Math. Statist.} }
\newcommand\AnnPr{\emph{Ann. Probab.} }
\newcommand\CPC{\emph{Combin. Probab. Comput.} }
\newcommand\JMAA{\emph{J. Math. Anal. Appl.} }
\newcommand\RSA{\emph{Random Structures Algorithms} }
\newcommand\DMTCS{\jour{Discr. Math. Theor. Comput. Sci.} }

\newcommand\AMS{Amer. Math. Soc.}
\newcommand\Springer{Springer-Verlag}
\newcommand\Wiley{Wiley}

\newcommand\vol{\textbf}
\newcommand\jour{\emph}
\newcommand\book{\emph}
\newcommand\inbook{\emph}
\def\no#1#2,{\unskip#2, no. #1,} 
\newcommand\toappear{\unskip, to appear}

\newcommand\arxiv[1]{\texttt{arXiv}:#1}
\newcommand\arXiv{\arxiv}

\newcommand\xand{and }
\renewcommand\xand{\& }

\def\nobibitem#1\par{}

\end{document}